\documentclass[11pt, reqno]{amsart}

\usepackage[utf8]{inputenc}
\usepackage[T1]{fontenc}
\usepackage[sc]{mathpazo}          
\usepackage[final,protrusion=true,expansion=true]{microtype}
\usepackage{amsmath}
\usepackage{amssymb}
\usepackage{amsfonts}
\usepackage{amsbsy}
\usepackage{mathrsfs}
\usepackage{mathtools}
\usepackage{exscale}
\usepackage{relsize}
\usepackage[dvipsnames]{xcolor}
\usepackage[margin=1in]{geometry}
\usepackage{setspace}
\usepackage[inline]{enumitem}
\usepackage{comment}
\usepackage{url}
\usepackage{cite}
\usepackage{booktabs}
\usepackage{array}
\usepackage[bottom,hang]{footmisc}
\usepackage{tikz}
\usepackage{pgfplots}
\pgfplotsset{compat=1.16}
\usetikzlibrary{arrows.meta,patterns,calc}
\usepgfplotslibrary{fillbetween}

\definecolor{venetianred}{rgb}{0.78, 0.03, 0.08}
\definecolor{lava}{rgb}{0.62, 0.05, 0.10}      
\definecolor{cobalt}{rgb}{0.0, 0.28, 0.67}     
\definecolor{attn}{RGB}{204,0,153}             
\definecolor{revcol}{RGB}{0,110,90}            

\usepackage[colorlinks=true,linkcolor=lava,citecolor=blue,urlcolor=cobalt,
  pdfauthor={D. G. Bhimani, M. Majdoub, R. Manna},
  pdftitle={Confinement effects for semilinear fractional Hermite heat
  equations with nonlinearities of arbitrarily rapid growth}]{hyperref}

\newif\ifrevmarks
\revmarkstrue

\numberwithin{equation}{section}

\newtheoremstyle{thmstyle}
  {0.6\baselineskip}  
  {0.6\baselineskip}  
  {\itshape}          
  {}                  
  {\bfseries}         
  {.}                 
  {0.5em}             
  {}                  
\newtheoremstyle{defstyle}
  {0.6\baselineskip}
  {0.6\baselineskip}
  {\normalfont}
  {}
  {\bfseries}
  {.}
  {0.5em}
  {}

\theoremstyle{thmstyle}
\newtheorem{theorem}{Theorem}[section]
\newtheorem{proposition}[theorem]{Proposition}
\newtheorem{Lemma}[theorem]{Lemma}
\newtheorem{Corollary}[theorem]{Corollary}
\newtheorem{mainthm}{Theorem}

\theoremstyle{defstyle}
\newtheorem{Definition}[theorem]{Definition}
\newtheorem{Remark}[theorem]{Remark}
\newtheorem{Example}[theorem]{Example}

\setlist[enumerate]{itemsep=2pt,topsep=4pt}

\newcommand{\R}{\mathbb R}

\newcommand{\N}{\mathbb N}
\newcommand{\Hr}{\mathcal{H}_{\varrho}}
\newcommand{\Kb}{\mathcal{K}}
\newcommand{\Y}{\mathcal{Y}}
\newcommand{\Nf}{\mathcal{N}}
\newcommand{\dd}{\,\mathrm{d}}
\newcommand{\E}{\mathbb{E}}
\DeclareMathOperator{\supp}{supp}
\DeclareMathOperator*{\esssup}{ess\,sup}

\makeatletter
\def\section{\@startsection{section}{1}{\z@}%
  {-2.5ex\@plus -1ex \@minus -.2ex}%
  {1.5ex \@plus .2ex}%
  {\normalfont\large\bfseries}}
\def\subsection{\@startsection{subsection}{2}{\z@}%
  {-1.8ex\@plus -.5ex \@minus -.2ex}%
  {0.8ex \@plus .2ex}%
  {\normalfont\bfseries}}
\makeatother

\begin{document}

\title[Confinement and arbitrarily rapid growth]{Confinement effects for semilinear fractional
Hermite heat equations with nonlinearities of arbitrarily rapid growth}

\author[D. G. Bhimani]{Divyang G. Bhimani}
\address[D. G. Bhimani]{Department of Mathematics, Indian Institute of Science
Education and Research, Dr. Homi Bhabha Road, Pune 411008, India}
\email{\tt divyang.bhimani@iiserpune.ac.in}

\author[M. Majdoub]{Mohamed Majdoub}
\address[M. Majdoub]{Department of Mathematics, College of Science, Imam
Abdulrahman Bin Faisal University, P. O. Box 1982, Dammam, Saudi Arabia \&
Basic and Applied Scientific Research Center, Imam Abdulrahman Bin Faisal
University, P.O. Box 1982, 31441, Dammam, Saudi Arabia.}
\email{\tt mmajdoub@iau.edu.sa}
\email{\tt med.majdoub@gmail.com}
\email{\tt mohamed.majdoub@fst.rnu.tn}
\author[R. Manna]{Ramesh Manna}
\address[R. Manna]{School of Mathematical Sciences, National Institute of Science Education and
	Research, Bhubaneswar, India.}
\address[R. Manna]{Homi Bhabha National Institute, Training School Complex, Anushakti Nagar, Mumbai 400094, India}

\email{\tt rameshmanna@niser.ac.in}

\begin{abstract}
We study the Cauchy problem for the semilinear heat equation driven by the
fractional harmonic oscillator
$$
\partial_t u + \bigl(-\Delta + \varrho|x|^2\bigr)^{\beta} u = f(u),
\qquad (x,t)\in \R^d\times(0,\infty), \qquad u(\cdot,0)=u_0,
$$
where $\varrho>0$ and $0<\beta\le 1$. No \emph{a priori} growth restriction
is imposed on the nonlinearity $f$ at infinity. We develop a local and global
well-posedness theory in Orlicz spaces $L_\Phi(\R^d)$ and in their
Morse--Transue subspaces $M_\Phi(\R^d)$. The theory rests on two ingredients.
The first is a smoothing estimate for the semigroup generated by
$\bigl(-\Delta+\varrho|x|^2\bigr)^{\beta}$ between two arbitrary Orlicz spaces;
its constant does not depend on the Orlicz spaces, and it captures the sharp
exponential rate $\lambda_\beta=(d\sqrt{\varrho}\,)^{\beta}$ for large times.
The second is an integral compatibility condition relating the Young function
$\Phi$ to the Lipschitz modulus of $f$. The semigroup estimates follow from a
pointwise domination of the subordinated Mehler kernel by the free fractional
heat kernel, combined with an elementary bound for the Orlicz norm of a kernel
in terms of its $L^1$ and $L^\infty$ norms. For nonlinearities growing at least
sub-exponentially, we show that $\Phi_c=f$ is a critical Young function: it
separates local well-posedness in $M_\Phi$ from the instantaneous nonexistence
of nonnegative solutions for suitable data in $L_\Phi$. The nonexistence part
is proved directly, with an explicit threshold, for the singular profile
$F^{-1}(\gamma|x|^{2\beta})$, where $F(u)=\int_u^\infty \dd s/f(s)$. Finally,
we exhibit a genuine confinement effect. Under a mere Dini condition on $f$ at
the origin, small data give global solutions decaying at the sharp rate
$e^{-\lambda_\beta t}$, with no Fujita-type restriction. {This is in sharp
contrast with the free equation $\partial_tu+(-\Delta)^{\beta}u=f(u)$
($\varrho=0$): for $f(u)=|u|^{m-1}u$ with $1<m\le1+2\beta/d$, every nonnegative
nontrivial solution of the free problem blows up in finite time, however small
the initial data.}
\end{abstract}

\maketitle

\medskip


\section{Introduction}
\label{sec:intro}

We consider the Cauchy problem
\begin{equation}\label{eq:main}
\begin{cases}
\partial_t u + \Hr^{\beta} u = f(u), & (x,t)\in \R^d\times(0,\infty),\\[2pt]
u(x,0)=u_0(x), & x\in \R^d,
\end{cases}
\end{equation}
where $d\ge 1$, $\varrho>0$, $0<\beta\le 1$, and
$$
\Hr:=-\Delta+\varrho|x|^2
$$
is the harmonic oscillator on $\R^d$, {also known as the quantum harmonic
oscillator or the Hermite operator.} The nonlinearity $f:\R\to\R$ is locally
Lipschitz continuous with $f(0)=0$. The initial data $u_0$ is a possibly
unbounded element of some Banach space $X$ of real-valued functions on $\R^d$.
The operator $\Hr$ is nonnegative and self-adjoint on $L^2(\R^d)$, and the
power $\Hr^{\beta}$ is defined by the spectral theorem. We denote by
$S_\beta(t):=e^{-t\Hr^{\beta}}$ the semigroup generated by $-\Hr^{\beta}$; it
is extended to the function spaces used below through its integral kernel
(Section~\ref{subsec:operator}).

The Hermite operator is one of the basic models of mathematical physics. It
governs the quantum harmonic oscillator, and its eigenfunctions, the Hermite
functions, form the orthonormal basis underlying the harmonic analysis of
Hermite expansions \cite{thangavelu1993lectures}. Its spectral theory is
remarkably rich; for instance, sharp $L^p$ bounds for its eigenfunctions were
obtained by Koch and Tataru \cite{Koch2005}. The
fractional powers $\Hr^{\beta}$ are nonlocal operators of order $2\beta$ which
combine the long-range interactions of the fractional Laplacian with the
confinement produced by the potential. They have been studied from several
points of view, including extension problems and Harnack inequalities
\cite{StingaTorrea2010} and fixed-time estimates for the associated heat
semigroup \cite{BMMfrac}; see also \cite{thangavelu2024}.

We address the classical triple of questions, namely existence, uniqueness
and continuous dependence. The initial data are taken in a space large enough
to contain genuinely singular functions, and the growth of the nonlinearity at
infinity is not restricted in advance. We are especially interested in
nonlinearities that grow faster than any exponential. A prototypical example
is
$$
f(u)=|u|^{m-1}u\,\exp\bigl(e^{|u|^{p}}\bigr),\qquad m\ge 1,\ p\ge 1 .
$$
For the free heat equation ($\beta=1$, $\varrho=0$), a well-posedness theory
for such nonlinearities in a linear Banach space other than $L^\infty(\R^d)$
has only recently become available, through the work of Fujishima, Hisa and
Laister \cite{Laister}. No such theory exists for the fractional harmonic
oscillator. The purpose of the present paper is to provide one, and to
identify the effects produced by the confining potential.

\subsection{Background and related work}\label{subsec:background}

\emph{Lebesgue space theory.} When $\varrho=0$ and $\beta=1$, problem
\eqref{eq:main} reduces to the classical semilinear heat equation
$u_t=\Delta u+f(u)$. Its study in Lebesgue spaces for polynomially dominated
nonlinearities goes back to Weissler \cite{Weissler1979,weissler1980local,
weissler1981}. He showed that, for $|f(u)|\lesssim |u|^p$, the problem is
locally well posed in $L^q(\R^n)$ for $q\ge q_c:=n(p-1)/2$, $q>1$, and that
nonnegative solutions may fail to exist for $q<q_c$. Uniqueness was first
obtained in a restricted class, a restriction later removed by Brezis and
Cazenave \cite{brezis1996nonlinear}; see also Giga \cite{Giga1986}, and Haraux
and Weissler \cite{haraux1982non} for non-uniqueness.{\footnote{The
exponent $q_c$ is the \emph{scaling-critical} Lebesgue exponent. For
$f(u)=|u|^{p-1}u$ the equation $u_t=\Delta u+f(u)$ is invariant under the
dilations $u_\mu(x,t):=\mu^{2/(p-1)}u(\mu x,\mu^2t)$, $\mu>0$, and
$\|u_\mu(\cdot,0)\|_{L^q}=\mu^{\frac{2}{p-1}-\frac nq}\|u(\cdot,0)\|_{L^q}$. The
$L^q$ norm of the initial data is therefore invariant precisely when $q=q_c$.}} A
complete characterisation of local existence in Lebesgue spaces was obtained
by Laister, Robinson, Sierż\k{e}ga and Vidal-L\'opez \cite{AIHP2016}, and
well-posedness in $L^1$ by Laister and Sierż\k{e}ga
\cite{LaisterSierzega2020}; see also
\cite{RobinsonSierzega2013,LaisterSierzega2021}. The large body of work on the
global behaviour of solutions was initiated by Fujita \cite{Fujita1966}. He
identified the critical exponent $1+2/n$, which separates global existence for
small data from blow-up of all nonnegative nontrivial solutions. The critical
case was settled by Hayakawa \cite{Hayakawa1973} for $n=1,2$, and by
Kobayashi, Sirao and Tanaka \cite{KST1977} in general. For the fractional
Laplacian $(-\Delta)^{\beta}$, the corresponding exponent is $1+2\beta/n$, as
shown by Sugitani \cite{Sugitani1975}. We refer to the monographs
\cite{QS,Hu,Pao,Henry,cazenave1990introduction} for comprehensive accounts.

\emph{Exponential nonlinearities and Orlicz spaces.} Once $f$ grows faster
than any power, Lebesgue spaces are no longer adapted to the problem, and
Orlicz spaces take their place. Exponential nonlinearities arise naturally in
the study of dispersive equations at the Trudinger--Moser threshold, where the
exponential Orlicz scale appears through the critical Sobolev embedding; see
Nakamura and Ozawa \cite{NakamuraOzawa1998}. For a series of results on
Schrödinger, Klein--Gordon and wave equations with exponential nonlinearities,
we refer to \cite{CPAM2006,AIHP2011,CIMM,Duke2009,APDE2011,IMMN}. Physical
motivation for exponential-type nonlinearities in nonlinear optics can be
found, for instance, in \cite{Lam77}. For the heat equation, initial data in
$\exp L^2(\R^2)$ were first considered by Ruf and Terraneo
\cite{RufTerraneo2002}. Ioku \cite{Ioku2011} proved global existence for small
data in $\exp L^2(\R^n)$. Ioku, Ruf and Terraneo \cite{ioku2015MPAG} then
obtained local existence and uniqueness in the closure $\exp L^2_0(\R^2)$ of
$C_0^\infty$, together with the nonexistence of nonnegative solutions for
suitable data in $\exp L^2(\R^2)$. Non-uniqueness for singular data was
studied in \cite{IRT2019,FIRT2026}. {Furioli, Kawakami, Ruf and Terraneo
\cite{furioli2017JDE} considered the nonlinearity $f(u)=|u|^{4/n}u\,e^{u^2}$ in
$\R^n$. For small data in $\exp L^2(\R^n)$ they proved the existence of global
solutions together with $L^q$ decay estimates, and they showed that,
for data in $\exp L^2\cap L^1$, the solution behaves for large times like a
multiple of the Gaussian heat kernel.} The case
$f(u)\approx|u|^{m-1}u\,e^{|u|^{p}}$, $p>1$, in arbitrary dimension was
studied by Majdoub and Tayachi \cite{majdoub2018picm,Mohamed2021}. They
established local well-posedness in $\exp L^p_0(\R^n)$, global existence for
small data and decay estimates. The biharmonic heat equation with exponential
nonlinearity was treated by Majdoub, Otsmane and Tayachi
\cite{majdoub2018ade}, and the fractional Laplacian by Fino and Kirane
\cite{fino2020cpaa}. Finally, Bhimani, Majdoub and Manna \cite{BMMexp}
considered the fractional harmonic oscillator $\Hr^{\beta}$, $0<\beta\le1$,
with nonlinearities of exponential type $e^{\lambda|u|^p}$, $p>1$. They
obtained local well-posedness in $\exp L^p_0$, global weak-mild solutions for
small data and, for $\beta=1$, nonexistence of nonnegative classical solutions
for suitable data in $\exp L^p$.

As pointed out in \cite[Section 1.2]{Laister}, all of these works share two
features. An exponential upper bound is imposed on the growth of $f$, and this
bound is used in an essential way: the smoothing estimates for the heat
semigroup between exponential Orlicz spaces and Lebesgue spaces are obtained
by expanding the exponential in a power series. This barrier has recently been
removed by Fujishima, Hisa and Laister \cite{Laister}. They established a
smoothing estimate for the Gaussian heat semigroup between two
\emph{arbitrary} Orlicz spaces, deduced from it a well-posedness theory for
nonlinearities of arbitrarily rapid growth, and identified $\Phi_c=f$ as a
critical Young function. Their sufficient condition for local well-posedness is
the convergence of an integral involving $f$ and $\Phi$ jointly. It thus acts
as a compatibility condition between the nonlinearity and the underlying space,
rather than as a growth restriction on either of them. Earlier, Fujishima and
Ioku \cite{Ioku2018} had developed a quasi-scaling method which yields
existence and nonexistence thresholds for nonlinearities of unrestricted
growth, for initial data in sets that are not linear spaces; see also \cite{FujishimaIoku2021,FujishimaIoku2022,
FujishimaIoku2023,Miyamoto2021}.

\emph{Fractional operators and singular data.} For the fractional heat
equation $\partial_tu+(-\Delta)^{\theta/2}u=f(u)$, local solvability for
singular initial data was studied by Hisa and Ishige \cite{HisaIshige2018} and
by Fujishima, Hisa, Ishige and Laister \cite{FHIL2023,Fujishima2024}. In
particular, it is shown in \cite{Fujishima2024} that, for supercritical $f$ (a
class containing all nonlinearities considered in
Section~\ref{sec:nonexistence}), the profile $F^{-1}(\gamma|x|^{\theta})$, with
$F(u)=\int_u^\infty\dd s/f(s)$, is a \emph{dilation-critical singularity}. In
other words, suitable dilations of it are, or are not, admissible initial data
according to the size of $\gamma$. See also
Miyamoto and Suzuki \cite{MiyamotoSuzuki2024} for critical and doubly critical
cases.

\emph{The harmonic oscillator.} The Hermite semigroup $e^{-t\Hr}$ is given
explicitly by Mehler's formula \cite{thangavelu1993lectures}. Fractional powers
of $\Hr$ have been studied, for example, by Stinga and Torrea
\cite{StingaTorrea2010}. Bhimani, Manna, Nicola, Thangavelu and Trapasso
\cite{bhimani2021Adv} developed phase space methods for the Hermite semigroup
and applied them to global well-posedness in modulation spaces. The same
authors obtained sharp fixed-time $L^p$--$L^q$ estimates for
$e^{-t\Hr^{\beta}}$, including the exponential decay for large times, and
applied them to nonlinear fractional heat equations \cite{BMMfrac}. See also Cordero \cite{cordero2021PsDO} for modulation space
estimates and Cardona, Chatzakou, Delgado, Kumar and Ruzhansky
\cite{CCDKR2025} for anharmonic oscillators $(-\Delta)^{\ell}+|x|^{2k}$. None of
these works allows growth faster than exponential.

Table~\ref{tab:literature} summarises the works most closely related to the
present paper.

\begin{table}[ht]
\centering
\small
\renewcommand{\arraystretch}{1.25}
\begin{tabular}{>{\raggedright}p{2.6cm} >{\raggedright}p{2.4cm}
>{\raggedright}p{3.0cm} >{\raggedright}p{2.2cm} >{\raggedright\arraybackslash}p{4.2cm}}
\toprule
Reference & Operator & Growth of $f$ & Space & Main results\\
\midrule
\cite{weissler1980local,brezis1996nonlinear} & $-\Delta$ & $|u|^{p}$ &
$L^q$ & LWP for $q\ge q_c$, $q>1$; nonexistence for $q<q_c$\\
\cite{AIHP2016,LaisterSierzega2020} & $-\Delta$ & polynomially dominated &
$L^q$, $L^1$ & characterisation of local existence; WP in $L^1$\\
\cite{RufTerraneo2002,Ioku2011,ioku2015MPAG} & $-\Delta$ & $e^{u^2}$ type &
$\exp L^2$ & LWP in $\exp L^2_0(\R^2)$, nonexistence; small-data GWP\\
\cite{furioli2017JDE} & $-\Delta$ & {$|u|^{4/n}u\,e^{u^2}$} &
$\exp L^2$ & {small-data GWP; $L^q$ decay; Gaussian asymptotics for data
in $\exp L^2\cap L^1$}\\
\cite{majdoub2018picm,Mohamed2021} & $-\Delta$ & $|u|^{m-1}u\,e^{|u|^p}$,
$p>1$ & $\exp L^p$ & LWP in $\exp L^p_0$; small-data GWP; decay\\
\cite{majdoub2018ade} & $\Delta^2$ & $e^{|u|^p}$ type & $\exp L^p$ & LWP;
small-data GWP\\
\cite{fino2020cpaa} & {$(-\Delta)^{\beta}$, $0<\beta\le1$} &
{$e^{|u|^p}$ type, $p>1$} & $\exp L^p$ & {LWP in $\exp L^p_0$;
small-data global weak-mild solutions; decay}\\
\cite{BMMexp} & $\Hr^{\beta}$, $0<\beta\le1$ & $e^{\lambda|u|^p}$, $p>1$ &
$\exp L^p$ & LWP in $\exp L^p_0$; small-data global weak-mild solutions;
nonexistence ($\beta=1$)\\
\cite{Ioku2018} & $-\Delta$ & unrestricted & sets of data & existence and
nonexistence thresholds\\
\cite{Fujishima2024} & $(-\Delta)^{\theta/2}$, {$0<\theta\le2$} & supercritical & singular data &
local solvability; dilation-critical singularity\\
\cite{Laister} & $-\Delta$ & unrestricted & $M_\Phi$, $L_\Phi$ & LWP,
small-data GWP, critical Young function\\
Present paper & $\Hr^{\beta}$, $0<\beta\le1$ & unrestricted & $M_\Phi$,
$L_\Phi$ & LWP, small-data GWP with exponential decay, critical Young
function, explicit nonexistence\\
\bottomrule
\end{tabular}
\caption{Works most closely related to the present paper (LWP: local
well-posedness; GWP: global well-posedness).}
\label{tab:literature}
\end{table}

\subsection{The main difficulties}

Two features distinguish \eqref{eq:main} from the free problem treated in
\cite{Laister}. On the one hand, the short-time regularisation of
$S_\beta(t)$ is governed by the fractional parabolic scaling
$t^{-d/(2\beta)}$. Hence every occurrence of the exponent $n/2$ in the
Gaussian theory is replaced by $d/(2\beta)$, and the exponent $1+2/n$ in the
integral conditions of \cite{Laister} is replaced by
$$
p_\ast:=1+\frac{2\beta}{d},
$$
the \textbf{Fujita exponent} of the \emph{free} fractional heat equation
\cite{Sugitani1975}. On the other hand, the confining potential
$\varrho|x|^2$ creates a spectral gap. The bottom of the spectrum of $\Hr$ is
$d\sqrt{\varrho}$, so that the bottom of the spectrum of $\Hr^\beta$ is
$\lambda_\beta:=(d\sqrt{\varrho}\,)^{\beta}$, and the linear evolution decays
exponentially rather than polynomially. As a consequence, and in contrast with
the free case, $p_\ast$ plays no role in the global existence of small
solutions (Theorem~\ref{thm:confined}).

The main technical difficulty is the absence of translation invariance. In
\cite{Laister}, the heat semigroup acts by convolution with the Gaussian
kernel. O'Neil's convolution theorem in Orlicz spaces \cite{ONeil} then
applies directly, and the smoothing estimate reduces to the computation of the
Orlicz norm of the Gaussian. This computation is carried out in
\cite[Proposition 4.1]{Laister} through radial coordinates and an induction on
the dimension. Neither step applies verbatim here. We resolve the first point
by a pointwise domination of the subordinated Mehler kernel by the kernel of
the free fractional heat semigroup, which restores the convolution structure
at the level of upper bounds. We resolve the second by an elementary lemma
(Lemma~\ref{lem:orlicznorm}), valid for every Young function and every
nonnegative kernel controlled in $L^1$ and in $L^\infty$. As a by-product, we
obtain a new and shorter proof of \cite[Proposition 4.1]{Laister}, with a
constant that does not depend on the Young function.

\subsection{Overview of the main results}

We write $\Y$ for the set of Young's functions and $\Nf$ for the set of
$N$-functions (Definitions \ref{def:orlicz} and \ref{def:MT} below).
{Roughly speaking, $N$-functions are the ``nice'' Young's functions: they
are finite, positive and convex, and they grow faster than linearly at
infinity and slower than linearly at the origin, that is, $\Phi(x)/x\to0$ as
$x\to0$ and $\Phi(x)/x\to\infty$ as $x\to\infty$. Typical examples are
$x^q$ with $q>1$ and $e^{x^p}-1$ with $p>1$, whereas $x$ and $e^{x}-1$ are
Young's functions which are not $N$-functions.} Further, $L_\Phi$ denotes the
Orlicz space built on $\Phi\in\Y$, $L_\Phi^0$ the closure of
$C_0^\infty(\R^d)$ in $L_\Phi$, and $M_\Phi$ the Morse--Transue subspace. For
nonnegative functions $g,h$, we write $g\lesssim h$ if $g(x)\le h(Kx)$ for all
large $x$ and some constant $K>0$; we write $g\gtrsim h$ if $h\lesssim g$, and
$g\approx h$ if both relations hold.

Following \cite[(3.1)]{Laister}, two Young's functions $\Phi,\Psi$ are called
\textbf{\emph{compatible}} if there exist a constant $b\in(0,1]$ and a Young's function
$\Theta$ such that
\begin{equation}\label{E3.1}
b\,x\,\Psi^{-1}(x)\ \le\ \Phi^{-1}(x)\,\Theta^{-1}(x)\ \le\ x\,\Psi^{-1}(x),
\qquad x\ge 0 .
\end{equation}
The constant $b$ in \eqref{E3.1} is \emph{not} the fractional exponent
$\beta$; in \cite{Laister} the same letter is used for both.

\begin{mainthm}[outline; Theorem~\ref{thm:OrliczSmoothing} and
Corollaries~\ref{cor:Linfty}--\ref{cor:decay}]
Let $0<\beta\le 1$ and $\varrho>0$. For every $\Phi\in\Y$ and $t\ge0$,
$$
\|S_\beta(t)\varphi\|_{L_\Phi}\le \min\bigl\{1,2^{d/2}e^{-\lambda_\beta t}\bigr\}
\|\varphi\|_{L_\Phi},
$$
$S_\beta$ leaves $M_\Phi$ invariant, and it is a $C_0$-semigroup on
$M_\Phi$ whenever $\Phi\in\Nf$. For compatible Young's functions $\Phi$ and
$\Psi$ there is $C=C(d,\beta,b)$, independent of $\Phi,\Psi,\Theta$ and
$\varrho$, such that
$$
\bigl\|S_\beta(t)\varphi\bigr\|_{L_\Psi}
\ \le\ C\,
\frac{\Phi^{-1}\bigl(t^{-\frac{d}{2\beta}}\bigr)}
     {\Psi^{-1}\bigl(t^{-\frac{d}{2\beta}}\bigr)}\,
\|\varphi\|_{L_\Phi},\qquad t>0 ,
$$
and in particular
$\|S_\beta(t)\varphi\|_{L^\infty}\le C\,\Phi^{-1}(t^{-d/(2\beta)})
\|\varphi\|_{L_\Phi}$. For every $\theta\in(0,1)$, both estimates may be
multiplied by $C_\theta e^{-(1-\theta)\lambda_\beta t}$, and the rate
$\lambda_\beta$ cannot be improved.
\end{mainthm}

Two points deserve emphasis. First, the constant does not depend on the
Orlicz spaces involved, whereas the constant in \cite[Theorem A]{Laister}
depends on $\Phi$ and $\Psi$. Second, the semigroup is a genuine contraction on
every Orlicz space, and not merely bounded with constant $2$ as in
\cite[Theorem A(a)(i)]{Laister}. This is classical for sub-Markovian
operators, and it doubles the admissible size of the integral in the global
existence condition (Remark~\ref{rem:sharper}).

Several further comments are in order. To the best of our knowledge, Theorem~A
is new even for $\beta=1$, that is, for the Hermite semigroup $e^{-t\Hr}$
itself. For $\varrho=0$ and $\beta=1$, our argument also yields the smoothing
estimate of \cite{Laister} with a constant independent of the Young's
functions (Remark~\ref{rem:reproves}).

Theorem~A combines two regimes. For short times, the singular factor, namely
$\Phi^{-1}(t^{-d/(2\beta)})$ divided by $\Psi^{-1}(t^{-d/(2\beta)})$, is
exactly that of the free fractional heat semigroup; the confining potential does not remove
the singularity at $t=0$, since locally it is a lower-order perturbation. For
instance, when $\beta=1$, Mehler's formula \eqref{mc} gives
$\Kb_t(0,0)=\bigl(\omega/(2\pi\sinh 2\omega t)\bigr)^{d/2}\sim(4\pi t)^{-d/2}$
as $t\to0$, so the $L^1\to L^\infty$ norm of $e^{-t\Hr}$, which equals
$\sup_{x,y}\Kb_t(x,y)$, blows up exactly like that of the Gaussian semigroup.
For large times, on the other hand, the potential produces exponential decay
at the rate $\lambda_\beta$, a phenomenon without analogue in the free case.
Indeed, the free semigroup preserves constants, so that its norm on
$L^\infty$, and more generally on $L^1+L^\infty$, does not decay at all,
whereas the norm of $S_\beta(t)$ on
every Orlicz space is at most $2^{d/2}e^{-\lambda_\beta t}$.

Furthermore, the confining potential enters the proof of Theorem~A only through
the mass bound \eqref{eq:massbeta} and the ground state of $\Hr$. All the
assertions of Theorem~A which do not involve $\lambda_\beta$ (the contraction
property, the invariance of $M_\Phi$, the $C_0$-property and the smoothing
estimate with the constant $2\kappa/b$) therefore hold, with the same proofs,
for the free fractional heat semigroup $e^{-t(-\Delta)^{\beta}}$, whose kernel
$p_\beta$ is nonnegative and has unit mass. A careful inspection of the proofs suggests that one should expect the same conclusions to remain valid for the local theory (Theorem~\ref{thm:local}), Corollary~\ref{thm:global}, and the nonexistence results in Section~\ref{sec:nonexistence}. Thus, it is reasonable to expect that the case $\varrho=0$ could be included throughout, except in the statements involving $\lambda_\beta$ and in Theorem~\ref{thm:confined}.

Finally, Theorem~A concerns only the linear semigroup. It may therefore be
useful in the study of other nonlinear evolution problems whose Duhamel
formulation involves $S_\beta(t)$, such as reaction--diffusion systems or heat
equations with nonlocal nonlinear terms, and we expect it to be of independent
interest.

\begin{Remark}\rm[Sharpness of the large-time rate]\label{rem:sharp-intro}

For $\Phi=\Psi$, the exponential rate in Theorem~A is optimal in a strong
sense. Let $\omega=\sqrt\varrho$ and let $h_0(x)=e^{-\omega|x|^2/2}$ be the
ground state of $\Hr$, so that $\Hr h_0=d\omega h_0$. By
Lemma~\ref{lem:kernel}(iv), $S_\beta(t)h_0=e^{-\lambda_\beta t}h_0$, and
therefore
$$
\|S_\beta(t)h_0\|_{L_\Phi}=e^{-\lambda_\beta t}\|h_0\|_{L_\Phi},
\qquad t\ge0,\ \Phi\in\Y ,
$$
where $0<\|h_0\|_{L_\Phi}<\infty$ (see the proof of
Theorem~\ref{thm:OrliczSmoothing}(a)(iv)). We deduce that
$$
e^{-\lambda_\beta t}\ \le\ \|S_\beta(t)\|_{L_\Phi\to L_\Phi}
\ \le\ \min\bigl\{1,2^{d/2}e^{-\lambda_\beta t}\bigr\},
\qquad t\ge0 ,
$$
so that the operator norm is determined up to the factor $2^{d/2}$. On
$L^2$, the spectral theorem gives $\|S_\beta(t)\|_{L^2\to L^2}=e^{-\lambda_\beta t}$
exactly. For short times, the profile $\Phi^{-1}(t^{-d/(2\beta)})$ is sharp in
the case $\beta=1$, $\Phi(x)=x$ (that is, from $L^1$ to $L^\infty$), by the
computation of $\Kb_t(0,0)$ above. We do not address the optimality of the
short-time factor for general pairs of Orlicz spaces.
\end{Remark}

Recall that $\ell$ denotes the Lipschitz modulus of $f$, see
\eqref{eq:lipschitzmodulus}.

\begin{mainthm}[outline; Theorem~\ref{thm:local},
Corollary~\ref{thm:global} and Theorem~\ref{thm:confined}]
Let $\Phi\in\Nf$ and let $f$ be locally Lipschitz with $f(0)=0$. Assume that
there exists $\lambda>0$ such that
\begin{equation}\label{eq:IC-infty}
\int_1^\infty x^{-p_\ast}\,\ell\bigl(\lambda\,\Phi^{-1}(x)\bigr)\dd x<\infty .
\end{equation}
Then \eqref{eq:main} is locally well posed in $M_\Phi$: solutions exist on a
common interval for data in compact subsets of $M_\Phi$, they are unique in
the class of bounded $M_\Phi$-mild solutions, they depend continuously on the
data, and they are ordered if $f$ is increasing. If
\begin{equation}\label{eq:IC-zero}
\int_0^\infty x^{-p_\ast}\,\ell\bigl(\lambda\,\Phi^{-1}(x)\bigr)\dd x
<\frac{d}{2\beta},
\end{equation}
then small data give global solutions. Finally, if \eqref{eq:IC-infty}
holds and $\int_0^1\ell(r)\,\dd r/r<\infty$, then small data give global
solutions which decay like $e^{-\lambda_\beta t}$ in $L^\infty$ and in
$L_\Phi$.
\end{mainthm}

\emph{On the assumption $f(0)=0$.} This normalisation is used in an essential
way. First, it makes $u\equiv0$ a solution, around which the global results
are perturbative statements. Second, it gives the bound $|f(u)|\le\ell(s)|u|$
for $|u|\le s$ in \eqref{eq:lipbound}, which shows that $f(u)$ belongs to
$M_\Phi$ whenever $u\in M_\Phi\cap L^\infty$. This bound is what makes the
barrier $w=2S_\beta(t)|u_0|$ invariant under the Duhamel map in the proof of
Theorem~\ref{thm:local}. If $f(0)\ne0$, then $u\equiv0$ is no longer a
solution, and the Duhamel term generated by the constant source $f(0)$ is not
controlled by the initial data. For instance, for the free semigroup ($\varrho=0$),
which preserves constants, a nonzero constant does not belong to any Orlicz
space $L_\Phi$ with $\Phi$ positive. A different functional framework would
then be needed, and we do not pursue this here.

The dichotomy between \eqref{eq:IC-infty} and \eqref{eq:IC-zero} reflects the
familiar fact that local existence is governed by $f$ at infinity, whereas
global existence also feels $f$ near zero. For $\beta=1$, condition
\eqref{eq:IC-zero} relaxes the condition
$\int_0^\infty x^{-p_\ast}\ell(\lambda\Phi^{-1}(x))\dd x<n/4$ of
\cite[Corollary B]{Laister} to $<n/2$. This improvement comes from the
contraction property, and it is valid also for $\varrho=0$. The last statement
of Theorem~B has no counterpart for the free equation; it is the announced
confinement effect. For $f(u)=|u|^{m-1}u$, it yields global solutions for small
data for every $m>1$, whereas for $\varrho=0$ all nonnegative nontrivial
solutions blow up when $1<m\le p_\ast$
\cite{Fujita1966,Hayakawa1973,Sugitani1975}.

Following \cite{Laister}, we impose on $f$ a convexity assumption
$(\mathrm{C})$, a Lipschitz-majorant assumption $(\Lambda)$ and a growth
assumption $(\mathrm{G})$, recalled in Section \ref{sec:results}. The last of
these requires $(\log f)'$ to be regularly varying of index $\rho-1$ with
$\rho\in(0,\infty]$; the value $\rho=\infty$ corresponds to rapid variation,
and thus to arbitrarily fast growth.

\begin{mainthm}[outline; Theorem~\ref{thm:critical}]
Let $f$ satisfy $(\mathrm{C})$, $(\Lambda)$ and $(\mathrm{G})$. If
$\Phi\in\Nf$ and $\Phi\gtrsim f$ at infinity, then \eqref{eq:main} is locally
well posed in $M_\Phi$. Conversely, if $\Phi\in\Y$ and $\Phi\lesssim f$ at
infinity, then there exists a nonnegative $u_0\in L_\Phi$ for which
\eqref{eq:main} admits no nonnegative integral solution on $\R^d\times(0,T)$
for any $T>0$.
\end{mainthm}

Thus $\Phi_c(u):=f(u)$ is a critical Young's function for \eqref{eq:main},
exactly as in the free case \cite[Theorem C]{Laister}. The nonexistence half
rests on two ingredients, which we prove in full. The first is a necessary
condition for local solvability: $F\bigl(S_\beta(t)u_0\bigr)\ge t$ whenever a
nonnegative integral solution exists on $(0,T)$. The second is a lower bound
for $S_\beta(t)u_0$ when $u_0$ is the singular profile
$$
u_0(x)=F^{-1}\bigl(\gamma|x|^{2\beta}\bigr)\chi_{B_R}(x),
$$
We deduce nonexistence for every $\gamma$ below an explicit threshold
$\gamma_0(d,\beta)$ (Lemma~\ref{L7.1}). In \cite{Laister}, the corresponding
step is imported from \cite[Corollary 1.1, Theorem 1.1]{Fujishima2024}. The
exponent $2\beta$ is dictated by the self-similarity of the free fractional
heat kernel, and our argument uses the fact that the $\beta$-stable
subordinator has finite moments exactly of order less than $\beta$.

Table~\ref{tab:comparison} compares the present results with those of
\cite{Laister}.

\begin{table}[ht]
\centering
\small
\renewcommand{\arraystretch}{1.3}
\begin{tabular}{>{\raggedright}p{4.3cm} >{\raggedright}p{4.9cm}
>{\raggedright\arraybackslash}p{5.3cm}}
\toprule
 & Free heat equation \cite{Laister} ($\beta=1$, $\varrho=0$, $\R^n$) &
 Present paper ($0<\beta\le1$, $\varrho>0$, $\R^d$)\\
\midrule
Kernel & Gaussian $G(x-y,t)$ (convolution) & subordinated Mehler kernel
$\Kb^\beta_t(x,y)\le p_\beta(x-y,t)$\\
Bound of $S(t)$ on $L_\Phi$ & $2$ & $\min\{1,2^{d/2}e^{-\lambda_\beta t}\}$\\
Constant in $L_\Phi\to L_\Psi$ smoothing & depends on $n,\Phi,\Psi$ &
depends on $d,\beta,b$ only\\
Short-time scale & $t^{-n/2}$ & $t^{-d/(2\beta)}$\\
Exponent in the integral conditions & $1+2/n$ & $p_\ast=1+2\beta/d$\\
Global condition \eqref{eq:IC-zero} & $\int_0^\infty<n/4$ & $\int_0^\infty
<d/(2\beta)$\\
Large-time decay of small solutions & $\lambda\Phi^{-1}(t^{-n/2})$ &
$e^{-\lambda_\beta t}$ (sharp rate)\\
Small-data GWP without conditions on $\Phi$ near $0$ & not available &
under a Dini condition on $f$ at $0$ (Theorem~\ref{thm:confined})\\
Critical Young function & $\Phi_c=f$ & $\Phi_c=f$\\
Nonexistence data & $F^{-1}(\gamma|x|^2)\chi_{B_R}$, via
\cite{Fujishima2024} & $F^{-1}(\gamma|x|^{2\beta})\chi_{B_R}$, proved here,
explicit $\gamma_0$\\
Uniqueness of mild solutions & if $\Phi\gtrsim f^r$ near $0$, $r>n/2$ & if
$\Phi\gtrsim f^r$ near $0$, $r>d/(2\beta)$\\
\bottomrule
\end{tabular}
\caption{Comparison with the free case treated in \cite{Laister}.}
\label{tab:comparison}
\end{table}

\subsection{Organisation of the paper}
Section~\ref{sec:prelim} collects the preliminary material: Orlicz and
Morse--Transue spaces, regular variation, the fractional harmonic oscillator,
its Mehler kernel and Bochner subordination, the relevant properties of the
stable subordinator, and the solution concepts. The main results are stated
in Section~\ref{sec:results}. Section~\ref{sec:kernel} establishes the
pointwise, mass and Orlicz-norm bounds for the kernel of $S_\beta(t)$. These
bounds are used in Section~\ref{sec:proofA} to prove the smoothing estimates,
and in Section~\ref{sec:proofB} to prove the well-posedness results.
Section~\ref{sec:nonexistence} develops the nonexistence theory and proves the
results on critical Young's functions. Section~\ref{sec:examples} illustrates
the theory on concrete nonlinearities, and Section~\ref{sec:open} lists some
open problems.

\section{Preliminaries}\label{sec:prelim}

Throughout, $B_r$ denotes the open ball of radius $r>0$ centred at the origin
in $\R^d$, $\chi_E$ is the characteristic function of a measurable set
$E\subseteq\R^d$, and $\chi_r:=\chi_{B_r}$. All function spaces are over
$\R^d$, and this is suppressed from the notation, so that $L^q:=L^q(\R^d)$,
$L_\Phi:=L_\Phi(\R^d)$, and so forth. We write $C$ for a generic positive
constant, whose value may change from line to line, and we indicate its
dependence on parameters where this matters. ``Increasing'' means
nondecreasing, and ``decreasing'' means nonincreasing. For nonnegative $g,h$
defined near infinity, we write $g\lesssim h$ if there exist $K,x_0>0$ with
$g(x)\le h(Kx)$ for $x\ge x_0$; $g\gtrsim h$ means $h\lesssim g$, and $g\approx
h$ means that both relations hold. We write ``$g\lesssim h$ near $0$'' if
$g(x)\le h(Kx)$ for $0\le x\le x_0$. We say that $g$ is \emph{asymptotic} to
$h$, and write $g\sim h$, if $g(x)/h(x)\to1$. For a function $u$ on
$Q_T:=\R^d\times(0,T)$, we write $u(t)=u(\cdot,t)$; identities between
functions on $Q_T$ are understood almost everywhere unless stated otherwise.

\subsection{Orlicz and Morse--Transue spaces}

We follow the conventions of \cite[Section 2.1]{Laister}, in which
extended-real-valued Young's functions are permitted. This is precisely what
allows $L^\infty$ to be realised as an Orlicz space. As a consequence, the
$L^\infty$-smoothing estimate becomes a particular case of the
Orlicz-to-Orlicz one.
General references for Orlicz spaces are
\cite{RR1,RR,HH2019,KR,BS,Adams,Maligranda1989}.

\begin{Definition}[Young's function and Orlicz space]\label{def:orlicz}
A function $\Phi:[0,\infty]\to[0,\infty]$ with $\Phi(0)=0$ is a
\emph{Young's function} if one of the following holds:
\begin{enumerate}[label=\textup{(\Roman*)}]
\item $\Phi$ is convex, increasing, finite-valued and not identically zero on
$[0,\infty)$;
\item there exists $x_\infty^\Phi>0$ such that $\Phi$ is convex, increasing
and finite-valued on $[0,x_\infty^\Phi]$ and $\Phi(x)=\infty$ for
$x>x_\infty^\Phi$;
\item there exists $x_\infty^\Phi>0$ such that $\Phi$ is convex, increasing
and finite-valued on $[0,x_\infty^\Phi)$, $\lim_{x\uparrow
x_\infty^\Phi}\Phi(x)=\infty$ and $\Phi(x)=\infty$ for $x\ge x_\infty^\Phi$.
\end{enumerate}
The set of Young's functions is denoted by $\Y$. We say that $\Phi\in\Y$ is
\emph{finite} if it satisfies \textup{(I)}, and \emph{positive} if $\Phi>0$ on
$(0,\infty)$.

The \emph{Orlicz space} $L_\Phi$ is the set of measurable $u:\R^d\to\R$ such
that $\int_{\R^d}\Phi(|u(x)|/k)\dd x<\infty$ for some $k>0$, normed by the
Luxemburg norm
$$
\|u\|_{L_\Phi}
:=\inf\Bigl\{k>0:\ \int_{\R^d}\Phi\Bigl(\frac{|u(x)|}{k}\Bigr)\dd x\le 1\Bigr\}.
$$
We write $L_\Phi^0$ for the closure of $C_0^\infty(\R^d)$ in $L_\Phi$.
\end{Definition}

\begin{Definition}[$N$-function and Morse--Transue space]\label{def:MT}
A Young's function $\Phi$ is an \emph{$N$-function} if it is finite and
positive and satisfies
$$
\lim_{x\to0}\frac{\Phi(x)}{x}=\lim_{x\to\infty}\frac{x}{\Phi(x)}=0 .
$$
The set of $N$-functions is denoted by $\Nf$. The \emph{Morse--Transue space}
\cite{MorseTransue1950} is
$$
M_\Phi:=\Bigl\{u\ \text{measurable}:\
\int_{\R^d}\Phi\Bigl(\frac{|u(x)|}{k}\Bigr)\dd x<\infty\ \text{for every }k>0
\Bigr\},
$$
equipped with the norm of $L_\Phi$.
\end{Definition}

\begin{Definition}[Generalised inverse]\label{def:inverse}
For $\Phi\in\Y$ the \emph{generalised inverse}
$\Phi^{-1}:[0,\infty]\to[0,\infty]$ is defined by
$$
\Phi^{-1}(y):=\inf\{x> 0:\ \Phi(x)>y\},\qquad y\ge 0,
$$
with $\inf\emptyset=\infty$.
\end{Definition}

We shall use the following elementary facts, for which we refer to
\cite[Remark 2.2]{Laister} and \cite[Property 1.3]{RR1}. The function
$\Phi^{-1}$ is finite-valued, increasing and right-continuous on $[0,\infty)$,
with $\Phi^{-1}(x)>0$ for $x>0$. It is unbounded if and only if $\Phi$ is
finite. If $\Phi$ is finite and positive, then $\Phi$ is a bijection of
$[0,\infty)$ and $\Phi^{-1}$ is its inverse. Moreover,
\begin{equation}\label{eq:22}
\Phi\bigl(\Phi^{-1}(x)\bigr)\le x,\qquad \Phi^{-1}\bigl(\Phi(x)\bigr)\ge x,
\qquad x\in[0,\infty).
\end{equation}
In particular, if $\Phi$ is positive, then $\Phi^{-1}(0)=0$ and, by right
continuity, $\Phi^{-1}(x)\to0$ as $x\downarrow0$. The Young complement of
$\Phi\in\Y$ is $\Phi^{*}(x):=\sup_{y\ge0}\bigl(xy-\Phi(y)\bigr)$. It belongs to
$\Y$ and satisfies \cite[Property 1.6]{RR1}, \cite[(2.4)]{Laister}
\begin{equation}\label{eq:24}
x\ \le\ \Phi^{-1}(x)\,(\Phi^{*})^{-1}(x)\ \le\ 2x,\qquad x\ge 0 .
\end{equation}
From Young's inequality $xy\le\Phi(x)+\Phi^*(y)$, we obtain the generalised
H\"older inequality
$$
\int_{\R^d}|uv|\dd x\le 2\|u\|_{L_\Phi}\|v\|_{L_{\Phi^*}} .
$$

\begin{Lemma}[{\cite[Lemma 4.1]{Laister}}]\label{lem:concave}
If $\Phi\in\Y$ then $\Phi^{-1}$ is concave on $[0,\infty)$ and $x/\Phi^{-1}(x)$
is increasing on $(0,\infty)$. In particular
\begin{equation}\label{eq:dilate}
\Phi^{-1}(x)\ \le\ \Phi^{-1}(c x)\ \le\ c\,\Phi^{-1}(x),
\qquad x\ge0,\ c\ge1 .
\end{equation}
\end{Lemma}

\begin{proof}
Concavity and the monotonicity of $x/\Phi^{-1}(x)$ are
\cite[Lemma 4.1]{Laister}. In \eqref{eq:dilate} the left inequality is the
monotonicity of $\Phi^{-1}$. For the right one, write
$x=c^{-1}(c x)+(1-c^{-1})\cdot0$; concavity of $\Phi^{-1}$ and
$\Phi^{-1}(0)\ge0$ give
$\Phi^{-1}(x)\ge c^{-1}\Phi^{-1}(c x)+(1-c^{-1})\Phi^{-1}(0)
\ge c^{-1}\Phi^{-1}(c x)$.
\end{proof}

The next lemma collects the structural properties of the Luxemburg norm that
will be used. All of them are elementary, and we include the short proofs for
completeness.

\begin{Lemma}\label{lem:orliczbasic}
Let $\Phi\in\Y$.
\begin{enumerate}[label=\textup{(\roman*)}]
\item \emph{(Lattice and Fatou properties.)} If $|u|\le|v|$ a.e.\ then
$\|u\|_{L_\Phi}\le\|v\|_{L_\Phi}$, and $v\in M_\Phi$ implies $u\in M_\Phi$. If
$u_j\to u$ a.e.\ then $\|u\|_{L_\Phi}\le\liminf_{j}\|u_j\|_{L_\Phi}$.
\item There exist $c>0$ and $C_0\ge0$ such that $\Phi(z)\ge cz-C_0$ for all
$z\ge0$; in particular $\Phi(z)\to\infty$ as $z\to\infty$.
\item If $\Phi$ is finite, then
$\|h\|_{L_\Phi}\le\max\{1,\Phi(1)\}\bigl(\|h\|_{L^1}+\|h\|_{L^\infty}\bigr)$ for
every $h\in L^1\cap L^\infty$, and $L^1\cap L^\infty\subseteq M_\Phi$.
\item $L_\Phi$ is a Banach space and $M_\Phi$ and $L^0_\Phi$ are closed
subspaces of it. If $\Phi$ is finite then $L^0_\Phi\subseteq M_\Phi$; if
$\Phi\in\Nf$ then $L^0_\Phi=M_\Phi$; if $\Phi\in\Nf$ satisfies the
$\Delta_2$-condition $\Phi(2x)\le C\Phi(x)$, $x\ge0$, then
$L^0_\Phi=M_\Phi=L_\Phi$.
\end{enumerate}
\end{Lemma}

\begin{proof}
(i) The first two statements are immediate from the monotonicity of $\Phi$.
If $u_j\to u$ a.e.\ and $k>\liminf_j\|u_j\|_{L_\Phi}$, then along a subsequence
$\int\Phi(|u_j|/k)\le1$, and since $\Phi$ is increasing and lower
semicontinuous, Fatou's lemma gives $\int\Phi(|u|/k)\le1$.

(ii) If $\Phi=0$ on $[0,x_\infty^\Phi)$ then $\Phi(z)\ge z-x_\infty^\Phi$.
Otherwise there is $z_0>0$ with $0<\Phi(z_0)<\infty$, and convexity together
with $\Phi(0)=0$ gives $\Phi(z)\ge (z/z_0)\Phi(z_0)$ for $z\ge z_0$; hence
$\Phi(z)\ge c(z-z_0)$ with $c=\Phi(z_0)/z_0$.

(iii) Let $k:=\max\{1,\Phi(1)\}(\|h\|_{L^1}+\|h\|_{L^\infty})>0$. Then
$|h|/k\le1$ a.e., so by convexity
$\Phi(|h|/k)\le(|h|/k)\Phi(1)$ and
$\int\Phi(|h|/k)\le\Phi(1)\|h\|_{L^1}/k\le1$. The same computation with $k$
replaced by any $k'>0$, applied to $h/k'$, shows $h\in M_\Phi$.

(iv) These facts are \cite[Remark 2.1(a), Lemma 4.2]{Laister}; the equality
$L^0_\Phi=M_\Phi$ for $N$-functions relies on
\cite[Theorem 8.21(d)]{Adams}.
\end{proof}

We shall also need Minkowski's integral inequality for the Luxemburg norm.

\begin{Lemma}\label{lem:minkowski}
Let $\Phi\in\Y$, let $I\subseteq\R$ be an interval and let
$g:\R^d\times I\to[0,\infty]$ be measurable. Then $s\mapsto\|g(\cdot,s)\|_{L_\Phi}$
is measurable and
$$
\Bigl\|\int_I g(\cdot,s)\dd s\Bigr\|_{L_\Phi}\ \le\ \int_I\|g(\cdot,s)\|_{L_\Phi}\dd s .
$$
\end{Lemma}

\begin{proof}
For rational $k>0$ the function $s\mapsto\int\Phi(g(x,s)/k)\dd x$ is measurable
by Tonelli's theorem, and
$\{s:\|g(\cdot,s)\|_{L_\Phi}<a\}=\bigcup_{k\in\mathbb Q\cap(0,a)}
\{s:\int\Phi(g(x,s)/k)\dd x\le1\}$, which proves measurability. Let
$N(s):=\|g(\cdot,s)\|_{L_\Phi}$ and assume $A:=\int_IN(s)\dd s<\infty$. Fix
$\varepsilon>0$ and a positive $e\in L^1(I)$ with $\int_Ie=\varepsilon$, and put
$k(s):=N(s)+e(s)>N(s)$ and $K:=\int_Ik=A+\varepsilon$. By the definition of
the Luxemburg norm, $\int\Phi(g(x,s)/k(s))\dd x\le1$ for every $s$. For fixed
$x$, Jensen's inequality for the probability measure $k(s)\dd s/K$ and the
lower semicontinuous convex function $\Phi$ gives
$$
\Phi\Bigl(\frac1K\int_Ig(x,s)\dd s\Bigr)
=\Phi\Bigl(\int_I\frac{g(x,s)}{k(s)}\,\frac{k(s)\dd s}{K}\Bigr)
\le\int_I\Phi\Bigl(\frac{g(x,s)}{k(s)}\Bigr)\frac{k(s)\dd s}{K};
$$
if the mean on the left is infinite, the right-hand side is infinite as well
by Lemma~\ref{lem:orliczbasic}(ii). Integrating in $x$ and using Tonelli's
theorem, $\int\Phi\bigl(K^{-1}\int_Ig\dd s\bigr)\dd x\le\int_Ik(s)\dd s/K=1$,
whence $\|\int_Ig\dd s\|_{L_\Phi}\le A+\varepsilon$.
\end{proof}

Finally, we record a completeness statement adapted to the fixed point
arguments of Section~\ref{sec:proofB}.

\begin{Lemma}\label{lem:complete}
Let $\Phi\in\Nf$, $T\in(0,\infty]$, and let $w:Q_T\to[0,\infty)$ be measurable.
Let $X$ be the set of (equivalence classes of) measurable $u:Q_T\to\R$ with
$|u|\le w$ a.e.\ in $Q_T$, and put
$d(u,v):=\esssup_{t\in(0,T)}\|u(t)-v(t)\|_{L_\Phi}$. Then every sequence in $X$
{which is Cauchy with respect to $d$ converges, with respect to $d$, to an
element of $X$.} In particular, $(X,d)$ is a complete metric space whenever $d$ is finite on
$X\times X$, for instance if $\esssup_t\|w(t)\|_{L_\Phi}<\infty$.
\end{Lemma}

\begin{proof}
The measurability of $t\mapsto\|u(t)-v(t)\|_{L_\Phi}$ follows from
Lemma~\ref{lem:minkowski}. Let $(u_j)$ be $d$-Cauchy and
$\varepsilon_j:=\sup_{k\ge j}d(u_j,u_k)\to0$. Discarding a null set of times,
$\|u_j(t)-u_k(t)\|_{L_\Phi}\le\varepsilon_j$ for all $k\ge j$ and all remaining
$t$. If $\|h\|_{L_\Phi}\le\varepsilon$ then $\int\Phi(|h|/\varepsilon')\le1$
for every $\varepsilon'>\varepsilon$, so that
$|\{x:|h(x)|>\delta\}|\le1/\Phi(\delta/\varepsilon')$ for every $\delta>0$.
Integrating in $t$, for every $R,T'\in(0,\infty)$ with $T'\le T$ and every
$\delta>0$,
$$
\bigl|\{(x,t)\in B_R\times(0,T'):|u_j-u_k|>\delta\}\bigr|\le
\frac{T'}{\Phi(\delta/(2\varepsilon_j))}\longrightarrow0\qquad(j\to\infty),
$$
by Lemma~\ref{lem:orliczbasic}(ii). Thus $(u_j)$ is Cauchy in measure on each
$B_R\times(0,T')$, and a diagonal argument yields a subsequence $(u_{j_i})$
and a measurable $u$ with $u_{j_i}\to u$ a.e.\ in $Q_T$; clearly $|u|\le w$.
By Fubini's theorem, for a.e.\ $t$ we have $u_{j_i}(t)\to u(t)$ a.e.\ in
$\R^d$, and Lemma~\ref{lem:orliczbasic}(i) gives
$\|u(t)-u_k(t)\|_{L_\Phi}\le\liminf_i\|u_{j_i}(t)-u_k(t)\|_{L_\Phi}\le
\varepsilon_k$. Hence $d(u,u_k)\le\varepsilon_k\to0$.
\end{proof}

\subsection{Regularly varying functions}

A measurable $g:(a,\infty)\to(0,\infty)$ is \emph{regularly varying of index}
$\rho\in\R$ if $g(\lambda x)/g(x)\to\lambda^{\rho}$ as $x\to\infty$ for every
$\lambda>0$; when $\rho=0$ we say that $g$ is \emph{slowly varying}. If instead
\begin{equation}\label{eq:rapid}
\lim_{x\to\infty}\frac{g(\lambda x)}{g(x)}=
\begin{cases} 0, & 0<\lambda<1,\\ \infty,&\lambda>1,\end{cases}
\end{equation}
then $g$ is said to be \emph{rapidly varying of index} $\infty$; if the two
limits in \eqref{eq:rapid} are interchanged, $g$ is rapidly varying of index
$-\infty$. Following \cite{Laister}, we include rapidly varying functions
among the regularly varying ones, with $\rho=\pm\infty$. A function $g$ is
regularly (slowly) varying at the origin if $x\mapsto g(1/x)$ is regularly
(slowly) varying at infinity, the index changing sign. By Karamata's
characterisation theorem \cite{Karamata1930,S,BGT}, a regularly varying $g$ of
finite index $\rho$ can be written as $g(x)=x^{\rho}\sigma(x)$ with $\sigma$
slowly varying. If $\sigma$ is slowly varying at infinity, then $x^{k}\sigma(x)
\to\infty$ and $x^{-k}\sigma(x)\to0$ as $x\to\infty$ for every $k>0$, and
$\int_a^\infty x^{-k}\sigma(x)\dd x<\infty$ for every $k>1$
\cite[Propositions 1.3.6 and 1.5.10]{BGT}.

\subsection{The fractional harmonic oscillator}\label{subsec:operator}

Write $\omega:=\sqrt{\varrho}>0$. The Hermite operator
$\Hr=-\Delta+\omega^2|x|^2$ is essentially self-adjoint on $C_0^\infty(\R^d)$,
and its spectrum is the discrete set $\{\omega(2k+d):k\in\N\cup\{0\}\}$, the
bottom eigenvalue $d\omega$ being simple with eigenfunction
$h_0(x)=e^{-\omega|x|^2/2}$ \cite[Chapter 1]{thangavelu1993lectures}. In
particular
\begin{equation}\label{eq:gap}
\inf\operatorname{spec}\Hr=d\omega=d\sqrt{\varrho},
\qquad
\inf\operatorname{spec}\Hr^{\beta}=\lambda_\beta:=\bigl(d\sqrt{\varrho}\,\bigr)^{\beta}.
\end{equation}
The semigroup $e^{-t\Hr}$ is given by the Mehler kernel
\cite[Chapter 1]{thangavelu1993lectures}
\begin{equation}\label{mc}
\Kb_t(x,y)=\Bigl(\frac{\omega}{2\pi\sinh(2\omega t)}\Bigr)^{d/2}
\exp\Bigl(-\frac{\omega}{4}\Bigl[|x+y|^2\tanh(\omega t)
+|x-y|^2\coth(\omega t)\Bigr]\Bigr),
\end{equation}
which is the familiar expression
$(\omega/(2\pi\sinh 2\omega t))^{d/2}\exp\bigl(-\frac{\omega}{2}\bigl[(|x|^2+|y|^2)
\coth(2\omega t)-\frac{2x\cdot y}{\sinh(2\omega t)}\bigr]\bigr)$ rewritten by
means of the identities $\tanh z+\coth z=2\coth 2z$ and $\coth z-\tanh
z=2/\sinh 2z$. For heat kernel methods in general we refer to \cite{Davies,Ouhabaz}; the
Mehler kernel of the fractional oscillator is also discussed in
\cite{BMMfrac}.
The kernel is smooth, strictly positive and symmetric,
$\Kb_t(x,y)=\Kb_t(y,x)$, and satisfies the Chapman--Kolmogorov identity
$\int\Kb_t(x,z)\Kb_s(z,y)\dd z=\Kb_{t+s}(x,y)$.

For $0<\beta<1$, we use Bochner's subordination
\cite{Bochner1949}, \cite[Chapter IX.11]{Yosida}, \cite{SSV}. Let
$\eta^\beta_t$ denote the density of the $\beta$-stable subordinator at time
$t$, that is, the unique family of nonnegative functions on $(0,\infty)$ with
\begin{equation}\label{p1}
\int_0^{\infty}e^{-\sigma s}\,\eta^\beta_t(s)\dd s=e^{-t\sigma^{\beta}},
\qquad \sigma\ge 0,\ t>0 ;
\end{equation}
in particular $\int_0^\infty \eta^\beta_t(s)\dd s=1$, $\eta^\beta_t*\eta^\beta_\tau
=\eta^\beta_{t+\tau}$ (convolution on $(0,\infty)$), and
\begin{equation}\label{eq:scaling-eta}
\eta^\beta_t(s)=t^{-1/\beta}\,\eta^\beta_1\bigl(t^{-1/\beta}s\bigr),
\qquad s,t>0 .
\end{equation}
By the spectral theorem and \eqref{p1}, for $\varphi\in L^2$,
\begin{equation}\label{E7.4}
S_\beta(t)\varphi=e^{-t\Hr^{\beta}}\varphi
=\int_0^{\infty}e^{-s\Hr}\varphi\,\eta^\beta_t(s)\dd s ,
\end{equation}
whence $S_\beta(t)$ is an integral operator with the nonnegative symmetric
kernel
\begin{equation}\label{E7.5}
\Kb^\beta_t(x,y)=\int_0^{\infty}\Kb_s(x,y)\,\eta^\beta_t(s)\dd s .
\end{equation}
For $\beta=1$, we set $\Kb^1_t:=\Kb_t$, with the convention that $\eta^1_t$
is the Dirac mass at $s=t$; with this convention, \eqref{E7.4} and \eqref{E7.5}
remain valid.
The kernels $\Kb^\beta_t$ inherit the Chapman--Kolmogorov identity from
$\Kb_t$ and from $\eta^\beta_t*\eta^\beta_\tau=\eta^\beta_{t+\tau}$. For a
measurable $\varphi$ such that $\int\Kb^\beta_t(x,y)|\varphi(y)|\dd y<\infty$
for a.e.\ $x$, we \emph{define}
$[S_\beta(t)\varphi](x):=\int\Kb^\beta_t(x,y)\varphi(y)\dd y$. For $\varphi\ge0$,
this is meaningful with values in $[0,\infty]$. The semigroup property
$S_\beta(t)S_\beta(\tau)\varphi=S_\beta(t+\tau)\varphi$ holds for $\varphi\ge0$
by Tonelli's theorem, and for every $\varphi$ with
$S_\beta(t+\tau)|\varphi|<\infty$ a.e. We denote by
\begin{equation}\label{eq:freekernel}
G(z,t):=(4\pi t)^{-d/2}e^{-|z|^2/(4t)},
\qquad
p_\beta(z,t):=\int_0^\infty G(z,s)\,\eta^\beta_t(s)\dd s
\end{equation}
the Gaussian heat kernel and the kernel of the free fractional heat semigroup
$e^{-t(-\Delta)^{\beta}}$, so that $p_1=G$.

We shall use the following moment properties of the stable subordinator.
Since they are elementary consequences of \eqref{p1}, we include the proofs.
Probabilistically, $\eta^\beta_t$ is the density of the law of a nonnegative
random variable $\mathsf S_t$ with $\E e^{-\sigma\mathsf S_t}=e^{-t\sigma^\beta}$,
{where $\E$ denotes the expectation with respect to the underlying
probability measure. In other words, \eqref{p1} says that the Laplace
transform of $\mathsf S_t$ is $e^{-t\sigma^\beta}$.}

\begin{Lemma}\label{lem:stable}
Let $0<\beta\le1$ and $t>0$, {and let $\Gamma$ denote Euler's Gamma
function.}
\begin{enumerate}[label=\textup{(\roman*)}]
\item For every $q>0$,
$\int_0^\infty s^{-q}\eta^\beta_t(s)\dd s=t^{-q/\beta}\,
\Gamma(1+q/\beta)/\Gamma(1+q)$.
\item If $\beta<1$, then
$\int_0^\infty s^{r}\eta^\beta_t(s)\dd s=t^{r/\beta}\,\Gamma(1-r/\beta)/\Gamma(1-r)$
for $0<r<\beta$, while the integral diverges for $r\ge\beta$. If $\beta=1$, it
equals $t^r$ for every $r>0$.
\item $\int_0^\infty\min\{1,s\}\,\eta^\beta_t(s)\dd s\le(1-e^{-1})^{-1}(1-e^{-t})\le
2t$.
\end{enumerate}
\end{Lemma}

\begin{proof}
For $\beta=1$ all statements are immediate, so let $\beta<1$.

(i) Insert the identity
$s^{-q}=\Gamma(q)^{-1}\int_0^\infty\sigma^{q-1}e^{-\sigma s}\dd\sigma$ and use Tonelli's theorem and \eqref{p1}:
$$
\int_0^\infty s^{-q}\eta^\beta_t(s)\dd s=\frac{1}{\Gamma(q)}\int_0^\infty
\sigma^{q-1}e^{-t\sigma^\beta}\dd\sigma
=\frac{t^{-q/\beta}\,\Gamma(q/\beta)}{\beta\,\Gamma(q)}
=t^{-q/\beta}\,\frac{\Gamma(1+q/\beta)}{\Gamma(1+q)} .
$$
(ii) For $0<r<1$ one has
$s^{r}=\frac{r}{\Gamma(1-r)}\int_0^\infty(1-e^{-\sigma s})\sigma^{-1-r}\dd\sigma$,
so that, by Tonelli's theorem and \eqref{p1},
$$
\int_0^\infty s^r\eta^\beta_t(s)\dd s=\frac{r}{\Gamma(1-r)}\int_0^\infty
\bigl(1-e^{-t\sigma^\beta}\bigr)\sigma^{-1-r}\dd\sigma .
$$ Near $\sigma=0$ the integrand behaves like
$t\sigma^{\beta-1-r}$, which is integrable if and only if $r<\beta$; for $r<\beta$
the substitution $\tau=t\sigma^\beta$ gives the stated value. For $r\ge1$
the divergence follows from the case $r'\in[\beta,1)$ and the inequality
$s^r\ge s^{r'}-1$, $s>0$.
(iii) Since $1-e^{-s}\ge(1-e^{-1})\min\{1,s\}$ for $s\ge0$, the claim follows
from \eqref{p1} with $\sigma=1$ and $1-e^{-t}\le t$.
\end{proof}

\begin{Remark}\rm\label{rem:betarange}
The restriction $0<\beta\le1$ is used throughout and cannot be dispensed with
by the present methods. For $\beta>1$ the subordination formula \eqref{E7.4}
is unavailable, and the semigroup $e^{-t\Hr^{\beta}}$ is in general not
positivity preserving; already the kernel of the free biharmonic semigroup
$e^{-t\Delta^2}$ changes sign. Fixed-time $L^p$--$L^q$ estimates for
$e^{-t\Hr^\beta}$ with $\beta>1$ are available in a restricted range of
exponents \cite{BMMfrac}. Every argument below that uses the positivity of
the kernel would have to be rethought; this concerns the contraction and
comparison properties, the barrier construction in the fixed point argument,
and the whole of Section~\ref{sec:nonexistence}. We refer to
Section~\ref{sec:open} for further comments.
\end{Remark}

\subsection{Solution concepts}

With $S_\beta$ as above we set
\begin{equation}\label{eq:duhamel}
\mathcal{F}(u;u_0)(t):=S_\beta(t)u_0+\int_0^t S_\beta(t-s)f(u(s))\dd s,
\qquad t\in(0,T),
\end{equation}
and adopt the following terminology, which follows
\cite[Definition 2.7]{Laister} and \cite[Section 15]{QS}.

\begin{Definition}[Solution concepts]\label{def:solutions}
Let $T\in(0,\infty]$ and let $u_0$ be measurable.
\begin{enumerate}[label=\textup{(\roman*)}]
\item A measurable and almost everywhere finite $u:Q_T\to\R$ is an
\emph{integral solution} of \eqref{eq:main} in $Q_T$ if $S_\beta(t)|u_0|$ and
$\int_0^tS_\beta(t-s)|f(u(s))|\dd s$ are finite a.e.\ in $Q_T$ and
$\mathcal F(u;u_0)=u$ a.e.\ in $Q_T$.
\item An integral solution $u$ is called an \emph{$M_\Phi$-mild solution} on
$[0,T)$ if, in addition, $u\in C\bigl([0,T),M_\Phi\bigr)$ and $u(0)=u_0$.
\item An $M_\Phi$-mild solution is a \emph{bounded $M_\Phi$-mild solution} if
moreover $u\in L^\infty_{\mathrm{loc}}\bigl((0,T),L^\infty\bigr)$, that is,
$u\in L^\infty(\R^d\times(a,b))$ whenever $0<a<b<T$.
\end{enumerate}
\end{Definition}

In (i), we require the finiteness of $S_\beta(t)|u_0|$ and of
$\int_0^tS_\beta(t-s)|f(u(s))|\dd s$, and not merely that of $S_\beta(t)u_0$
and of the Duhamel integral. The reason is that $[S_\beta(t)u_0](x)$ is defined
by the integral $\int\Kb^\beta_t(x,y)u_0(y)\dd y$, which is a Lebesgue
integral only if it converges absolutely. For a sign-changing data, the
integrals of the positive and negative parts of the integrand could both be
infinite, and the integral would then be undefined. The finiteness of
$S_\beta(t)|u_0|(x)$ is therefore exactly what makes $[S_\beta(t)u_0](x)$
meaningful, and it then gives $|S_\beta(t)u_0|\le S_\beta(t)|u_0|$. The same
applies to the Duhamel term; absolute convergence is moreover needed there to
apply Fubini's theorem, for instance in Lemma~\ref{lem:restart}. For
nonnegative data and nonnegative nonlinearities, the two requirements
coincide.

In (ii) and (iii), the continuity is understood for the representative
$t\mapsto u(t)$ obtained from the Duhamel formula, which is defined for every
$t$. When $\beta=1$, bounded $M_\Phi$-mild solutions are precisely the
$M_\Phi$-classical solutions of \cite[Definition 2.7]{Laister}; see
Remark~\ref{rem:classical}.

\begin{Remark}\rm\label{rem:classical}
We work throughout with bounded mild solutions rather than with classical
ones: for $\beta<1$ the operator $\Hr^{\beta}$ is nonlocal and the class
$C^{2,1}$ is no longer the natural one.

For $\beta=1$, every bounded $M_\Phi$-mild solution $u$ on $[0,T)$ is a
classical solution in $Q_T$, as in \cite[Section 2.3]{Laister} for
$\varrho=0$. Since this fact is not used anywhere in the paper, we only sketch
the argument. Fix $0<a<b<T$ and let $M$ bound $|u|$ on $\R^d\times(a,b)$. By
Lemma~\ref{lem:restart}(iii), for $t\in(a,b)$ we have
$u(t)=e^{-(t-a)\Hr}u(a)+v(t)$, where
$v(t):=\int_a^te^{-(t-s)\Hr}g(s)\dd s$ and $g:=f(u)$ satisfies
$|g|\le M\ell(M)$ on $\R^d\times(a,b)$. The Mehler kernel \eqref{mc} is smooth
for $t>0$, solves $\partial_t\Kb_t=-\Hr\Kb_t$, and its derivatives are bounded,
locally uniformly in $(x,t)$, by polynomials times Gaussians in $y$. Hence
the first term is smooth in $\R^d\times(a,b)$ and solves the homogeneous
equation classically. The second term is a bounded distributional solution of
$\partial_tv-\Delta v+\varrho|x|^2v=g$ in $\R^d\times(a,b)$. Since the
coefficients are smooth and $g$ is bounded, interior $L^p$ estimates give
$v\in W^{2,1}_{p,\mathrm{loc}}$ for every $p<\infty$, and the parabolic Sobolev
embedding shows that $v$, and hence $u$, is locally H\"older continuous in
$(x,t)$ \cite{LSU,Lieberman}. As $f$ is locally Lipschitz, $g=f(u)$ is then
locally H\"older continuous as well, and the interior Schauder estimates
\cite{LSU,Lieberman} yield $v\in C^{2+\alpha,1+\alpha/2}_{\mathrm{loc}}$ for
some $\alpha\in(0,1)$. We conclude that $u\in C^{2,1}(Q_T)$ and that $u$
satisfies $\partial_tu-\Delta u+\varrho|x|^2u=f(u)$ pointwise in $Q_T$.
\end{Remark}

We shall repeatedly use the following form of Gronwall's lemma.

\begin{Lemma}\label{lem:gronwall}
Let $0<\tau<\infty$, let $k\in L^1(0,\tau)$ be nonnegative, let $a\ge0$ and let
$y:(0,\tau]\to[0,\infty)$ be bounded and measurable with
$y(t)\le a+\int_0^tk(s)y(s)\dd s$ for a.e.\ $t\in(0,\tau]$. Then
$y(t)\le a\exp\bigl(\int_0^tk(s)\dd s\bigr)$ for a.e.\ $t\in(0,\tau]$.
\end{Lemma}

\begin{proof}
The function $Y(t):=a+\int_0^tky$ is absolutely continuous, and
$Y'=ky\le kY$ a.e. Hence $\bigl(Y(t)e^{-\int_0^tk}\bigr)'\le0$ a.e., so that
$Y(t)\le ae^{\int_0^tk}$, and $y\le Y$ a.e.
\end{proof}

\section{Statement of the main results}\label{sec:results}

Throughout this section, $0<\beta\le1$ and $\varrho>0$. We use the
constants
\begin{equation}\label{eq:constants}
\begin{gathered}
\kappa_{d,\beta}:=p_\beta(0,1)
=(4\pi)^{-\frac d2}\,\frac{\Gamma\bigl(1+\frac{d}{2\beta}\bigr)}
{\Gamma\bigl(1+\frac d2\bigr)},\qquad
\kappa:=\max\{1,\kappa_{d,\beta}\},\\
m_\beta(t):=\min\bigl\{1,\,2^{\frac d2}e^{-\lambda_\beta t}\bigr\}.
\end{gathered}
\end{equation}
The value of $\kappa_{d,\beta}$ is computed in Lemma~\ref{lem:kernel}.

\subsection{Smoothing estimates}

\begin{theorem}[Orlicz smoothing estimate]\label{thm:OrliczSmoothing}
Let $0<\beta\le1$ and $\varrho>0$.
\begin{enumerate}[label=\textup{(\alph*)}]
\item \emph{(Semigroup properties.)}
\begin{enumerate}[label=\textup{(\roman*)}]
\item For every $\Phi\in\Y$, every $t\ge0$ and every $\varphi\in L_\Phi$,
\begin{equation}\label{E5.2}
\|S_\beta(t)\varphi\|_{L_\Phi}\ \le\ m_\beta(t)\,\|\varphi\|_{L_\Phi},
\end{equation}
and $S_\beta(t)M_\Phi\subseteq M_\Phi$.
\item If $\Phi\in\Y$ is finite, then $\|S_\beta(t)\varphi-\varphi\|_{L_\Phi}\to0$
as $t\to0$ for every $\varphi\in L^0_\Phi$.
\item If $\Phi\in\Nf$, then $(S_\beta(t))_{t\ge0}$ is a $C_0$-semigroup of
contractions on $M_\Phi$.
\item The exponential rate in \eqref{E5.2} is sharp: for every $\Phi\in\Y$
and $t\ge0$ the operator norm of $S_\beta(t)$ on $L_\Phi$ is at least
$e^{-\lambda_\beta t}$.
\end{enumerate}
\item \emph{(Smoothing.)} Let $\Phi,\Psi\in\Y$ be compatible in the sense of
\eqref{E3.1}. Then
\begin{equation}\label{eq:OrliczOrlicz}
\|S_\beta(t)\varphi\|_{L_\Psi}
\ \le\ \frac{2\kappa}{b}\,\|\varphi\|_{L_\Phi}\,
\frac{\Phi^{-1}\bigl(t^{-\frac{d}{2\beta}}\bigr)}
     {\Psi^{-1}\bigl(t^{-\frac{d}{2\beta}}\bigr)},
\qquad t>0,\ \varphi\in L_\Phi .
\end{equation}
\item \emph{(Small times.)} If in addition $\Phi$ is finite and
$\Psi^{-1}(s)/\Phi^{-1}(s)\to0$ as $s\to\infty$, then
$$
\lim_{t\to0}\ \frac{\Psi^{-1}\bigl(t^{-\frac{d}{2\beta}}\bigr)}
{\Phi^{-1}\bigl(t^{-\frac{d}{2\beta}}\bigr)}\,\|S_\beta(t)\varphi\|_{L_\Psi}=0
$$
uniformly for $\varphi$ in compact subsets of $L^0_\Phi$.
\end{enumerate}
\end{theorem}

The constant in \eqref{eq:OrliczOrlicz} depends only on $d$, $\beta$ and the
compatibility constant $b$; it is independent of $\Phi$, $\Psi$, $\Theta$ and
$\varrho$. Taking $\Psi=\Phi_2^\infty$ in part (b), that is, $L_\Psi=L^\infty$
(see Remark~\ref{rem:onLpLq}), we obtain the following result.

\begin{Corollary}[$L^\infty$-smoothing]\label{cor:Linfty}
Let $0<\beta\le1$, $\varrho>0$ and $\Phi\in\Y$. Then, with $C_\infty:=2\kappa$,
\begin{equation}\label{eq:MTtoLinfty}
\|S_\beta(t)\varphi\|_{L^\infty}
\ \le\ C_\infty\,\|\varphi\|_{L_\Phi}\,\Phi^{-1}\bigl(t^{-\frac{d}{2\beta}}\bigr),
\qquad t>0,\ \varphi\in L_\Phi .
\end{equation}
\end{Corollary}

\begin{Corollary}\label{cor:Linftyzero}
Let $0<\beta\le1$, $\varrho>0$ and let $\Phi\in\Y$ be finite. Then
$$
\lim_{t\to0}\
\frac{\|S_\beta(t)\varphi\|_{L^\infty}}{\Phi^{-1}\bigl(t^{-\frac{d}{2\beta}}\bigr)}=0
$$
uniformly for $\varphi$ in compact subsets of $L_\Phi^0$.
\end{Corollary}

\begin{Corollary}[Exponential improvement for large times]\label{cor:decay}
Let $0<\beta\le1$, $\varrho>0$, $\Phi,\Psi\in\Y$ compatible with constant $b$,
and let $\lambda_\beta$ be as in \eqref{eq:gap}. For every $\theta\in(0,1)$ and
every $\varphi\in L_\Phi$,
\begin{equation}\label{eq:largetime}
\|S_\beta(t)\varphi\|_{L^\infty}
\ \le\ C_\infty\,2^{\frac d2}\theta^{-\frac{d}{2\beta}}\,
e^{-(1-\theta)\lambda_\beta t}\,
\Phi^{-1}\bigl(t^{-\frac{d}{2\beta}}\bigr)\,\|\varphi\|_{L_\Phi},
\qquad t>0,
\end{equation}
\begin{equation}\label{eq:largetime-sharp}
\|S_\beta(t)\varphi\|_{L^\infty}
\ \le\ C_\infty\,2^{\frac d2}e^{\lambda_\beta}\,\Phi^{-1}(1)\,
e^{-\lambda_\beta t}\,\|\varphi\|_{L_\Phi},\qquad t\ge1,
\end{equation}
and the same estimates hold for $\|S_\beta(t)\varphi\|_{L_\Psi}$ with $C_\infty$
replaced by $2\kappa/b$ and $\Phi^{-1}$ replaced by $\Phi^{-1}/\Psi^{-1}$.
The rate $\lambda_\beta$ in \eqref{eq:largetime-sharp} cannot be improved.
\end{Corollary}

\begin{Remark}\rm\label{rem:onLpLq}
Take $\Phi(x)=x^q/q$ and $\Psi(x)=x^r/r$ with $1\le q\le r<\infty$. Then
$\Phi^{-1}(x)=(qx)^{1/q}$ and $\Psi^{-1}(x)=(rx)^{1/r}$, and the choice
\eqref{eq:Theta-def} below gives
$\Theta^{-1}(x)=r^{1/r}q^{-1/q}x^{1+\frac1r-\frac1q}$. This function is concave
since $0<1+\frac1r-\frac1q\le1$. Thus \eqref{E3.1} holds with $b=1$ and
$\Theta(x)=C(q,r)\,x^{\frac{qr}{qr+q-r}}$, and \eqref{eq:OrliczOrlicz} becomes
\begin{equation}\label{eq:LpLq}
\|S_\beta(t)\varphi\|_{L^r}\ \le\ C\,t^{-\frac{d}{2\beta}
\left(\frac1q-\frac1r\right)}\|\varphi\|_{L^q},\qquad t>0 .
\end{equation}
Choose instead $\Psi=\Phi_2^\infty$, where $\Phi^\infty(x)=0$ for
$x\in[0,1]$, $\Phi^\infty(x)=\infty$ for $x>1$, and
$\Phi^\infty_2(x)=\Phi^\infty(x/2)$. Then $L_\Psi=L^\infty$, with
$\|\cdot\|_{L^\infty}=\frac12\|\cdot\|_{L_\Psi}$ and $\Psi^{-1}\equiv2$
\cite[Example 2.1(a)]{Laister}. With $\Theta=\Phi^*$, condition \eqref{E3.1}
holds with $b=1/2$ by \eqref{eq:24}, and \eqref{eq:OrliczOrlicz} reduces to
\eqref{eq:MTtoLinfty} with the constant $2\kappa$. We shall nevertheless give
a direct proof of Corollary \ref{cor:Linfty}.
\end{Remark}

\begin{proposition}[$L^q$--$L^r$ bounds]\label{prop:LpLq}
Let $0<\beta\le1$, $\varrho>0$ and $1\le q\le r\le\infty$. Then
\eqref{eq:LpLq} holds for all $t>0$ with a constant $C=C(d,\beta,q,r)$, and
for every $\theta\in(0,1)$ its right-hand side may be multiplied by
$C_\theta e^{-(1-\theta)\lambda_\beta t}$. Moreover
$\|S_\beta(t)\varphi\|_{L^r}\le Ce^{-\lambda_\beta t}\|\varphi\|_{L^q}$ for
$t\ge1$.
\end{proposition}

Proposition~\ref{prop:LpLq} recovers, for $0<\beta\le1$, general $\varrho>0$
and $q\le r$, the fixed-time estimates of \cite[Theorem 1.1]{BMMfrac}, where
$\varrho=1$ and the large-time rate is $e^{-td^\beta}=e^{-\lambda_\beta t}$.

\begin{Remark}\rm\label{rem:thetachoice}
Suppose that the function
\begin{equation}\label{eq:Theta-def}
g(x):=\frac{x\,\Psi^{-1}(x)}{\Phi^{-1}(x)},\qquad x>0,\qquad g(0):=0,
\end{equation}
is a concave, increasing bijection of $[0,\infty)$. Then its inverse $\Theta$
is convex, increasing and finite, hence a Young's function with
$\Theta^{-1}=g$, and \eqref{E3.1} holds with $b=1$; compare
\cite[Remark 3.1(ii)]{Laister}.
Concavity is necessary, since the generalised inverse of every Young's
function is concave (Lemma~\ref{lem:concave}). Note that the right-hand
inequality in \eqref{E3.1} is the hypothesis
\begin{equation}\label{eq:compat}
\Theta^{-1}(x)\,\Phi^{-1}(x)\le x\,\Psi^{-1}(x),\qquad x\ge0 ,
\end{equation}
of O'Neil's convolution inequality (Lemma~\ref{lem:oneil}); the choice
$\Theta^{-1}=g$ turns it into an equality.
\end{Remark}

\begin{Example}\label{ex:smoothing}
The following estimates follow from Corollaries~\ref{cor:Linfty} and
\ref{cor:decay} and Theorem~\ref{thm:OrliczSmoothing}(b); they are the
fractional Hermite counterparts of \cite[Examples 3.2--3.5]{Laister}. We put
$a:=d/(2\beta)$.
\begin{enumerate}[label=\textup{(\alph*)}]
\item $\Phi(x)=\max\{x-1,0\}$, so that $L_\Phi=L^1+L^\infty$ and
$\Phi^{-1}(x)=1+x$ \cite[Example 2.1(c)]{Laister}:
$\|S_\beta(t)\varphi\|_{L^\infty}\le C(1+t^{-a})\|\varphi\|_{L^1+L^\infty}$,
and $\le Ce^{-\lambda_\beta t}\|\varphi\|_{L^1+L^\infty}$ for $t\ge1$.
\item $\Phi(x)=\exp\exp(x^q)-e$, $q\ge1$, so that
$\Phi^{-1}(y)=\bigl(\log\log(e+y)\bigr)^{1/q}$. Then, for $t>0$,
$\|S_\beta(t)\varphi\|_{L^\infty}\le C\bigl(\log\log(e+t^{-a})\bigr)^{1/q}
\|\varphi\|_{L_\Phi}$.
\item $\Phi(x)=e^{x^q}-1$ and $\Psi(x)=e^{x^r}-1$ with $1\le q\le r$: by
\cite[Example 3.4]{Laister}, the function
$\Theta^{-1}(x)=x(\log(1+x))^{\frac1r-\frac1q}$ is concave, increasing and
invertible, and \eqref{eq:OrliczOrlicz} gives
$\|S_\beta(t)\varphi\|_{\exp L^r}\le C\bigl(\log(1+t^{-a})\bigr)^{\frac1q-\frac1r}
\|\varphi\|_{\exp L^q}$.
\item $\Phi(x)=x^q/q$ and $\Psi(x)=e^{x^r}-1$ with $1\le q\le r$: by
\cite[Example 3.5]{Laister} the function
$\Theta^{-1}(x)=C(q)x^{1-\frac1q}(\log(1+x))^{\frac1r}$ is admissible, and
$\|S_\beta(t)\varphi\|_{\exp L^r}\le Ct^{-\frac{a}{q}}
\bigl(\log(1+t^{-a})\bigr)^{-\frac1r}\|\varphi\|_{L^q}$.
\end{enumerate}
In (b)--(d), the right-hand sides may be multiplied by
$C_\theta e^{-(1-\theta)\lambda_\beta t}$. The Young's functions of (c) and (d)
do not depend on the operator, so the compatibility conditions are exactly
those verified in \cite{Laister}. The operator enters only through the
constant $2\kappa/b$ and the exponent $a$.
\end{Example}

\begin{figure}[ht]
\centering
\begin{tikzpicture}
\begin{axis}[
  width=0.82\textwidth, height=6.2cm,
  xmin=0.02, xmax=6, ymin=1e-3, ymax=10,
  ymode=log, xlabel={$t$}, ylabel={upper bound},
  legend style={at={(0.03,0.03)}, anchor=south west, font=\small},
  samples=300, domain=0.02:6]
\addplot[black, thick, dashed] {x^(-1/2)};
\addlegendentry{free: $t^{-d/(2\beta q)}$}
\addplot[black, very thick] {min(x^(-1/2), sqrt(2)*exp(-(x-1)))};
\addlegendentry{confined: $\min\{t^{-d/(2\beta q)},\,2^{d/2}e^{-\lambda_\beta(t-1)}\}$}
\end{axis}
\end{tikzpicture}
\caption{The $L^q\to L^\infty$ bounds of Proposition~\ref{prop:LpLq} (up to
the multiplicative constant), for $d=1$, $q=1$, $\beta=1$, $\varrho=1$, so
that $d/(2\beta q)=1/2$ and $\lambda_\beta=1$. For short times the confined
bound follows the self-similar free bound; for large times it decays
exponentially at the sharp rate $\lambda_\beta$, because of the spectral gap
of $\Hr$.}
\label{fig:decay}
\end{figure}
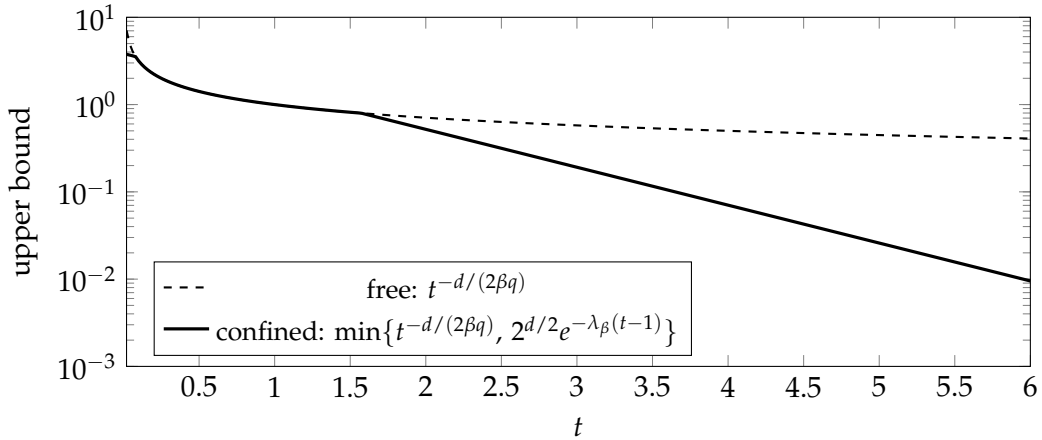

\subsection{Local and global well-posedness}

Throughout this subsection, $f$ is assumed to satisfy
\begin{equation*}
(\mathrm{L})\qquad
f:\R\to\R\ \text{is locally Lipschitz continuous and } f(0)=0 ,
\end{equation*}
and $\ell:(0,\infty)\to[0,\infty)$ denotes its increasing Lipschitz modulus,
\begin{equation}\label{eq:lipschitzmodulus}
\ell(s):=\sup_{\substack{|u|,|v|\le s\\ u\ne v}}\frac{|f(u)-f(v)|}{|u-v|},
\end{equation}
so that
\begin{equation}\label{eq:lipbound}
|f(u)-f(v)|\le \ell(s)|u-v|\quad\text{for }|u|,|v|\le s,
\qquad |f(u)|\le\ell(s)|u|\quad\text{for }|u|\le s .
\end{equation}
We call $f$ \emph{positone} if $uf(u)\ge0$ for all $u\in\R$. We also use the
change of variables $x=s^{-d/(2\beta)}$, which gives, for $0<T\le\infty$,
\begin{equation}\label{eq:changevar}
\int_0^{T}\ell\bigl(\lambda\Phi^{-1}(s^{-\frac{d}{2\beta}})\bigr)\dd s
=\frac{2\beta}{d}\int_{T^{-\frac{d}{2\beta}}}^{\infty}x^{-p_\ast}
\ell\bigl(\lambda\Phi^{-1}(x)\bigr)\dd x .
\end{equation}

\begin{theorem}[Local well-posedness]\label{thm:local}
Let $0<\beta\le1$, $\varrho>0$, $\Phi\in\Nf$ and let $f$ satisfy
$(\mathrm{L})$. Suppose that \eqref{eq:IC-infty} holds for some $\lambda>0$,
where $p_\ast=1+2\beta/d$. Then the following hold.
\begin{enumerate}[label=\textup{(\alph*)}]
\item \emph{(Uniform existence.)} For every compact subset $\mathcal K$ of
$M_\Phi$ there is $T_{\mathcal K}>0$ such that for every $u_0\in\mathcal K$
problem \eqref{eq:main} has a bounded $M_\Phi$-mild solution $u(\cdot;u_0)$ on
$[0,T_{\mathcal K})$ satisfying $|u(t;u_0)|\le 2S_\beta(t)|u_0|$ and
\begin{equation}\label{eq:smallrate}
\lim_{t\to0}\
\frac{\|u(t;u_0)\|_{L^\infty}}{\Phi^{-1}\bigl(t^{-\frac{d}{2\beta}}\bigr)}=0
\end{equation}
uniformly for $u_0\in\mathcal K$. If $f$ is positone and $u_0\ge0$ (resp.
$u_0\le0$) then $u(t;u_0)\ge0$ (resp. $u(t;u_0)\le0$).
\item \emph{(Uniqueness.)} For every $T>0$, a bounded $M_\Phi$-mild solution of
\eqref{eq:main} on $[0,T)$ with data $u_0\in M_\Phi$ is unique. Consequently
$u(\cdot;u_0)$ extends to a maximal interval $[0,T_{u_0})$ as a bounded
$M_\Phi$-mild solution and, if $T_{u_0}<\infty$, then
$\|u(t)\|_{L^\infty}\to\infty$ as $t\uparrow T_{u_0}$.
\item \emph{(Continuous dependence.)} For every compact $\mathcal K\subseteq
M_\Phi$ there exist $C_{\mathcal K}>0$ and
$\tau_{\mathcal K}\in(0,T_{\mathcal K})$ such that for all
$\phi,\varphi\in\mathcal K$ and $t\in(0,\tau_{\mathcal K}]$,
\begin{equation}\label{eq:contdep}
\|u(t;\phi)-u(t;\varphi)\|_{L_\Phi}
+\frac{\|u(t;\phi)-u(t;\varphi)\|_{L^\infty}}
      {\Phi^{-1}\bigl(t^{-\frac{d}{2\beta}}\bigr)}
\ \le\ C_{\mathcal K}\,\|\phi-\varphi\|_{L_\Phi}.
\end{equation}
\item \emph{(Comparison principle.)} If $f$ is increasing and
$\phi\le\varphi$ in $M_\Phi$ then $u(t;\phi)\le u(t;\varphi)$ on
$[0,\min\{T_\phi,T_\varphi\})$.
\end{enumerate}
\end{theorem}

In (b), uniqueness is unconditional within the class of bounded
$M_\Phi$-mild solutions: no smallness condition such as \eqref{eq:smallrate}
is imposed on the competitor. This is the counterpart of
\cite[Theorem B(b)]{Laister}. As in \cite[Remark 3.3]{Laister}, parts (a)--(c)
define a local semiflow on $M_\Phi$. This semiflow is positone if $f$ is
positone, and order preserving if $f$ is increasing.

\begin{Corollary}[Global well-posedness for small data]\label{thm:global}
Let $0<\beta\le1$, $\varrho>0$, $\Phi\in\Nf$ and let $f$ satisfy $(\mathrm L)$.
Suppose that \eqref{eq:IC-zero} holds for some $\lambda>0$, that is,
$$
I_0:=\frac{2\beta}{d}\int_0^{\infty}x^{-p_\ast}
\ell\bigl(\lambda\Phi^{-1}(x)\bigr)\dd x<1,
\qquad\text{and set}\qquad A:=(1-I_0)^{-1}\in[1,\infty).
$$
Then for every $u_0\in M_\Phi$ with $\|u_0\|_{L_\Phi}\le\lambda/(AC_\infty)$,
problem \eqref{eq:main} has a unique global bounded $M_\Phi$-mild solution
$u$; it satisfies $|u(t)|\le A\,S_\beta(t)|u_0|$ for all $t>0$, and
\begin{align}
\|u(t)\|_{L^\infty}&\le\lambda\,\Phi^{-1}\bigl(t^{-\frac{d}{2\beta}}\bigr)
\xrightarrow[t\to\infty]{}0,\qquad t>0,\label{eq:globaldecay1}\\
\|u(t)\|_{L_\Phi}&\le A\,m_\beta(t)\,\|u_0\|_{L_\Phi},\qquad t\ge0,
\label{eq:globaldecay2}\\
\|u(t)\|_{L^\infty}&\le C_\infty 2^{\frac d2}e^{\lambda_\beta}\Phi^{-1}(1)\,A\,
e^{-\lambda_\beta t}\,\|u_0\|_{L_\Phi},\qquad t\ge1 .\label{eq:expdecay}
\end{align}
\end{Corollary}

The confining potential allows a much more general statement, in which the
behaviour of $\Phi$ near the origin plays no role and $f$ is only required to
satisfy a Dini condition at $0$.

\begin{theorem}[Global well-posedness under confinement]\label{thm:confined}
Let $0<\beta\le1$, $\varrho>0$, $\Phi\in\Nf$, and let $f$ satisfy
$(\mathrm L)$ and \eqref{eq:IC-infty} for some $\lambda>0$. Assume moreover
the Dini condition
\begin{equation}\label{eq:dini}
\int_0^1\frac{\ell(r)}{r}\dd r<\infty .
\end{equation}
Then there exists $\delta>0$ such that for every $u_0\in M_\Phi$ with
$\|u_0\|_{L_\Phi}\le\delta$ problem \eqref{eq:main} has a unique global bounded
$M_\Phi$-mild solution $u$. It satisfies $|u(t)|\le2S_\beta(t)|u_0|$ for all
$t>0$, and
\begin{gather*}
\|u(t)\|_{L_\Phi}\le2m_\beta(t)\|u_0\|_{L_\Phi}\quad(t\ge0),\\
\|u(t)\|_{L^\infty}\le 2C_\infty2^{\frac d2}e^{\lambda_\beta}\Phi^{-1}(1)\,
e^{-\lambda_\beta t}\|u_0\|_{L_\Phi}\quad(t\ge1).
\end{gather*}
\end{theorem}

\begin{Remark}\rm[No Fujita phenomenon under confinement]\label{rem:nofujita}
Let $f(u)=|u|^{m-1}u$ with $m>1$, so that $\ell(s)=ms^{m-1}$ and
\eqref{eq:dini} holds, and let $\Phi(x)=x^q/q$ with $q>\max\{1,q_c\}$, where
$q_c:=d(m-1)/(2\beta)$; then \eqref{eq:IC-infty} holds
(Section~\ref{subsec:power}). Theorem~\ref{thm:confined} thus gives global
solutions with exponential decay for small data in $L^q$, for \emph{every}
$m>1$. For the free equation, this fails when $1<m\le p_\ast$, since every
nonnegative nontrivial solution then blows up in finite time
\cite{Fujita1966,Hayakawa1973,KST1977,Sugitani1975}. The spectral gap of
$\Hr^\beta$ thus removes the Fujita phenomenon, as a bounded domain with
Dirichlet boundary conditions does. Moreover, the condition near $x=0$ in
\eqref{eq:IC-zero}, which encodes the free decay rate
$\Phi^{-1}(t^{-d/(2\beta)})$, becomes superfluous. This is why
Corollary~\ref{thm:global} is a genuinely ``free'' statement, while
Theorem~\ref{thm:confined} is specific to $\varrho>0$.
\end{Remark}

\begin{Remark}\rm[Why $M_\Phi$?]\label{rem:whyMT}
For a general Young's function $\Phi$, the space $C_0^\infty$ need not be
dense in $L_\Phi$, and $S_\beta(t)$ need not be strongly continuous on
$L_\Phi$ at $t=0$. The standard example is $\Phi(u)=e^{u^2}-1$, see
\cite[Section 2]{ioku2015MPAG} and \cite[Section 1.1]{Laister}. The
Morse--Transue subspace restores the $C_0$-semigroup property, and for
$N$-functions it coincides with $L^0_\Phi$. If $\Phi\in\Nf$ satisfies in
addition the $\Delta_2$-condition, then $M_\Phi=L_\Phi$ and all the results
above hold in $L_\Phi$ itself.
\end{Remark}

\begin{Remark}\rm[On the constants]\label{rem:constants}
The constant $2\beta/d$ in Corollary~\ref{thm:global} comes from
\eqref{eq:changevar}; it reduces to $2/n$ when $\beta=1$, $d=n$, in agreement
with \cite[(3.9)]{Laister}. The value $A=(1-I_0)^{-1}$ is the smallest number
for which the barrier $w=AS_\beta(t)|u_0|$ is invariant under the Duhamel
map, as the proof in Section~\ref{sec:proofB} shows.
\end{Remark}

\begin{Remark}\rm\label{rem:sharper}
Since $S_\beta(t)$ is a contraction on $L_\Phi$, rather than merely bounded
by $2$ as in \cite[Theorem A(a)(i)]{Laister}, the contraction constant of the
Duhamel map in the global argument is $I_0$, and not $2I_0$. This is why the
threshold in \eqref{eq:IC-zero} is $d/(2\beta)$, whereas the argument of
\cite[Section 5.5]{Laister} leads to $n/4$ for $\beta=1$. Since the Gaussian
semigroup is likewise a contraction on every Orlicz space
(Lemma~\ref{L5.1}), the same observation improves the threshold $n/4$ of
\cite[Corollary B]{Laister} to $n/2$ also for $\varrho=0$.
\end{Remark}

\subsection{Critical Young's functions}

We now restrict attention to positone convex nonlinearities. The hypotheses
below are those of \cite[Section 3.3]{Laister}.
\begin{equation*}
(\mathrm{C})\qquad
f\in C^1(\R)\ \text{is positone, positive on }(0,\infty)\ \text{and convex on }
[0,\infty).
\end{equation*}
\begin{equation*}
(\Lambda)\qquad
\text{there exist }\kappa_0,C>0\ \text{such that }\ell(u)\le Cf'(u)\
\text{for all } u>\kappa_0 ;
\end{equation*}
when the inequality in $(\Lambda)$ holds for all $u>0$, we write
$(\Lambda_0)$. Every $f$ satisfying $(\mathrm C)$ satisfies $(\mathrm L)$
\cite[Remark 3.5(a)]{Laister} and is increasing on $[0,\infty)$, since
$f(\theta u)\le\theta f(u)\le f(u)$ for $\theta\in[0,1]$, $u\ge0$. If $f$ is odd
and satisfies $(\mathrm C)$, then it satisfies $(\Lambda_0)$
\cite[Remark 3.5(b)]{Laister}. With
\begin{equation}\label{E 1.2}
F(u):=\int_u^{\infty}\frac{\dd s}{f(s)},\qquad u>0 ,
\end{equation}
we impose
\begin{equation*}
(\mathrm{G})\qquad
\begin{array}{l}
F(u)<\infty\ \text{for } u>0,\ (\log f)'\ \text{is regularly varying of index }
\rho-1\\[2pt]
\text{with }\rho\in(0,\infty],\ \text{and }
\lim_{u\to\infty}f'(u)F(u)\ \text{exists and is finite}.
\end{array}
\end{equation*}

\begin{Remark}\rm\label{rem:onG}
Hypothesis $(\mathrm G)$ encodes the class of nonlinearities of the form
\begin{equation}\label{E7.1}
f(u)=e^{u^{\rho}\mu(u)} ,
\end{equation}
with $\mu$ slowly varying, in which case $\log f$ is regularly varying of index
$\rho$; the value $\rho=\infty$ is admitted and corresponds to $\log f$ rapidly
varying, that is, to a nonlinearity of unrestricted growth rate
\cite[Lemma 6.1(i)]{Laister}. By \cite[Lemma 6.1(iv)]{Laister} the limit in
$(\mathrm G)$ is necessarily equal to $1$:
\begin{equation}\label{eq:limitone}
\lim_{u\to\infty}f'(u)F(u)=1 .
\end{equation}
Under $(\mathrm C)$ one has $f(s)\le sf(1)$ for $s\in[0,1]$, so that
$\int_0^1\dd s/f(s)=\infty$; hence $F$ is a decreasing $C^1$-bijection of
$(0,\infty)$ onto itself. Since $F'=-1/f$, its inverse satisfies
\begin{equation}\label{E7.3}
\bigl(F^{-1}\bigr)'(\sigma)=-f\bigl(F^{-1}(\sigma)\bigr),
\qquad \sigma>0 ,
\end{equation}
and \eqref{eq:limitone} may be rewritten as
\begin{equation}\label{E7.2}
\lim_{\sigma\to0^+}\sigma\,f'\bigl(F^{-1}(\sigma)\bigr)=1 .
\end{equation}
Identity \eqref{E7.2} is the form of $(\mathrm G)$ that we shall use in
Section~\ref{sec:nonexistence}. It is more precise than the two-sided bounds
$C^{-1}\sigma^{\varepsilon-1}\le f(F^{-1}(\sigma))\le C\sigma^{-\varepsilon-1}$
of \cite[Lemma 3.1]{Fujishima2024} (with $q_f=1$ in the notation there).
\end{Remark}

\begin{theorem}[Critical Young's functions]\label{thm:critical}
Let $0<\beta\le1$, $\varrho>0$, and let $f$ satisfy $(\mathrm C)$ and
$(\mathrm G)$.
\begin{enumerate}[label=\textup{(\alph*)}]
\item \emph{(Local well-posedness.)}
\begin{enumerate}[label=\textup{(\roman*)}]
\item If $f$ satisfies $(\Lambda)$, $\Phi\in\Nf$ and $\Phi\gtrsim f$ at
infinity, then \eqref{eq:IC-infty} holds for some $\lambda>0$, and hence the
conclusions of Theorem~\ref{thm:local} hold.
\item If, in addition to the hypotheses of \textup{(i)}, there is $C>0$ with
$|f(u)|\le Cf(|u|)$ for all $u\in\R$, and $\Phi\gtrsim f^r$ near $0$ for some
$r\ge1$ with $r>d/(2\beta)$, then every $M_\Phi$-mild solution of
\eqref{eq:main} on $[0,T)$ with data $u_0\in M_\Phi$ is a bounded
$M_\Phi$-mild solution; in particular it coincides with $u(\cdot;u_0)$, and
$T\le T_{u_0}$.
\end{enumerate}
\item \emph{(Nonexistence.)} If $\Phi\in\Y$ and $\Phi\lesssim f$ at infinity,
then there exists a nonnegative $u_0\in L_\Phi$ with compact support such that
\eqref{eq:main} has no nonnegative integral solution in $Q_T$ for any $T>0$.
\item \emph{(Global well-posedness.)} Suppose $f$ satisfies $(\Lambda_0)$ and
$f'(0)=0$, that $\Phi\in\Nf$ and $\Phi\gtrsim f$ at infinity, and that there
exists $\lambda_0>0$ with
\begin{equation}\label{eq:C-b}
\int_0^{1}x^{-p_\ast}f'\bigl(\lambda_0\Phi^{-1}(x)\bigr)\dd x<\infty .
\end{equation}
Then the conclusions of Theorem \ref{thm:local} and Corollary
\ref{thm:global} hold for some $\lambda\in(0,\lambda_0]$.
\end{enumerate}
\end{theorem}

Set $\Phi_c(u):=f(u)$ for $u\ge0$; this function belongs to $\Y$ by
$(\mathrm C)$. Parts (a)(i) and (b) say that $\Phi\gtrsim\Phi_c$ at infinity
suffices for local well-posedness in $M_\Phi$, while $\Phi\lesssim\Phi_c$ at
infinity already forces nonexistence of nonnegative solutions for some data
in $L_\Phi$. In this sense, $\Phi_c$ is a critical Young's function for
\eqref{eq:main}. It is asymptotically unique modulo $\approx$. The behaviour of
$\Phi$ near the origin, however, may be chosen freely; it affects the space
$L_\Phi$ but not the local theory \cite[Section 3.3]{Laister}. For $f$
Lipschitz near $0$, the
condition $\Phi\gtrsim f^r$ near $0$ in (a)(ii) holds for every
$\Phi$ with $\liminf_{u\to0}\Phi(u)/u^D>0$ for some $D\ge1$, upon choosing
$r>\max\{D,1,d/(2\beta)\}$; compare \cite[Remark 3.7(b)]{Laister}.
Figure~\ref{fig:critical} illustrates the dichotomy for $f(u)\approx e^{u^2}$.

\begin{figure}[ht]
\centering
\begin{tikzpicture}
\begin{axis}[
  width=0.78\textwidth, height=6.4cm,
  xmin=-1, xmax=1.6, ymin=-3, ymax=4.5,
  xlabel={$\log u$}, ylabel={$\log\log\Phi(u)$},
  xtick={-1,-0.5,0,0.5,1,1.5},
  legend style={at={(0.97,0.03)}, anchor=south east, font=\small},
  samples=2, domain=-1:1.6]
\addplot[black, thick, dashed] {x};
\addlegendentry{$\Phi=e^{u}$: $\Phi\lesssim f$, nonexistence in $L_\Phi$}
\addplot[black, very thick] {2*x};
\addlegendentry{$\Phi=e^{u^2}\approx f$: critical}
\addplot[black, thick, dotted] {3*x};
\addlegendentry{$\Phi=e^{u^3}$: $\Phi\gtrsim f$, LWP in $M_\Phi$}
\end{axis}
\end{tikzpicture}
\caption{Schematic illustration of Theorem~\ref{thm:critical} for $f$ odd with
$f(u)=u^m e^{u^2}$, $u\ge0$ ($\rho=2$), and $\Phi=e^{u^p}$ at infinity (only the
behaviour at infinity matters). In the coordinates $(\log u,\log\log\Phi)$ the
three Young's functions are lines of slopes $p=1,2,3$; dilations $u\mapsto Ku$
translate them vertically by $p\log K$ and do not change the order relation
$\lesssim$. For $p<2$ nonexistence occurs for some nonnegative data in
$L_\Phi$; for $p\ge2$ the problem is locally well posed in $M_\Phi$; for $p=2$
both statements hold, and the initial data of part (b) then necessarily lies in
$L_\Phi\setminus M_\Phi$ (Remark~\ref{rem:notinM}).}
\label{fig:critical}
\end{figure}
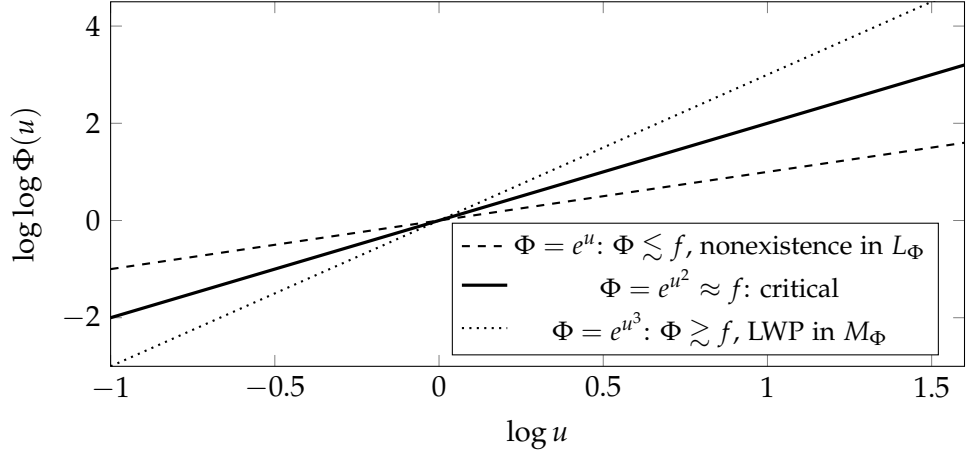

\begin{Corollary}\label{cor:C}
Let $f$ satisfy $(\mathrm C)$, $(\Lambda_0)$ and $(\mathrm G)$, and suppose
that $f'$ is regularly varying at zero of index $m-1$ for some $m>p_\ast$. If
$\Phi\in\Nf$ satisfies $\Phi(u)\gtrsim u^{q}$ near $0$ for some
$0<q<\frac{d(m-1)}{2\beta}$ and $\Phi\gtrsim f$ at infinity, then the
conclusions of Theorem \ref{thm:local} and Corollary \ref{thm:global} hold for
some $\lambda>0$.
\end{Corollary}

Since $\Phi\in\Nf$ and $\Phi(u)\gtrsim u^q$ near $0$ force $q>1$, the
hypothesis $m>p_\ast$ is exactly what makes the range of $q$ nonempty. Under
confinement, these restrictions disappear.

\begin{Corollary}\label{cor:Cconf}
Let $f$ satisfy $(\mathrm C)$, $(\Lambda_0)$ and $(\mathrm G)$, and assume that
$\int_0^1f'(r)\,\dd r/r<\infty$ (for instance, $f'$ regularly varying at zero
of some index $m-1>0$). If $\Phi\in\Nf$ and $\Phi\gtrsim f$ at infinity, then
the conclusions of Theorem~\ref{thm:local} and Theorem~\ref{thm:confined}
hold.
\end{Corollary}

\section{The fractional Hermite kernel}\label{sec:kernel}

This section collects the properties of $\Kb^\beta_t$ on which everything
else rests. All of them follow from the explicit formula \eqref{mc} and from
subordination.

\begin{Lemma}[Kernel bounds]\label{lem:kernel}
Let $0<\beta\le1$, $\varrho>0$ and $\omega=\sqrt\varrho$. Then for all
$t>0$ and $x,y\in\R^d$:
\begin{enumerate}[label=\textup{(\roman*)}]
\item $\Kb^\beta_t(x,y)=\Kb^\beta_t(y,x)>0$, and for fixed $y$ the map
$(x,t)\mapsto\Kb^\beta_t(x,y)$ is continuous on $\R^d\times(0,\infty)$;
\item $\Kb^\beta_t(x,y)\le p_\beta(x-y,t)$, and consequently
\begin{equation}\label{eq:supkernel}
\sup_{x,y\in\R^d}\Kb_t^\beta(x,y)\ \le\ \kappa_{d,\beta}\,t^{-\frac{d}{2\beta}},
\qquad \kappa_{d,\beta}=p_\beta(0,1)
=(4\pi)^{-\frac d2}\frac{\Gamma(1+\frac{d}{2\beta})}{\Gamma(1+\frac d2)};
\end{equation}
\item one has the exact identity
\begin{equation}\label{eq:mass1}
\int_{\R^d}\Kb_t(x,y)\dd x
=\bigl(\cosh(2\omega t)\bigr)^{-d/2}
\exp\Bigl(-\frac{\omega\tanh(2\omega t)}{2}|y|^2\Bigr),
\end{equation}
and therefore
\begin{equation}\label{eq:massbeta}
\int_{\R^d}\Kb^\beta_t(x,y)\dd x=\int_{\R^d}\Kb^\beta_t(y,x)\dd x
\ \le\ m_\beta(t)=\min\bigl\{1,\ 2^{d/2}e^{-\lambda_\beta t}\bigr\};
\end{equation}
\item $S_\beta(t)h_0=e^{-\lambda_\beta t}h_0$ pointwise, where
$h_0(x)=e^{-\omega|x|^2/2}$;
\item if $\varphi\in L^1$, then $(x,t)\mapsto[S_\beta(t)\varphi](x)$ is
continuous on $\R^d\times(0,\infty)$.
\end{enumerate}
\end{Lemma}

\begin{proof}
(i) Positivity, symmetry and continuity of $\Kb_t$ are clear from \eqref{mc}.
For $\beta<1$, positivity and symmetry of $\Kb^\beta_t$ follow from
\eqref{E7.5} and $\eta_t^\beta\ge0$. By \eqref{eq:scaling-eta},
$\Kb^\beta_t(x,y)=\int_0^\infty\Kb_{t^{1/\beta}\sigma}(x,y)\eta^\beta_1(\sigma)\dd\sigma$;
the integrand is continuous in $(x,t)$ and, by (ii) below for $\beta=1$,
bounded by $(4\pi t^{1/\beta}\sigma)^{-d/2}$, which is integrable against
$\eta^\beta_1$ uniformly for $t$ in compact subsets of $(0,\infty)$, by
Lemma~\ref{lem:stable}(i). The continuity then follows from dominated
convergence.

(ii) We first treat $\beta=1$. Comparing the prefactors in \eqref{mc} and
\eqref{eq:freekernel}, we have $\frac{\omega}{2\pi\sinh(2\omega t)}\le
\frac{1}{4\pi t}$ because $\sinh z\ge z$ for $z\ge0$; and, discarding the
nonnegative term $|x+y|^2\tanh(\omega t)$ in the exponent and using $z\coth
z\ge1$ for $z>0$ with $z=\omega t$, we obtain $\frac{\omega}{4}\coth(\omega
t)\ge\frac{1}{4t}$. Hence $\Kb_t(x,y)\le G(x-y,t)$ for every $t>0$. (This is
also a consequence of the Feynman--Kac formula, since the potential
$\varrho|x|^2$ is nonnegative \cite{Simon1982}.) Integrating against
$\eta^\beta_t$ and recalling \eqref{E7.5} and \eqref{eq:freekernel}, we obtain
$\Kb_t^\beta(x,y)\le p_\beta(x-y,t)$. By \eqref{eq:scaling-eta} and the
self-similarity of the Gaussian,
$p_\beta(z,t)=t^{-d/(2\beta)}p_\beta\bigl(t^{-1/(2\beta)}z,1\bigr)$, and since
$G(\cdot,s)$ is radially decreasing, $\sup_z p_\beta(z,t)=p_\beta(0,t)=
t^{-d/(2\beta)}p_\beta(0,1)$. Finally
$p_\beta(0,1)=(4\pi)^{-d/2}\int_0^\infty s^{-d/2}\eta^\beta_1(s)\dd s$, whose value
is given by Lemma~\ref{lem:stable}(i) with $q=d/2$.

(iii) Consider first $d=1$ and $\omega=1$, and write $a:=\tanh t$,
$c:=\coth t$ and $s:=\sinh 2t$. Since $a+c=2\cosh 2t/s$ and $c-a=2/s$, the
exponent in \eqref{mc} equals
$-\frac a4\bigl[(x+y)^2+(x-y)^2\bigr]-\frac{c-a}{4}(x-y)^2
=-\frac a2(x^2+y^2)-\frac{(x-y)^2}{2s}$, so that
$$
\Kb_t(x,y)=(2\pi s)^{-1/2}
\exp\Bigl(-\frac{a}{2}\bigl(x^2+y^2\bigr)-\frac{(x-y)^2}{2s}\Bigr).
$$
Integrating in $x$, the exponent is
$-\mathsf A x^2+\mathsf Bx-\tfrac{a}{2}y^2-\tfrac{y^2}{2s}$ with
$\mathsf A=\tfrac a2+\tfrac1{2s}=\tfrac{\cosh 2t}{2s}$ (because
$\tanh t=(\cosh 2t-1)/\sinh 2t$) and $\mathsf B=y/s$, whence
\begin{align*}
\int_\R\Kb_t(x,y)\dd x
&=(2\pi s)^{-1/2}\sqrt{\frac{\pi}{\mathsf A}}\,
\exp\Bigl(\frac{\mathsf B^2}{4\mathsf A}-\frac a2 y^2-\frac{y^2}{2s}\Bigr)\\
&=(\cosh 2t)^{-1/2}\exp\Bigl(-y^2\Bigl[\frac{\cosh 2t}{2s}
-\frac{1}{2s\cosh 2t}\Bigr]\Bigr).
\end{align*}
The bracket equals $\frac{\cosh^2 2t-1}{2s\cosh 2t}=\frac{\tanh 2t}{2}$, which
is \eqref{eq:mass1} for $d=1$, $\omega=1$. The kernel \eqref{mc} is a product
of one-dimensional kernels, which gives the case $d\ge1$, $\omega=1$. For
general $\omega$, the change of variables $x=\omega^{-1/2}z$ transforms
$\Hr$ into $\omega\mathcal H_1$, so that
$\Kb_t(x,y)=\omega^{d/2}\Kb^{(1)}_{\omega t}(\omega^{1/2}x,\omega^{1/2}y)$, where
$\Kb^{(1)}$ is the kernel for $\omega=1$ (this can also be read off
directly from \eqref{mc}); integrating in $x$ gives \eqref{eq:mass1}.

Since $\cosh z\ge\max\{1,e^{z}/2\}$, we infer from \eqref{eq:mass1} that
$\int\Kb_t(x,y)\dd x\le\min\{1,2^{d/2}e^{-d\omega t}\}$. Integrating against
$\eta_t^\beta$ and using \eqref{p1} with $\sigma=d\omega$, we get
$$
\int_{\R^d}\Kb^\beta_t(x,y)\dd x
\le\min\Bigl\{1,\ 2^{d/2}\int_0^\infty e^{-d\omega s}\eta^\beta_t(s)\dd s\Bigr\}
=\min\bigl\{1,\ 2^{d/2}e^{-t(d\omega)^{\beta}}\bigr\},
$$
which is \eqref{eq:massbeta}; the equality of the two integrals in
\eqref{eq:massbeta} is the symmetry from (i).

(iv) A direct computation gives $\Hr h_0=d\omega h_0$, so that
$e^{-s\Hr}h_0=e^{-sd\omega}h_0$ in $L^2$. Both sides are continuous functions
(the left-hand side by (v) with $\beta=1$), and hence the identity holds
pointwise. Integrating against $\eta^\beta_t$ and using \eqref{p1}, we obtain
the claim.

(v) By (i) and (ii), for $t\ge t_0>0$ the integrand
$\Kb^\beta_t(x,y)\varphi(y)$ is continuous in $(x,t)$ and bounded by
$\kappa_{d,\beta}t_0^{-d/(2\beta)}|\varphi(y)|$, and the conclusion follows
from dominated convergence.
\end{proof}

\begin{Remark}\rm\label{rem:sharpgap}
The rate $\lambda_\beta=(d\sqrt\varrho\,)^\beta$ in \eqref{eq:massbeta} is
sharp, since it is the ground state energy of $\Hr^\beta$; see \eqref{eq:gap}
and Lemma~\ref{lem:kernel}(iv). This is the only point at which the confining
potential is used quantitatively; every other estimate in this section is
inherited from the free kernel $p_\beta$ through Lemma~\ref{lem:kernel}(ii).
\end{Remark}

The next lemma replaces \cite[Proposition 4.1]{Laister}. It requires no
structure whatsoever on the kernel, and it applies to every Young's function,
finite or not.

\begin{Lemma}[Orlicz norm of a bounded integrable kernel]\label{lem:orlicznorm}
Let $\Theta\in\Y$ and let $k:\R^d\to[0,\infty)$ be measurable with
$\|k\|_{L^1}\le1$ and $\|k\|_{L^\infty}\le M$ for some $M\in(0,\infty)$. Then
\begin{equation}\label{eq:kernel-Theta-raw}
\|k\|_{L_\Theta}\ \le\ \frac{M}{\Theta^{-1}(M)} .
\end{equation}
\end{Lemma}

\begin{proof}
The value $\Theta^{-1}(M)$ is finite and strictly positive, so that
$\mu:=M/\Theta^{-1}(M)\in(0,\infty)$, and $\Theta(\Theta^{-1}(M))\le M<\infty$
by \eqref{eq:22}; in particular $\Theta$ is finite and convex on
$[0,\Theta^{-1}(M)]$. For almost every $x$ we have $k(x)/M\in[0,1]$ and
$$
\Theta\Bigl(\frac{k(x)}{\mu}\Bigr)
=\Theta\Bigl(\frac{k(x)}{M}\,\Theta^{-1}(M)\Bigr)
\le\frac{k(x)}{M}\,\Theta\bigl(\Theta^{-1}(M)\bigr)\le k(x),
$$
where the first inequality uses the convexity of $\Theta$ on
$[0,\Theta^{-1}(M)]$ together with $\Theta(0)=0$. Integrating over $\R^d$,
we get $\int_{\R^d}\Theta(k/\mu)\dd x\le\|k\|_{L^1}\le1$, and
\eqref{eq:kernel-Theta-raw} follows from the definition of the Luxemburg norm.
\end{proof}

For $\Theta(x)=x^r/r$ the bound \eqref{eq:kernel-Theta-raw} is the elementary
interpolation inequality $\|k\|_{L^r}\le\|k\|_{L^1}^{1/r}\|k\|_{L^\infty}^{1-1/r}$
(up to the normalisation of the Luxemburg norm).

\begin{Corollary}\label{cor:kernelTheta}
Let $0<\beta\le1$ and $\Theta\in\Y$. Then
\begin{equation}\label{eq:kernel-Theta}
\bigl\|p_\beta(\cdot,t)\bigr\|_{L_\Theta}
\ \le\ \kappa\,\frac{t^{-\frac{d}{2\beta}}}
{\Theta^{-1}\bigl(t^{-\frac{d}{2\beta}}\bigr)},
\qquad t>0 ,
\end{equation}
with $\kappa=\max\{1,\kappa_{d,\beta}\}$, and the same bound holds for
$\sup_{y}\|\Kb^\beta_t(\cdot,y)\|_{L_\Theta}$ and
$\sup_{x}\|\Kb^\beta_t(x,\cdot)\|_{L_\Theta}$. In particular,
\begin{equation}\label{eq:row-Lr}
\sup_{y\in\R^d}\bigl\|\Kb^\beta_t(\cdot,y)\bigr\|_{L^{r}}
\ \le\ \kappa_{d,\beta}^{1-\frac1r}\,t^{-\frac{d}{2\beta}\left(1-\frac1r\right)},
\qquad t>0,\ r\in[1,\infty].
\end{equation}
\end{Corollary}

\begin{proof}
Put $M_t:=\kappa\,t^{-d/(2\beta)}$. Then $\|p_\beta(\cdot,t)\|_{L^1}=1$ and,
by Lemma \ref{lem:kernel}(ii), $\|p_\beta(\cdot,t)\|_{L^\infty}\le M_t$. From
Lemma \ref{lem:orlicznorm} and the monotonicity of $\Theta^{-1}$, we deduce
$$
\bigl\|p_\beta(\cdot,t)\bigr\|_{L_\Theta}\le\frac{M_t}{\Theta^{-1}(M_t)}
\le\kappa\,\frac{t^{-\frac{d}{2\beta}}}
{\Theta^{-1}\bigl(t^{-\frac{d}{2\beta}}\bigr)} ,
$$
where we used $\kappa\ge1$. The statement for $\Kb^\beta_t$ follows from
$0\le\Kb_t^\beta(\cdot,y)\le p_\beta(\cdot-y,t)$, the lattice property of
the Luxemburg norm and its translation invariance, together with the
symmetry of $\Kb^\beta_t$. Estimate \eqref{eq:row-Lr} is H\"older's
interpolation between $\|\Kb^\beta_t(\cdot,y)\|_{L^1}\le1$ and
$\|\Kb^\beta_t(\cdot,y)\|_{L^\infty}\le\kappa_{d,\beta}t^{-d/(2\beta)}$.
\end{proof}

\begin{Remark}\rm\label{rem:reproves}
Taking $\beta=1$ in Corollary \ref{cor:kernelTheta} gives
$$
\|G(\cdot,t)\|_{L_\Theta}\le \max\{1,(4\pi)^{-d/2}\}\,
\frac{t^{-d/2}}{\Theta^{-1}(t^{-d/2})},\qquad t>0,
$$
that is, \cite[Proposition 4.1]{Laister}, with a constant independent of
$\Theta$. The proof in \cite{Laister} proceeds through radial coordinates and
an induction on the dimension, carried out separately for even and odd
dimensions; its constant depends on $\Theta$ through $x_\infty^\Theta$.
Lemma \ref{lem:orlicznorm} shows that only the $L^1$ and $L^\infty$ norms of
the kernel enter. We do not know whether O'Neil's inequality itself extends to
general non-convolution kernels with uniform row and column bounds in
$L_\Theta$; the pointwise domination of Lemma~\ref{lem:kernel}(ii) makes the
question irrelevant here.
\end{Remark}

We close this section with the continuity statement needed for the
$C_0$-property.

\begin{Lemma}\label{lem:strongcont}
Let $0<\beta\le1$ and $\varrho>0$. For every $\varphi\in C_0^\infty(\R^d)$,
$$
\|S_\beta(t)\varphi-\varphi\|_{L^1}\to0
\qquad\text{and}\qquad
\|S_\beta(t)\varphi-\varphi\|_{L^\infty}\to0
\qquad\text{as }t\to0 .
$$
\end{Lemma}

\begin{proof}
Consider first $\beta=1$. Write $E_t(x,y):=G(x-y,t)-\Kb_t(x,y)\ge0$ (Lemma
\ref{lem:kernel}(ii)) and recall that
$\|G(\cdot,t)*\varphi-\varphi\|_{L^1}\to0$ and
$\|G(\cdot,t)*\varphi-\varphi\|_{L^\infty}\to0$ as $t\to0$. Put
$m_t(y):=\int_{\R^d}\Kb_t(x,y)\dd x=\int_{\R^d}\Kb_t(y,x)\dd x$; by
\eqref{eq:mass1}, $m_t(y)\uparrow1$ as $t\downarrow0$, uniformly for $y$ in
compact sets. Since $\int_{\R^d}G(x-y,t)\dd x=1$, Tonelli's theorem yields
$$
\bigl\|G(\cdot,t)*\varphi-e^{-t\Hr}\varphi\bigr\|_{L^1}
\le\int_{\R^d}\int_{\R^d}E_t(x,y)|\varphi(y)|\dd x\dd y
=\int_{\R^d}\bigl(1-m_t(y)\bigr)|\varphi(y)|\dd y\longrightarrow0
$$
by dominated convergence, which proves the $L^1$ statement. For the
$L^\infty$ statement, let $R>0$ with $\supp\varphi\subseteq B_R$. For $|x|\le2R$,
\begin{align*}
\Bigl|\int_{\R^d}E_t(x,y)\varphi(y)\dd y\Bigr|
&\le\|\varphi\|_{L^\infty}\int_{\R^d}E_t(x,y)\dd y
=\|\varphi\|_{L^\infty}\bigl(1-m_t(x)\bigr)\\
&\le\|\varphi\|_{L^\infty}\sup_{|z|\le 2R}\bigl(1-m_t(z)\bigr)\longrightarrow0 .
\end{align*}
For $|x|>2R$ we have $\varphi(x)=0$, and since $|x-y|>R$ for $y\in B_R$,
$|e^{-t\Hr}\varphi(x)|\le\|\varphi\|_{L^\infty}\int_{|z|>R}G(z,t)\dd z\to0$
uniformly in $x$. Combining the two regimes with the uniform convergence of
$G(\cdot,t)*\varphi$ to $\varphi$ completes the proof for $\beta=1$.

For $0<\beta<1$, let $X$ be either $L^1$ or $L^\infty$. We use
\eqref{E7.4}, the identity $\int_0^\infty\eta_t^\beta=1$ (together with
Tonelli's theorem for $X=L^1$ and a pointwise estimate for $X=L^\infty$), and
the contractivity of $e^{-s\Hr}$ on $X$. For every $\delta>0$, we obtain
\begin{align*}
\|S_\beta(t)\varphi-\varphi\|_{X}
&\le\int_0^{\delta}\bigl\|e^{-s\Hr}\varphi-\varphi\bigr\|_{X}
\eta^\beta_t(s)\dd s
+2\|\varphi\|_{X}\int_{\delta}^{\infty}\eta^\beta_t(s)\dd s\\
&\le\sup_{0<s\le\delta}\bigl\|e^{-s\Hr}\varphi-\varphi\bigr\|_X
+2\|\varphi\|_X\int_{\delta t^{-1/\beta}}^{\infty}\eta_1^\beta(\sigma)
\dd\sigma ,
\end{align*}
where the last step uses \eqref{eq:scaling-eta}. Given $\varepsilon>0$, choose
$\delta$ so that the first term is at most $\varepsilon/2$ by the case
$\beta=1$, and then $t$ so small that the second is at most $\varepsilon/2$.
\end{proof}

\section{Proof of Theorem \ref{thm:OrliczSmoothing} and its corollaries}
\label{sec:proofA}

The contraction property on Orlicz spaces is a consequence of the
sub-Markovian character of $S_\beta(t)$ alone. It is classical, but we include
the short proof, since we use it in the following quantitative form.

\begin{Lemma}[Sub-Markovian kernels on Orlicz spaces]\label{L5.1}
Let $k:\R^d\times\R^d\to[0,\infty)$ be measurable, let $m>0$, and assume
$$
\sup_{x\in\R^d}\int_{\R^d}k(x,y)\dd y\le m
\qquad\text{and}\qquad
\sup_{y\in\R^d}\int_{\R^d}k(x,y)\dd x\le m .
$$
Let $T\varphi(x):=\int_{\R^d}k(x,y)\varphi(y)\dd y$. If $m\le1$, then for every
$\Phi\in\Y$ and every $\varphi\in L_\Phi$ the integral defining $T\varphi$
converges absolutely a.e., and
\begin{equation}\label{eq:Orlicz-op}
\int_{\R^d}\Phi\Bigl(\frac{|T\varphi(x)|}{\mu}\Bigr)\dd x
\ \le\ \int_{\R^d}\Phi\Bigl(\frac{|\varphi(y)|}{\mu}\Bigr)\dd y
\qquad\text{for every }\mu>0 .
\end{equation}
In particular $T(M_\Phi)\subseteq M_\Phi$. For every $m>0$ and $\Phi\in\Y$,
$\|T\varphi\|_{L_\Phi}\le m\,\|\varphi\|_{L_\Phi}$.
\end{Lemma}

\begin{proof}
Let $m\le1$, fix $\mu>0$ and set $M(x):=\int_{\R^d}k(x,y)\dd y\in[0,1]$. If
$M(x)=0$ then $T\varphi(x)=0$, so assume $M(x)>0$. Since $\Phi$ is convex with
$\Phi(0)=0$ and $M(x)\le1$, we have $\Phi(M(x)z)\le M(x)\Phi(z)$ for $z\ge0$;
applying this with $z=M(x)^{-1}\int k(x,y)|\varphi(y)|\mu^{-1}\dd y$, and then
Jensen's inequality for the probability measure $k(x,y)\dd y/M(x)$ and the
lower semicontinuous convex function $\Phi$, yields
\begin{align*}
\Phi\Bigl(\frac{1}{\mu}\int_{\R^d}k(x,y)|\varphi(y)|\dd y\Bigr)
&\le M(x)\,\Phi\Bigl(\frac{1}{M(x)}\int_{\R^d}k(x,y)\frac{|\varphi(y)|}{\mu}\dd y\Bigr)\\
&\le\int_{\R^d}k(x,y)\,\Phi\Bigl(\frac{|\varphi(y)|}{\mu}\Bigr)\dd y .
\end{align*}
Integrating in $x$ and invoking Tonelli's theorem and the second hypothesis
on $k$, we obtain \eqref{eq:Orlicz-op} with $|T\varphi|$ replaced by
$\int k|\varphi|\dd y\ge|T\varphi|$. If we choose $\mu$ with
$\int\Phi(|\varphi|/\mu)<\infty$, the right-hand side is finite, so
$\Phi(\mu^{-1}\int k(x,y)|\varphi(y)|\dd y)<\infty$ for a.e.\ $x$, and
Lemma~\ref{lem:orliczbasic}(ii) shows that $\int k(x,y)|\varphi(y)|\dd y<\infty$
for a.e.\ $x$. Taking $\mu>\|\varphi\|_{L_\Phi}$ makes the right-hand side of
\eqref{eq:Orlicz-op} at most $1$, whence $\|T\varphi\|_{L_\Phi}\le\mu$. Letting
$\mu\downarrow\|\varphi\|_{L_\Phi}$, we get $\|T\varphi\|_{L_\Phi}\le
\|\varphi\|_{L_\Phi}$. The invariance of $M_\Phi$ is immediate from
\eqref{eq:Orlicz-op}. For general $m>0$, it suffices to apply the case $m=1$
to the kernel $k/m$ and to use the homogeneity of the norm.
\end{proof}

We also record the form of O'Neil's convolution theorem that we shall use.

\begin{Lemma}[{O'Neil \cite[Theorem 2.5]{ONeil}; see also
\cite[Lemma 4.3]{Laister}}]\label{lem:oneil}
Let $\Phi,\Psi,\Theta\in\Y$ satisfy \eqref{eq:compat}. If $\varphi\in L_\Phi$
and $g\in L_\Theta$ then $g*\varphi\in L_\Psi$ and
\begin{equation}\label{eq:Oneil}
\|g*\varphi\|_{L_\Psi}\ \le\ 2\,\|g\|_{L_\Theta}\,\|\varphi\|_{L_\Phi} .
\end{equation}
\end{Lemma}

\begin{proof}[Proof of Theorem \ref{thm:OrliczSmoothing}]
(a)(i) By Lemma \ref{lem:kernel}(iii), the kernel $\Kb^\beta_t$ satisfies the
hypotheses of Lemma \ref{L5.1} with $m=m_\beta(t)\le1$. We thus obtain
\eqref{E5.2} and the invariance of $M_\Phi$.

(a)(ii) Let $\Phi$ be finite, $\varphi\in L^0_\Phi$ and $\varepsilon>0$, and
choose $\psi\in C_0^\infty(\R^d)$ with $\|\varphi-\psi\|_{L_\Phi}<\varepsilon$. By
\eqref{E5.2} and Lemma~\ref{lem:orliczbasic}(iii), with
$c_\Phi:=\max\{1,\Phi(1)\}$,
\begin{align*}
\|S_\beta(t)\varphi-\varphi\|_{L_\Phi}
&\le\|S_\beta(t)(\varphi-\psi)\|_{L_\Phi}+\|S_\beta(t)\psi-\psi\|_{L_\Phi}
+\|\psi-\varphi\|_{L_\Phi}\\
&\le 2\varepsilon+c_\Phi\bigl(\|S_\beta(t)\psi-\psi\|_{L^1}
+\|S_\beta(t)\psi-\psi\|_{L^\infty}\bigr).
\end{align*}
By Lemma \ref{lem:strongcont}, the last bracket tends to $0$ as $t\to0$.
Since $\varepsilon>0$ was arbitrary, the claim follows.

(a)(iii) If $\Phi\in\Nf$, then $M_\Phi=L^0_\Phi$ by
Lemma~\ref{lem:orliczbasic}(iv). By (i), $S_\beta(t)$ is a contraction on
$M_\Phi$. The semigroup property holds by Section~\ref{subsec:operator}, and
strong continuity at $t=0$ is (ii). Strong continuity at $t_0>0$ then follows
from the contraction property: for $0<h\le t_0$,
$\|S_\beta(t_0+h)\varphi-S_\beta(t_0)\varphi\|_{L_\Phi}
+\|S_\beta(t_0-h)\varphi-S_\beta(t_0)\varphi\|_{L_\Phi}\le
2\|S_\beta(h)\varphi-\varphi\|_{L_\Phi}$.

(a)(iv) The Gaussian $h_0$ belongs to $L_\Phi$. Indeed, if $z_0>0$ satisfies
$\Phi(z_0)<\infty$ and $k\ge\|h_0\|_{L^\infty}/z_0=1/z_0$, then
$\Phi(h_0/k)\le(h_0/(kz_0))\Phi(z_0)$, which is integrable. Moreover
$\|h_0\|_{L_\Phi}>0$. By Lemma \ref{lem:kernel}(iv),
$\|S_\beta(t)h_0\|_{L_\Phi}=e^{-\lambda_\beta t}\|h_0\|_{L_\Phi}$.

(b) Let $\Theta\in\Y$ and $b\in(0,1]$ be as in \eqref{E3.1}. Since
$\Kb^\beta_t\ge0$, Lemma \ref{lem:kernel}(ii) gives the pointwise domination
$$
\bigl|S_\beta(t)\varphi(x)\bigr|
\le\int_{\R^d}\Kb^\beta_t(x,y)|\varphi(y)|\dd y
\le\bigl(p_\beta(\cdot,t)*|\varphi|\bigr)(x),\qquad x\in\R^d .
$$
The right-hand inequality in \eqref{E3.1} is \eqref{eq:compat}. Hence the
lattice property of the Luxemburg norm and Lemma \ref{lem:oneil} give
$$
\|S_\beta(t)\varphi\|_{L_\Psi}
\le\bigl\|p_\beta(\cdot,t)*|\varphi|\bigr\|_{L_\Psi}
\le 2\,\bigl\|p_\beta(\cdot,t)\bigr\|_{L_\Theta}\,\|\varphi\|_{L_\Phi} .
$$
Corollary \ref{cor:kernelTheta} bounds the middle factor by
$\kappa t^{-d/(2\beta)}/\Theta^{-1}(t^{-d/(2\beta)})$, while the left-hand
inequality in \eqref{E3.1} gives, for every $x>0$,
$$
\frac{x}{\Theta^{-1}(x)}\ \le\ \frac{1}{b}\,
\frac{\Phi^{-1}(x)}{\Psi^{-1}(x)} .
$$
Combining the two displays with $x=t^{-d/(2\beta)}$, we arrive at
\eqref{eq:OrliczOrlicz}.

(c) Put $\gamma_\ast(t):=\Psi^{-1}(t^{-d/(2\beta)})/\Phi^{-1}(t^{-d/(2\beta)})$,
so that $\gamma_\ast(t)\to0$ as $t\to0$, and let $\mathcal K\subseteq L^0_\Phi$
be compact. Suppose that there are $\varepsilon>0$, $t_j\to0$ and
$\varphi_j\in\mathcal K$ with
$\gamma_\ast(t_j)\|S_\beta(t_j)\varphi_j\|_{L_\Psi}\ge\varepsilon$. Passing to a
subsequence, $\varphi_j\to\varphi\in\mathcal K$ in $L_\Phi$. For $s>0$, by
(b) and (a)(i),
\begin{align*}
\gamma_\ast(t_j)\|S_\beta(t_j)\varphi_j\|_{L_\Psi}
&\le\gamma_\ast(t_j)\|S_\beta(t_j)(\varphi_j-S_\beta(s)\varphi)\|_{L_\Psi}
+\gamma_\ast(t_j)\|S_\beta(s)S_\beta(t_j)\varphi\|_{L_\Psi}\\
&\le\tfrac{2\kappa}{b}\|\varphi_j-S_\beta(s)\varphi\|_{L_\Phi}
+\tfrac{2\kappa}{b}\,\frac{\gamma_\ast(t_j)}{\gamma_\ast(s)}\,\|\varphi\|_{L_\Phi}.
\end{align*}
Letting $j\to\infty$ and then $s\to0$, and using (a)(ii), we obtain
$\varepsilon\le0$, a contradiction.
\end{proof}

\begin{proof}[Proof of Corollary \ref{cor:Linfty}]
By Lemma \ref{lem:kernel}(ii) and the generalised H\"older inequality,
\begin{align*}
\bigl|S_\beta(t)\varphi(x)\bigr|
&\le\int_{\R^d}p_\beta(x-y,t)|\varphi(y)|\dd y
\le 2\,\bigl\|p_\beta(x-\cdot,t)\bigr\|_{L_{\Phi^*}}\,\|\varphi\|_{L_\Phi}\\
&=2\,\bigl\|p_\beta(\cdot,t)\bigr\|_{L_{\Phi^*}}\,\|\varphi\|_{L_\Phi} ,
\end{align*}
where we used that $p_\beta(\cdot,t)$ is even. By Corollary
\ref{cor:kernelTheta} with $\Theta=\Phi^*$ and the left inequality in
\eqref{eq:24},
$$
\bigl\|p_\beta(\cdot,t)\bigr\|_{L_{\Phi^*}}
\le \kappa\,\frac{t^{-\frac{d}{2\beta}}}{(\Phi^*)^{-1}(t^{-\frac{d}{2\beta}})}
\le \kappa\,\Phi^{-1}\bigl(t^{-\tfrac{d}{2\beta}}\bigr),
$$
which is \eqref{eq:MTtoLinfty} with $C_\infty=2\kappa$.
\end{proof}

\begin{proof}[Proof of Corollary \ref{cor:Linftyzero}]
Write $\alpha(t):=\Phi^{-1}(t^{-d/(2\beta)})$ and $\gamma(t):=1/\alpha(t)$, and
let $\mathcal K\subseteq L^0_\Phi$ be compact. Since $\Phi$ is finite,
$\Phi^{-1}$ is unbounded, and therefore $\gamma(t)\to0$ as $t\to0$. Suppose that there
are $\varepsilon>0$, $t_j\to0$ and $\varphi_j\in\mathcal K$ with
$\gamma(t_j)\|S_\beta(t_j)\varphi_j\|_{L^\infty}\ge\varepsilon$, and, passing to
a subsequence, $\varphi_j\to\varphi\in\mathcal K$ in $L_\Phi$. For $s>0$,
by \eqref{eq:MTtoLinfty}, the $L^\infty$-contractivity of $S_\beta(t_j)$
(Lemma~\ref{L5.1} with $\Phi=\Phi^\infty$) and \eqref{E5.2},
\begin{align*}
\gamma(t_j)\|S_\beta(t_j)\varphi_j\|_{L^\infty}
&\le\gamma(t_j)\bigl\|S_\beta(t_j)\bigl(\varphi_j-S_\beta(s)\varphi\bigr)\bigr\|_{L^\infty}
+\gamma(t_j)\bigl\|S_\beta(t_j)S_\beta(s)\varphi\bigr\|_{L^\infty}\\
&\le C_\infty\|\varphi_j-\varphi\|_{L_\Phi}
+C_\infty\bigl\|\varphi-S_\beta(s)\varphi\bigr\|_{L_\Phi}
+\gamma(t_j)\bigl\|S_\beta(s)\varphi\bigr\|_{L^\infty} .
\end{align*}
The last term tends to $0$ as $j\to\infty$ because
$\|S_\beta(s)\varphi\|_{L^\infty}<\infty$ for fixed $s>0$. Letting $j\to\infty$
and then $s\to0$, and using Theorem \ref{thm:OrliczSmoothing}(a)(ii), we obtain
$\varepsilon\le0$, a contradiction.
\end{proof}

\begin{proof}[Proof of Corollary \ref{cor:decay} and of Proposition
\ref{prop:LpLq}]
Let $\theta\in(0,1)$ and $t>0$. We use the semigroup property,
Lemma~\ref{L5.1} applied with $\Phi^\infty$ (that is, on $L^\infty$) and with
$m=m_\beta((1-\theta)t)$, and Corollary \ref{cor:Linfty}. We obtain
$$
\|S_\beta(t)\varphi\|_{L^\infty}
\le m_\beta((1-\theta)t)\,\bigl\|S_\beta(\theta t)\varphi\bigr\|_{L^\infty}
\le 2^{\frac d2}e^{-(1-\theta)\lambda_\beta t}\,C_\infty\,
\Phi^{-1}\bigl((\theta t)^{-\frac{d}{2\beta}}\bigr)\|\varphi\|_{L_\Phi} ,
$$
and $\Phi^{-1}((\theta t)^{-d/(2\beta)})\le\theta^{-d/(2\beta)}
\Phi^{-1}(t^{-d/(2\beta)})$ by \eqref{eq:dilate}. This is \eqref{eq:largetime}.
For $t\ge1$, the same argument with the splitting $t=(t-1)+1$ gives
\eqref{eq:largetime-sharp}. For $L_\Psi$, one argues in the same way, using
Lemma~\ref{L5.1} on $L_\Psi$ and \eqref{eq:OrliczOrlicz}, together with
$\Phi^{-1}(\theta^{-a}x)/\Psi^{-1}(\theta^{-a}x)\le\theta^{-a}\Phi^{-1}(x)/\Psi^{-1}(x)$,
$a=d/(2\beta)$, which follows from \eqref{eq:dilate} and the monotonicity of
$\Psi^{-1}$. The optimality of the rate follows from
Lemma~\ref{lem:kernel}(iv), since $h_0\in L_\Phi\cap L^\infty$.

We turn to Proposition \ref{prop:LpLq}. For $q=r$, the estimate is
\eqref{E5.2}, since
$L^q=L_\Phi$ with $\Phi(x)=x^q/q$ and $\|\cdot\|_{L_\Phi}=q^{-1/q}\|\cdot\|_{L^q}$
(for $q=\infty$, use $\Phi^\infty$). For $q<r<\infty$, it is Theorem
\ref{thm:OrliczSmoothing}(b) with the choice of Remark \ref{rem:onLpLq}, and
for $r=\infty$, it is Corollary \ref{cor:Linfty} with $\Phi(x)=x^q/q$, for which
$\Phi^{-1}(x)=(qx)^{1/q}$. The exponential improvements follow from the first
part of the proof.
\end{proof}

\section{Proof of the well-posedness results}
\label{sec:proofB}

Throughout this section, $\Phi\in\Nf$, $f$ satisfies $(\mathrm L)$, and
$$
\alpha(t):=\Phi^{-1}\bigl(t^{-\frac{d}{2\beta}}\bigr),
\qquad
\gamma(t):=\frac{1}{\alpha(t)},\qquad t>0 ,
$$
so that $\alpha$ is decreasing, $\gamma$ is increasing, and by \eqref{eq:dilate}
\begin{equation}\label{eq:alphadouble}
\alpha(t/2)\le 2^{\frac{d}{2\beta}}\alpha(t),\qquad t>0 .
\end{equation}
Given $\lambda>0$ as in \eqref{eq:IC-infty}, we abbreviate
$$
I(\tau):=\int_0^{\tau}\ell\bigl(\lambda\alpha(s)\bigr)\dd s
=\frac{2\beta}{d}\int_{\tau^{-\frac{d}{2\beta}}}^{\infty}
x^{-p_\ast}\ell\bigl(\lambda\Phi^{-1}(x)\bigr)\dd x ,
$$
the equality being \eqref{eq:changevar}. Hypothesis \eqref{eq:IC-infty} says
precisely that $I(\tau)<\infty$ for every $\tau>0$; it follows that
$I(\tau)\downarrow0$ as $\tau\downarrow0$. We write
$\|\cdot\|_\infty:=\|\cdot\|_{L^\infty}$ and $\|\cdot\|_\Phi:=\|\cdot\|_{L_\Phi}$.
Since $\|\cdot\|_\infty$ is the Luxemburg norm associated with
$\Phi^\infty\in\Y$, Lemmas~\ref{lem:minkowski} and \ref{L5.1} apply to it.

\subsection{Auxiliary facts on integral solutions}

\begin{Lemma}\label{lem:restart}
Let $u$ be an integral solution of \eqref{eq:main} in $Q_T$.
\begin{enumerate}[label=\textup{(\roman*)}]
\item For a.e.\ $s\in(0,T)$, the function $u(s+\cdot)$ is an integral solution
in $Q_{T-s}$ with data $u(s)$.
\item If $s$ is as in \textup{(i)} and $v$ is an integral solution in $Q_\tau$
with data $u(s)$, then the function equal to $u$ on $\R^d\times(0,s]$ and to
$v(\cdot-s)$ on $\R^d\times(s,s+\tau)$ is an integral solution in $Q_{s+\tau}$
with data $u(0)$.
\item If $u$ is a bounded $M_\Phi$-mild solution on $[0,T)$, then
$u(t)=\mathcal F(u;u_0)(t)$ a.e.\ in $\R^d$ for \emph{every} $t\in(0,T)$, and
\textup{(i)} holds for every $s\in(0,T)$; the restarted function
$u(s+\cdot)$ is a bounded $M_\Phi$-mild solution on $[0,T-s)$.
\end{enumerate}
\end{Lemma}

\begin{proof}
(i) Let $s$ be such that $u(s)=\mathcal F(u;u_0)(s)$ a.e.\ and the integrals
defining $\mathcal F(u;u_0)(s)$ converge absolutely a.e.; this holds for
a.e.\ $s$. By the semigroup property and Fubini's theorem (the integrals being
absolutely convergent for a.e.\ $(x,t)$),
$$
\mathcal F(u;u_0)(s+t)=S_\beta(t)\Bigl[S_\beta(s)u_0+\int_0^sS_\beta(s-r)f(u(r))
\dd r\Bigr]+\int_0^tS_\beta(t-r)f(u(s+r))\dd r ,
$$
where the bracket is $u(s)$, and $S_\beta(t)|u(s)|\le S_\beta(s+t)|u_0|+
\int_0^sS_\beta(s+t-r)|f(u(r))|\dd r<\infty$ a.e.

(ii) This is proved by the same computation, read backwards.

(iii) Let $t\in(0,T)$, choose $t_0\in(0,t)$ as in
(i), and let $M$ bound $|u|$ on $\R^d\times(t_0,t+\varepsilon)$. For $\tau$
near $t$, (i) gives $u(\tau)=S_\beta(\tau-t_0)u(t_0)+D(\tau)$ with
$D(\tau):=\int_{t_0}^\tau S_\beta(\tau-r)f(u(r))\dd r$ for a.e.\ such $\tau$.
Since $\|f(u(r))\|_\Phi\le\ell(M)\|u(r)\|_\Phi$ is bounded, the argument given
in the proof of Theorem~\ref{thm:local}(a) below shows that $\tau\mapsto D(\tau)$
is continuous into $L_\Phi$; so is $\tau\mapsto S_\beta(\tau-t_0)u(t_0)$.
As $u\in C([0,T),M_\Phi)$, the identity extends to $\tau=t$. By the computation
in (i), the right-hand side at $\tau=t$ coincides with $\mathcal F(u;u_0)(t)$.
The remaining assertions follow.
\end{proof}

\begin{Lemma}\label{lem:Linfty}
Let $\psi\in M_\Phi\cap L^\infty$ with $\|\psi\|_\infty\le M$, and let
$\tau_M>0$ satisfy $\tau_M\,\ell(2M)\le\frac12$. Then \eqref{eq:main} with data
$\psi$ has a bounded $M_\Phi$-mild solution on $[0,\tau_M]$ with
$\|u(t)\|_\infty\le2M$. Moreover, two integral solutions in $Q_\tau$ with the
same data which both belong to $L^\infty(Q_\tau)$ coincide.
\end{Lemma}

\begin{proof}
Let $X$ be the set of measurable $u$ on $Q_{\tau_M}$ with $|u|\le2M$ and
$\esssup_t\|u(t)\|_\Phi\le2\|\psi\|_\Phi$, with the metric of
$L^\infty(Q_{\tau_M})$; it is complete by Lemma~\ref{lem:orliczbasic}(i). For
$u\in X$, we infer from \eqref{E5.2}, its $L^\infty$ analogue, and
Lemma~\ref{lem:minkowski} that
\begin{gather*}
\|\mathcal F(u;\psi)(t)\|_\infty\le M+\tau_M\ell(2M)2M\le2M,\\
\|\mathcal F(u;\psi)(t)\|_\Phi\le\|\psi\|_\Phi+\tau_M\ell(2M)2\|\psi\|_\Phi
\le2\|\psi\|_\Phi,\\
\|\mathcal F(u;\psi)-\mathcal F(v;\psi)\|_{L^\infty(Q_{\tau_M})}\le\tfrac12
\|u-v\|_{L^\infty(Q_{\tau_M})}.
\end{gather*} The fixed point is a bounded integral
solution, and its continuity with values in $M_\Phi$ is proved as in
Theorem~\ref{thm:local}(a) below. Finally, if $v_1,v_2$ are two integral
solutions in $Q_\tau$ bounded by $M'$, then
$\|v_1(t)-v_2(t)\|_\infty\le\ell(M')\int_0^t\|v_1(s)-v_2(s)\|_\infty\dd s$, and
we conclude from Lemma~\ref{lem:gronwall} that $v_1=v_2$.
\end{proof}

\subsection{Uniform local existence}

\begin{proof}[Proof of Theorem \ref{thm:local}\textup{(a)}]
Let $\mathcal K\subseteq M_\Phi$ be compact. By the lattice property,
$\||\phi|-|\psi|\|_\Phi\le\|\phi-\psi\|_\Phi$, so that
$|\mathcal K|:=\{|\psi|:\psi\in\mathcal K\}$ is compact in $M_\Phi=L^0_\Phi$.
By Corollary \ref{cor:Linftyzero} there exists $\tau_{\mathcal K}>0$ such that
\begin{equation}\label{eq:barrierinfty}
\bigl\|2S_\beta(s)|\psi|\bigr\|_{\infty}\le\lambda\,\alpha(s),
\qquad s\in(0,\tau_{\mathcal K}),\ \psi\in\mathcal K .
\end{equation}
Choose $T=T_{\mathcal K}\in(0,\tau_{\mathcal K})$ with $I(T)\le\frac12$; this
choice does not depend on $u_0\in\mathcal K$. Fix $u_0\in\mathcal K$, put
\begin{equation}\label{eq:barrier}
w(t):=2S_\beta(t)|u_0|,\qquad t\ge0 ,
\end{equation}
and let
\begin{equation}\label{eq:space}
\begin{gathered}
X_{u_0}:=\Bigl\{u:Q_T\to\R\ \text{measurable}:\
|u|\le w\ \text{a.e.\ in } Q_T\Bigr\},\\
d(u,v):=\esssup_{t\in(0,T)}\|u(t)-v(t)\|_{\Phi}.
\end{gathered}
\end{equation}
Since $\|w(t)\|_\Phi\le2\|u_0\|_\Phi$ by \eqref{E5.2}, Lemma~\ref{lem:complete}
shows that $(X_{u_0},d)$ is a nonempty complete metric space. Note that
$u(t)\in M_\Phi$ for every $u\in X_{u_0}$ and a.e.\ $t$, since $w(t)\in
M_\Phi$ by Theorem~\ref{thm:OrliczSmoothing}(a)(i).

We first show that $\mathcal F=\mathcal F(\cdot\,;u_0)$ maps $X_{u_0}$ into
itself. Let $u\in X_{u_0}$. By \eqref{eq:barrierinfty}, $|u(x,s)|\le
w(x,s)\le\lambda\alpha(s)$ for a.e.\ $(x,s)\in Q_T$, and \eqref{eq:lipbound}
yields
\begin{equation}\label{eq:fbound}
|f(u(s))|\le\ell\bigl(\lambda\alpha(s)\bigr)\,|u(s)|
\le\ell\bigl(\lambda\alpha(s)\bigr)\,w(s)\qquad\text{a.e.\ in } Q_T .
\end{equation}
Since $S_\beta(t-s)$ has a nonnegative kernel and $S_\beta(t-s)w(s)
=2S_\beta(t-s)S_\beta(s)|u_0|=w(t)$, we obtain, for every $t\in(0,T)$,
\begin{equation}\label{eq:selfmap}
\bigl|\mathcal F(u)(t)\bigr|
\le S_\beta(t)|u_0|+\int_0^t\ell\bigl(\lambda\alpha(s)\bigr)S_\beta(t-s)w(s)\dd s
=\frac{w(t)}{2}+w(t)\,I(t)\ \le\ w(t) .
\end{equation}
In particular all integrals in $\mathcal F(u)$ converge absolutely a.e., and
$\mathcal F(u)$ is measurable by Tonelli's theorem. Hence $\mathcal
F(X_{u_0})\subseteq X_{u_0}$.

Next, we show that $\mathcal F$ is a contraction. Let $u,v\in X_{u_0}$ and
$t\in(0,T)$. By Lemma~\ref{lem:minkowski}, \eqref{E5.2}, \eqref{eq:lipbound}
and \eqref{eq:barrierinfty}, we get
\begin{equation}\label{eq:contraction}
\begin{aligned}
\bigl\|\mathcal F(u)(t)-\mathcal F(v)(t)\bigr\|_{\Phi}
&\le\int_0^t\bigl\|S_\beta(t-s)\bigl|f(u(s))-f(v(s))\bigr|\bigr\|_{\Phi}\dd s\\
&\le\int_0^t\bigl\|f(u(s))-f(v(s))\bigr\|_{\Phi}\dd s\\
&\le\int_0^t\ell\bigl(\lambda\alpha(s)\bigr)\|u(s)-v(s)\|_{\Phi}\dd s
\ \le\ I(T)\,d(u,v)\le\tfrac12\,d(u,v).
\end{aligned}
\end{equation}
Banach's fixed point theorem provides a unique $u=u(\cdot\,;u_0)\in X_{u_0}$
with $\mathcal F(u)=u$ a.e. We use the representative
$u(t):=\mathcal F(u)(t)$, which is defined for every $t\in(0,T)$ and satisfies
$|u(t)|\le w(t)$ by \eqref{eq:selfmap}. Thus $u$ is an integral solution in
$Q_T$, with $T=T_{\mathcal K}$ independent of $u_0\in\mathcal K$.

\emph{Boundedness and \eqref{eq:smallrate}.} By \eqref{eq:barrierinfty}, we
have $\|u(t)\|_\infty\le\|w(t)\|_\infty\le\lambda\alpha(t)$. Hence $u\in
L^\infty_{\rm loc}((0,T),L^\infty)$, and \eqref{eq:smallrate} follows from
$|u(t)|\le2S_\beta(t)|u_0|$ and Corollary \ref{cor:Linftyzero} applied to the
compact set $|\mathcal K|$.

\emph{Continuity in $M_\Phi$.} By Theorem~\ref{thm:OrliczSmoothing}(a)(iii),
the map $t\mapsto S_\beta(t)u_0$ belongs to $C([0,\infty),M_\Phi)$. It
therefore suffices to consider $D(t):=\int_0^tS_\beta(t-s)f(u(s))\dd s$. By \eqref{eq:fbound},
$f(u(s))\in M_\Phi$ with
$\|f(u(s))\|_\Phi\le\ell(\lambda\alpha(s))\|w(s)\|_\Phi\le2\|u_0\|_\Phi\,
\ell(\lambda\alpha(s))=:g(s)$, and $g\in L^1(0,T)$. By
Lemma~\ref{lem:minkowski} and \eqref{E5.2}, $\|D(t)\|_\Phi\le\int_0^tg\to0$ as
$t\to0$. For $0<t_0<t_1<T$, the semigroup property gives
$$
D(t_1)-D(t_0)=\int_0^{t_0}\bigl(S_\beta(t_1-t_0)-I\bigr)S_\beta(t_0-s)f(u(s))\dd s
+\int_{t_0}^{t_1}S_\beta(t_1-s)f(u(s))\dd s ,
$$
whence, by Lemma~\ref{lem:minkowski} and \eqref{E5.2},
$$
\|D(t_1)-D(t_0)\|_\Phi\le\int_0^{t_0}\bigl\|\bigl(S_\beta(t_1-t_0)-I\bigr)
f(u(s))\bigr\|_\Phi\dd s+\int_{t_0}^{t_1}g(s)\dd s ,
$$
where we used that $S_\beta(t_0-s)$ commutes with $S_\beta(t_1-t_0)-I$ and is
a contraction. Let $t_1-t_0\to0$, with either $t_0$ or $t_1$ fixed. By the
strong continuity of $S_\beta$ on $M_\Phi$, the integrand of the first
integral tends to $0$ for a.e.\ $s$, and it is dominated by $2g(s)$. By
dominated convergence, both integrals tend to $0$. Hence $D\in
C([0,T),M_\Phi)$, and $u$ is a bounded $M_\Phi$-mild solution.

\emph{Sign.} Let $f$ be positone and $u_0\ge0$. The set
$X^+_{u_0}:=\{u\in X_{u_0}:u\ge0\}$ is nonempty and closed in $(X_{u_0},d)$,
since the limit in Lemma~\ref{lem:complete} is an a.e.\ limit. It is also
invariant under $\mathcal F$, because $f\ge0$ on $[0,\infty)$. Consequently,
the fixed point lies in $X^+_{u_0}$. The case $u_0\le0$ is symmetric.
\end{proof}

\subsection{Continuous dependence and uniqueness}

\begin{proof}[Proof of Theorem \ref{thm:local}\textup{(c)}]
Let $\mathcal K\subseteq M_\Phi$ be compact, let $\phi,\varphi\in\mathcal K$ and
write $u:=u(\cdot\,;\phi)$, $v:=u(\cdot\,;\varphi)$ for the solutions of part
(a), and $\tau_{\mathcal K}':=\min\{\tau_{\mathcal K},T_{\mathcal K}\}/2$. By
construction and \eqref{eq:barrierinfty},
$\|u(t)\|_\infty,\|v(t)\|_\infty\le\lambda\alpha(t)$ for $t\in(0,T_{\mathcal
K})$. Put
$$
y(t):=\|u(t)-v(t)\|_{\Phi}+\gamma(t)\,\|u(t)-v(t)\|_{\infty},
\qquad t\in(0,\tau'_{\mathcal K}] ;
$$
$y$ is measurable by Lemma~\ref{lem:minkowski}. It is also bounded, since
$\|u(t)-v(t)\|_\Phi\le2(\|\phi\|_\Phi+\|\varphi\|_\Phi)$ and
$\gamma(t)\|u(t)-v(t)\|_\infty\le2\lambda$. Arguing as in
\eqref{eq:contraction}, but keeping the initial data, and using \eqref{E5.2}, we get
\begin{equation}\label{eq:cd1}
\|u(t)-v(t)\|_{\Phi}\le\|\phi-\varphi\|_{\Phi}
+\int_0^t\ell\bigl(\lambda\alpha(s)\bigr)\|u(s)-v(s)\|_{\Phi}\dd s .
\end{equation}
For the $L^\infty$ part, we split the Duhamel integral at $t/2$. On $(t/2,t)$,
we use the $L^\infty$-contractivity of $S_\beta$; on $(0,t/2)$, we use
\eqref{eq:MTtoLinfty} together with $\alpha(t-s)\le\alpha(t/2)\le
2^{d/(2\beta)}\alpha(t)$ for $s\le t/2$:
\begin{align*}
\|u(t)-v(t)\|_{\infty}
&\le C_\infty\alpha(t)\|\phi-\varphi\|_{\Phi}
+\int_{t/2}^{t}\ell\bigl(\lambda\alpha(s)\bigr)\|u(s)-v(s)\|_{\infty}\dd s\\
&\qquad
+C_\infty2^{\frac{d}{2\beta}}\alpha(t)\int_0^{t/2}\ell\bigl(\lambda\alpha(s)\bigr)
\|u(s)-v(s)\|_{\Phi}\dd s .
\end{align*}
Multiplying by $\gamma(t)$ and noting that $\gamma(t)\le2^{d/(2\beta)}\gamma(s)$
for $s\in[t/2,t]$, again by \eqref{eq:alphadouble}, we arrive at
\begin{equation}\label{eq:cd2}
\gamma(t)\|u(t)-v(t)\|_{\infty}
\le C_\infty\|\phi-\varphi\|_{\Phi}
+C'\int_0^{t}\ell\bigl(\lambda\alpha(s)\bigr)\,y(s)\dd s,
\qquad C':=\max\{1,C_\infty\}2^{\frac{d}{2\beta}} .
\end{equation}
We add \eqref{eq:cd1} and \eqref{eq:cd2}, and apply Lemma~\ref{lem:gronwall}
on $(0,\tau'_{\mathcal K}]$ with $k(s)=(1+C')\ell(\lambda\alpha(s))$. We obtain
\eqref{eq:contdep} with $\tau_{\mathcal K}$ replaced by $\tau'_{\mathcal K}$ and
$C_{\mathcal K}:=(1+C_\infty)\exp\bigl((1+C')I(\tau'_{\mathcal K})\bigr)$. (The
inequality holds for a.e.\ $t$, and then for every $t$ by the continuity of
$t\mapsto u(t)-v(t)$ in $L_\Phi$ and the lower semicontinuity of
$\|\cdot\|_\infty$ under a.e.\ convergence.) Renaming $\tau'_{\mathcal K}$ as
$\tau_{\mathcal K}$ proves (c).
\end{proof}

\begin{proof}[Proof of Theorem \ref{thm:local}\textup{(b)}]
We follow \cite[Section 5.3]{Laister}, which in turn adapts
\cite{brezis1996nonlinear,LaisterSierzega2020}. Let $v$ be a bounded
$M_\Phi$-mild solution on $[0,T)$ with $v(0)=u_0$, and fix $T_1\in(0,T)$. The set
$\mathcal K':=v([0,T_1])$ is compact in $M_\Phi$; let $T_{\mathcal K'}$,
$\tau_{\mathcal K'}$ and $C_{\mathcal K'}$ be given by (a) and (c). Let
$s\in(0,T_1)$. By Lemma~\ref{lem:restart}(iii), $v(s+\cdot)$ is a bounded
$M_\Phi$-mild solution with data $v(s)\in\mathcal K'\cap L^\infty$, while
$u(\cdot\,;v(s))$ is a solution with the same data satisfying
$|u(t;v(s))|\le2S_\beta(t)|v(s)|\le2\|v(s)\|_\infty$. Both are bounded on
$\R^d\times(0,\tau')$ for every $\tau'<\min\{T_{\mathcal K'},T_1-s\}$, and
Lemma~\ref{lem:Linfty} gives
$$
v(s+t)=u(t;v(s)),\qquad 0\le t<\min\{T_{\mathcal K'},T_1-s\}.
$$
Fix $t\in(0,\min\{\tau_{\mathcal K'},T_1\})$ and let $s\downarrow0$. Then
$v(s+t)\to v(t)$ in $M_\Phi$ by continuity of $v$, while
$\|u(t;v(s))-u(t;u_0)\|_\Phi\le C_{\mathcal K'}\|v(s)-u_0\|_\Phi\to0$ by (c).
Hence $v(t)=u(t;u_0)$ for $0\le t\le t_1:=\frac12\min\{\tau_{\mathcal
K'},T_1\}$. On the remaining interval, both functions are integral solutions
starting from $v(t_1)$ (Lemma~\ref{lem:restart}(iii)), and they are bounded on
compact subintervals. Lemma~\ref{lem:Linfty} then yields $v=u(\cdot;u_0)$ on
their common interval of existence.

Uniqueness allows us to glue solutions defined on overlapping intervals. We
let $T_{u_0}\in(0,\infty]$ be the supremum of all $T$ for which a bounded
$M_\Phi$-mild solution exists on $[0,T)$. Suppose that $T_{u_0}<\infty$ and that
there are $M<\infty$ and $t_j\uparrow T_{u_0}$ with $\|u(t_j)\|_\infty\le M$. By
Lemma~\ref{lem:Linfty} there is a bounded $M_\Phi$-mild solution with data
$u(t_j)$ on $[0,\tau_M]$, with $\tau_M$ independent of $j$. By
Lemma~\ref{lem:restart}(ii),(iii), its concatenation with $u$ is a bounded
$M_\Phi$-mild solution on $[0,t_j+\tau_M)$. Since $t_j+\tau_M>T_{u_0}$ for $j$
large, this is a contradiction. Hence $\|u(t)\|_\infty\to\infty$ as $t\uparrow
T_{u_0}$.
\end{proof}

\subsection{Comparison principle}

\begin{proof}[Proof of Theorem \ref{thm:local}\textup{(d)}]
Let $f$ be increasing, $\phi\le\varphi$ in $M_\Phi$, $u:=u(\cdot\,;\phi)$,
$v:=u(\cdot\,;\varphi)$ and $z:=u-v$. Since $f$ is increasing, for $|a|,|b|\le
s$ we have $f(a)-f(b)\le\ell(s)(a-b)^+$. Apply (a) to the compact set
$\{\phi,\varphi\}$: for $0<t<T_0:=T_{\{\phi,\varphi\}}$, both solutions are
bounded by $\lambda\alpha(t)$. Since $S_\beta(t)(\phi-\varphi)\le0$, we get
$$
z(t)\le\int_0^tS_\beta(t-s)\bigl(f(u(s))-f(v(s))\bigr)\dd s
\le\int_0^t\ell\bigl(\lambda\alpha(s)\bigr)S_\beta(t-s)z^+(s)\dd s .
$$
The right-hand side is nonnegative, so $z^+(t)$ obeys the same bound. Then
Lemma~\ref{lem:minkowski} and \eqref{E5.2} give
$\|z^+(t)\|_\Phi\le\int_0^t\ell(\lambda\alpha(s))\|z^+(s)\|_\Phi\dd s$. As
$\|z^+(t)\|_\Phi$ is bounded, Lemma~\ref{lem:gronwall} yields $z^+=0$ on
$(0,T_0)$. Now let $t_0\in(0,T_0)$ and $T'<\min\{T_\phi,T_\varphi\}$, and let
$M$ bound $|u|$ and $|v|$ on $\R^d\times(t_0,T')$. By
Lemma~\ref{lem:restart}(iii), for $t\in(t_0,T')$,
$z(t)\le S_\beta(t-t_0)z(t_0)+\ell(M)\int_{t_0}^tS_\beta(t-s)z^+(s)\dd s$ with
$z(t_0)\le0$. Hence $\|z^+(t)\|_\infty\le\ell(M)\int_{t_0}^t\|z^+(s)\|_\infty\dd s$,
and we conclude from Lemma~\ref{lem:gronwall} that $z^+=0$ on $(t_0,T')$.
\end{proof}

\begin{Remark}\rm
The comparison principle can also be obtained, as in
\cite[Section 5.4]{Laister}, by monotone iteration from the super- and
subsolutions $\pm2S_\beta(t)\phi^\pm$. The direct argument above avoids the
passage to the limit in the iteration.
\end{Remark}

\subsection{Global well-posedness for small data}

\begin{proof}[Proof of Corollary \ref{thm:global}]
Let $u_0\in M_\Phi$ with $\|u_0\|_\Phi\le\lambda/(AC_\infty)$ and put
$w(t):=A\,S_\beta(t)|u_0|$. By \eqref{eq:MTtoLinfty},
\begin{equation}\label{eq:globalbarrier}
\|w(s)\|_{\infty}\le AC_\infty\,\alpha(s)\,\|u_0\|_{\Phi}\le\lambda\,\alpha(s),
\qquad s>0 .
\end{equation}
Let $X^\infty_{u_0}$ be the set of measurable $u$ on $Q_\infty$ with $|u|\le w$,
equipped with the metric $d_\infty(u,v):=\esssup_{t>0}\|u(t)-v(t)\|_\Phi$.
Since $\|w(t)\|_\Phi\le A\|u_0\|_\Phi$, it is a nonempty complete metric space
by Lemma~\ref{lem:complete}. For $u\in X^\infty_{u_0}$ we have
$|f(u(s))|\le\ell(\lambda\alpha(s))w(s)$ by \eqref{eq:lipbound} and
\eqref{eq:globalbarrier}, whence, using $S_\beta(t-s)w(s)=w(t)$ and
\eqref{eq:changevar} with $T=\infty$,
$$
\bigl|\mathcal F(u)(t)\bigr|
\le S_\beta(t)|u_0|+\int_0^{t}\ell\bigl(\lambda\alpha(s)\bigr)
S_\beta(t-s)w(s)\dd s
\le\frac{w(t)}{A}+w(t)\,I_0=w(t),\qquad t>0 ,
$$
by the definition of $A$. Thus $\mathcal F$ maps $X^\infty_{u_0}$ into itself.
The computation \eqref{eq:contraction}, performed on $(0,\infty)$, gives
$d_\infty(\mathcal F(u),\mathcal F(v))\le I_0\,d_\infty(u,v)$ with $I_0<1$.
Hence $\mathcal F$ has a unique fixed point $u\in X^\infty_{u_0}$. Exactly as in the
proof of Theorem \ref{thm:local}(a) (with $g(s)=A\|u_0\|_\Phi\ell(\lambda\alpha(s))
\in L^1(0,\infty)$), $u$ is a bounded $M_\Phi$-mild solution on $[0,\infty)$,
and it is unique by Theorem \ref{thm:local}(b).

Estimate \eqref{eq:globaldecay1} follows from $|u(t)|\le w(t)$ and
\eqref{eq:globalbarrier}. Since $\Phi$ is positive, $\Phi^{-1}(x)\to0$ as
$x\downarrow0$, and therefore $\alpha(t)\to0$ as $t\to\infty$. Estimates
\eqref{eq:globaldecay2} and \eqref{eq:expdecay} follow from $|u(t)|\le
AS_\beta(t)|u_0|$, \eqref{E5.2} and \eqref{eq:largetime-sharp}.
\end{proof}

\begin{proof}[Proof of Theorem \ref{thm:confined}]
Put $C_\ast:=C_\infty2^{d/2}e^{\lambda_\beta}\Phi^{-1}(1)$ and, for $\delta>0$,
$$
\omega_\delta(s):=
\begin{cases}
2C_\infty\delta\,\alpha(s), & 0<s<1,\\
2C_\ast\delta\,e^{-\lambda_\beta s}, & s\ge1,
\end{cases}
\qquad
J(\delta):=\int_0^\infty\ell\bigl(\omega_\delta(s)\bigr)\dd s .
$$
The Dini condition \eqref{eq:dini} forces $\ell(0^+)=0$; otherwise, we would
have $\int_0^1\ell(r)\,\dd r/r\ge\ell(0^+)\int_0^1\dd r/r=\infty$. The substitution
$r=2C_\ast\delta e^{-\lambda_\beta s}$ gives
$$
J(\delta)=\int_0^1\ell\bigl(2C_\infty\delta\,\alpha(s)\bigr)\dd s
+\frac{1}{\lambda_\beta}\int_0^{2C_\ast\delta e^{-\lambda_\beta}}
\frac{\ell(r)}{r}\dd r .
$$
For $2C_\infty\delta\le\lambda$ the first integrand is bounded by
$\ell(\lambda\alpha(s))$, which is integrable on $(0,1)$ by \eqref{eq:IC-infty}
and \eqref{eq:changevar}, and it tends to $\ell(0^+)=0$ as $\delta\to0$. The
second term tends to $0$ by \eqref{eq:dini}. Hence we may fix $\delta>0$ with
$2C_\infty\delta\le\lambda$ and $J(\delta)\le\frac12$.

Let $\|u_0\|_\Phi\le\delta$ and $w(t):=2S_\beta(t)|u_0|$. By
Corollary~\ref{cor:Linfty} for $s<1$ and \eqref{eq:largetime-sharp} for
$s\ge1$, we have $\|w(s)\|_\infty\le\omega_\delta(s)$ for all $s>0$. Consider
the complete metric space of measurable $u$ on $Q_\infty$ with $|u|\le w$,
equipped with $d_\infty$. The argument of the preceding proof gives
$|\mathcal F(u)(t)|\le\frac{w(t)}2+w(t)J(\delta)\le w(t)$ and
$d_\infty(\mathcal F(u),\mathcal F(v))\le J(\delta)d_\infty(u,v)\le\frac12
d_\infty(u,v)$. The fixed point is a bounded $M_\Phi$-mild solution on
$[0,\infty)$ (with $g(s)=2\delta\,\ell(\omega_\delta(s))\in L^1(0,\infty)$ in the
continuity argument), and it is unique by Theorem~\ref{thm:local}(b). The
estimates follow from $|u(t)|\le2S_\beta(t)|u_0|$, \eqref{E5.2} and
\eqref{eq:largetime-sharp}.
\end{proof}

\section{Critical Young's functions and nonexistence}\label{sec:nonexistence}

In this section, $f$ satisfies $(\mathrm C)$ and $(\mathrm G)$, and $F$ is
defined by \eqref{E 1.2}. By Remark~\ref{rem:onG}, $F$ is a strictly
decreasing $C^1$-bijection of $(0,\infty)$ onto itself, and
\begin{equation}\label{eq:Finv}
F\bigl(F^{-1}(\sigma)\bigr)=\sigma\ \ \text{for }\sigma>0,
\qquad
F^{-1}(\sigma)\to\infty\ \text{ as }\sigma\to0^{+}.
\end{equation}
Recall also that $f$ is increasing on $[0,\infty)$ and that $f'(u)\ge
f(u)/u>0$ for $u>0$, by convexity.

\subsection{Monotone iteration and comparison of data}

\begin{Lemma}\label{lem:supersol}
Let $f$ satisfy $(\mathrm C)$, let $\mu\ge0$ be measurable and let
$U:Q_T\to[0,\infty)$ be measurable with
$S_\beta(t)\mu+\int_0^tS_\beta(t-s)f(U(s))\dd s\le U(t)$ a.e.\ in $Q_T$. Then
\eqref{eq:main} with data $\mu$ has a nonnegative integral solution $u\le U$
in $Q_T$.
\end{Lemma}

\begin{proof}
Let $v_0:=0$ and $v_{k+1}:=\mathcal F(v_k;\mu)$. Since $f\ge0$ is increasing on
$[0,\infty)$ and the kernel is nonnegative, we get $0\le v_k\le v_{k+1}\le U$
a.e.\ by induction. The limit $u:=\lim_kv_k\le U$ is finite a.e., and the
monotone convergence theorem and the continuity of $f$ give $\mathcal
F(u;\mu)=\lim_k\mathcal F(v_k;\mu)=u$, with absolutely convergent integrals.
\end{proof}

\begin{Corollary}\label{cor:datacomp}
Let $f$ satisfy $(\mathrm C)$ and $0\le\mu_2\le\mu_1$. If \eqref{eq:main} has a
nonnegative integral solution in $Q_T$ with data $\mu_1$, then it has one with
data $\mu_2$.
\end{Corollary}

\begin{proof}
It suffices to apply Lemma~\ref{lem:supersol} with $\mu=\mu_2$ and with $U$
the solution for $\mu_1$.
\end{proof}

\begin{Lemma}[Comparison with the free problem]\label{lem:freecomp}
Let $f$ satisfy $(\mathrm C)$ and $\mu\ge0$. If the free problem
$\partial_tu+(-\Delta)^\beta u=f(u)$ has a nonnegative integral solution $U$ in
$Q_T$ with data $\mu$ (defined with the kernel $p_\beta$), then
\eqref{eq:main} has a nonnegative integral solution $u\le U$ in $Q_T$ with
data $\mu$.
\end{Lemma}

\begin{proof}
Since $0\le\Kb^\beta_t(x,y)\le p_\beta(x-y,t)$ and all integrands are
nonnegative, we have $S_\beta(t)\mu+\int_0^tS_\beta(t-s)f(U(s))\dd s\le U(t)$.
The claim follows from Lemma~\ref{lem:supersol}.
\end{proof}

\subsection{A necessary condition for local solvability}

\begin{Lemma}\label{L7.2}
Let $0<\beta\le1$, let $f$ satisfy $(\mathrm C)$ with $F(u)<\infty$ for
$u>0$, and let $u_0\ge0$ be measurable and not a.e.\ zero. If \eqref{eq:main}
possesses a nonnegative integral solution in $Q_T$, then for a.e.\
$T'\in(0,T)$,
\begin{equation}\label{eq:necessary}
F\bigl([S_\beta(T')u_0](x)\bigr)\ \ge\ T'
\qquad\text{for a.e. }x\in\R^d .
\end{equation}
If moreover $u_0\in L^1$, then $F([S_\beta(t)u_0](x))\ge t$ for every
$x\in\R^d$ and every $t\in(0,T)$.
\end{Lemma}

\begin{proof}
\emph{Jensen's inequality for $S_\beta$.} Let $v\ge0$ be measurable. Since
$f$ is convex with $f(0)=0$ and
$M^\beta_\tau(x):=\int\Kb^\beta_\tau(x,y)\dd y\le1$ by Lemma
\ref{lem:kernel}(iii), we have $f(z)\le M^\beta_\tau(x)f(z/M_\tau^\beta(x))$ for
$z\ge0$. Moreover, Jensen's inequality for the probability measure
$\Kb^\beta_\tau(x,y)\dd y/M^\beta_\tau(x)$ gives
$f\bigl(M_\tau^\beta(x)^{-1}[S_\beta(\tau)v](x)\bigr)\le
M^\beta_\tau(x)^{-1}[S_\beta(\tau)f(v)](x)$. Hence
\begin{equation}\label{eq:jensen}
f\bigl([S_\beta(\tau)v](x)\bigr)\ \le\ [S_\beta(\tau)f(v)](x),
\qquad \tau>0,\ x\in\R^d ,
\end{equation}
with the convention $f(\infty)=\infty$.

\emph{A differential inequality.} Let $u\ge0$ be an integral solution in
$Q_T$. By Fubini's theorem, for a.e.\ $T'\in(0,T)$, the function
$\mathcal F(u;u_0)(T')$ coincides a.e.\ with $u(T')$ and is therefore finite
a.e. Fix such a $T'$ and let $\Omega:=\{x:\mathcal F(u;u_0)(T')(x)<\infty\}$,
a set of full measure. For $s\in(0,T')$ put
$$
U(s):=S_\beta(T')u_0+\int_0^sS_\beta(T'-\tau)f(u(\tau))\dd\tau\ \in[0,\infty].
$$
By the semigroup property and Tonelli's theorem,
$U(s)=S_\beta(T'-s)\mathcal F(u;u_0)(s)$. Since $\mathcal F(u;u_0)(s)=u(s)$
a.e.\ for a.e.\ $s$, we deduce that $U(s)=S_\beta(T'-s)u(s)$ everywhere in $\R^d$ for
a.e.\ $s\in(0,T')$. Moreover $s\mapsto U(s)(x)$ is nondecreasing, and
$U(s)\le\mathcal F(u;u_0)(T')$, so that $U(s)(x)<\infty$ for $x\in\Omega$. For
$x\in\Omega$ and $0<s<t<T'$, \eqref{eq:jensen} with $v=u(\tau)$ gives
\begin{equation}\label{eq:semigroupid}
U(t)(x)=U(s)(x)+\int_s^t[S_\beta(T'-\tau)f(u(\tau))](x)\dd\tau
\ \ge\ U(s)(x)+\int_s^tf\bigl(U(\tau)(x)\bigr)\dd\tau .
\end{equation}
Since $u_0\ge0$ is not a.e.\ zero and $\Kb^\beta_{T'}>0$, we also have
$U(s)(x)\ge[S_\beta(T')u_0](x)>0$.

\emph{Integration.} Fix $x\in\Omega$ and $s\in(0,T')$, and set
$G(t):=U(s)(x)+\int_s^tf(U(\tau)(x))\dd\tau$ for $t\in[s,T')$. Since
$0\le f(U(\tau)(x))\le f(\mathcal F(u;u_0)(T')(x))<\infty$, $G$ is finite,
nondecreasing and absolutely continuous, with $0<G\le U(\cdot)(x)$ by
\eqref{eq:semigroupid}. Moreover, $G'=f(U(\cdot)(x))\ge f(G)$ a.e., since $f$
is increasing. Consequently
$$
-\frac{\dd}{\dd t}F\bigl(G(t)\bigr)=\frac{G'(t)}{f(G(t))}\ \ge\ 1
\qquad\text{for a.e. }t\in(s,T'),
$$
Integrating from $s$ to $t$ and using $F>0$, we obtain
$F(U(s)(x))=F(G(s))\ge t-s+F(G(t))>t-s$. Letting $t\uparrow T'$, we get
$F(U(s)(x))\ge T'-s$. Since $U(s)\ge S_\beta(T')u_0$ and $F$ is decreasing, it
follows that $F([S_\beta(T')u_0](x))\ge T'-s$. Letting $s\downarrow0$ proves
\eqref{eq:necessary} for $x\in\Omega$.

If $u_0\in L^1$, then $(x,t)\mapsto[S_\beta(t)u_0](x)$ is continuous and
positive on $\R^d\times(0,\infty)$ by Lemma~\ref{lem:kernel}(v), and $F$ is
continuous on $(0,\infty)$. Since \eqref{eq:necessary} holds on a dense subset
of $\R^d\times(0,T)$, we conclude that it holds everywhere.
\end{proof}

\subsection{Two consequences of hypothesis \texorpdfstring{$(\mathrm G)$}{(G)}}

\begin{Lemma}\label{lem:Fprops}
Let $f$ satisfy $(\mathrm C)$ and $(\mathrm G)$. Then
\begin{enumerate}[label=\textup{(\roman*)}]
\item $F$ is rapidly varying of index $-\infty$ at infinity, and $F^{-1}$ is
slowly varying at $0^{+}$;
\item for every $q>1$ the function $g_q(\sigma):=F^{-1}(\sigma^{q})$ is
decreasing on $(0,\infty)$ and there is $\sigma_*=\sigma_*(q)>0$ such that
$g_q$ is convex on $(0,\sigma_*)$;
\item $f(u)/u\to\infty$ as $u\to\infty$.
\end{enumerate}
\end{Lemma}

\begin{proof}
By \cite[Lemma 6.1(i)--(ii)]{Laister}, $\log f$ is regularly varying of index
$\rho\in(0,\infty]$, and $u^{-k}f(u)\to\infty$ for every $k>0$. The latter
property gives (iii). For $\lambda>1$, $\log f(\lambda u)-\log f(u)=\log
f(u)\bigl(\frac{\log f(\lambda u)}{\log f(u)}-1\bigr)\to\infty$, since $\log
f(u)\to\infty$ and the ratio tends to $\lambda^\rho>1$ (or to $\infty$). Thus
$f(\lambda u)/f(u)\to\infty$. Moreover $(\log f)'=f'/f$ is regularly varying
of index $\rho-1$, so $(f'/f)(\lambda u)/(f'/f)(u)$ tends to
$\lambda^{\rho-1}>0$ (or to $\infty$ if $\rho=\infty$). Hence
$f'(\lambda u)/f'(u)\to\infty$ for $\lambda>1$, and consequently
$f'(\lambda u)/f'(u)\to0$ for $\lambda<1$. By \eqref{eq:limitone}, $F\sim1/f'$
at infinity, whence $F(\lambda u)/F(u)\to0$ for $\lambda>1$ and $\to\infty$ for
$\lambda<1$. In other words, $F$ is rapidly varying of index $-\infty$.

For the second half of (i), fix $\lambda>1$ and let $\sigma\to0^{+}$. Writing
$u:=F^{-1}(\sigma)\to\infty$ and $c_\sigma:=F^{-1}(\lambda\sigma)/F^{-1}(\sigma)
\in(0,1]$, \eqref{eq:Finv} gives $F(c_\sigma u)=\lambda F(u)$. If
$\liminf_{\sigma\to0}c_\sigma=c<1$, pick $c'\in(c,1)$; along a sequence
$\sigma_j\to0$ with $c_{\sigma_j}\le c'$, we get $F(c_{\sigma_j}u)\ge F(c'u)$,
and therefore $\lambda\ge F(c'u)/F(u)\to\infty$, a contradiction. Hence
$c_\sigma\to1$. The case $\lambda<1$ follows by replacing $\sigma$ with
$\lambda^{-1}\sigma$. Hence $F^{-1}$ is slowly varying at $0^{+}$.

(ii) $g_q$ is decreasing because $F^{-1}$ is. By \eqref{E7.3},
$g_q'(\sigma)=-q\sigma^{q-1}f\bigl(F^{-1}(\sigma^{q})\bigr)$, which is $C^1$ since
$f\in C^1$. A further differentiation gives
$$
g_q''(\sigma)=q\,\sigma^{q-2}\,f\bigl(F^{-1}(\sigma^{q})\bigr)\,
\Bigl[q\,\sigma^{q}f'\bigl(F^{-1}(\sigma^{q})\bigr)-(q-1)\Bigr].
$$
By \eqref{E7.2}, the quantity $\sigma^{q}f'(F^{-1}(\sigma^{q}))$ tends to $1$
as $\sigma\to0^{+}$. Hence the bracket tends to $1$, and $g_q''>0$ on some
interval $(0,\sigma_*)$.
\end{proof}

\subsection{The singular data}

For $0<\alpha<2\beta$, put
\begin{equation}\label{eq:calpha}
c_{d,\alpha,\beta}:=\int_{\R^d}|z|^{\alpha}p_\beta(z,1)\dd z
=2^{\alpha}\,\frac{\Gamma\bigl(\frac{d+\alpha}{2}\bigr)}{\Gamma\bigl(\frac d2\bigr)}
\,\frac{\Gamma\bigl(1-\frac{\alpha}{2\beta}\bigr)}{\Gamma\bigl(1-\frac\alpha2\bigr)},
\qquad
\gamma_0(d,\beta):=\sup_{0<\alpha<2\beta}c_{d,\alpha,\beta}^{-2\beta/\alpha}.
\end{equation}
The value of $c_{d,\alpha,\beta}$ follows from \eqref{eq:freekernel},
Lemma~\ref{lem:stable}(ii) with $r=\alpha/2<\beta$, and the moment of the
Gaussian
$$
\int_{\R^d}|z|^\alpha G(z,s)\dd z=(2s)^{\alpha/2}\,2^{\alpha/2}\,
\frac{\Gamma\bigl(\frac{d+\alpha}{2}\bigr)}{\Gamma\bigl(\frac d2\bigr)} .
$$ In particular,
$c_{d,\alpha,\beta}<\infty$ precisely because $\alpha<2\beta$, and
$\gamma_0(d,\beta)\in(0,\infty)$ depends only on $d$ and $\beta$.

\begin{Lemma}\label{L7.1}
Let $0<\beta\le1$, $\varrho>0$, and let $f$ satisfy $(\mathrm C)$ and
$(\mathrm G)$. For every $\gamma\in(0,\gamma_0(d,\beta))$ and every $R>0$,
problem \eqref{eq:main} possesses no nonnegative integral solution in $Q_T$,
for any $T>0$, with initial data
\begin{equation}\label{eq:data}
u_0(x)=F^{-1}\bigl(\gamma|x|^{2\beta}\bigr)\chi_{B_R}(x),
\qquad x\in\R^d\setminus\{0\} .
\end{equation}
Moreover $u_0\in L^1$.
\end{Lemma}

\begin{proof}
\emph{Integrability.} By Lemma~\ref{lem:Fprops}(i), $r\mapsto
F^{-1}(\gamma r^{2\beta})$ is slowly varying at $0^+$, hence $o(r^{-1/2})$ as
$r\to0$ \cite[Proposition 1.3.6]{BGT}. We conclude that $u_0\in L^1$, because
$d-1-\frac12>-1$.

\emph{Reductions.} Fix $\gamma<\gamma_0(d,\beta)$ and choose $\alpha\in(0,2\beta)$
with $\gamma\,c_{d,\alpha,\beta}^{2\beta/\alpha}<1$. Set $q:=2\beta/\alpha>1$,
$c:=c_{d,\alpha,\beta}$ and $g:=g_q=F^{-1}(\cdot^{\,q})$, so that
$u_0(x)=g\bigl(\gamma^{1/q}|x|^{\alpha}\bigr)\chi_{B_R}(x)$. Let $\sigma_*$ be as
in Lemma \ref{lem:Fprops}(ii). By Corollary~\ref{cor:datacomp}, it suffices to
prove the claim for small $R$; we may thus assume $\gamma^{1/q}R^{\alpha}<\sigma_*$.
Suppose, for a contradiction, that a nonnegative integral solution exists in
$Q_T$. Since $u_0\in L^1$, Lemma \ref{L7.2} gives
\begin{equation}\label{eq:nec0}
F\bigl([S_\beta(t)u_0](0)\bigr)\ge t,\qquad t\in(0,T).
\end{equation}

\emph{The mass defect.} Put $m(t):=\int_{B_R}\Kb^\beta_t(0,y)\dd y\in(0,1]$. We
claim that
\begin{equation}\label{eq:massdefect}
1-m(t)\le C_R\,t,\qquad t>0,\qquad
C_R:=2\max\{1,d\varrho\}+2\max\{1,2d/R^2\}.
\end{equation}
Indeed, $1-m(t)=\bigl(1-\int_{\R^d}\Kb^\beta_t(0,y)\dd y\bigr)+\int_{|y|\ge
R}\Kb^\beta_t(0,y)\dd y$. By \eqref{eq:mass1} with $y=0$, symmetry and
\eqref{E7.5}, the first term equals
$\int_0^\infty\bigl(1-(\cosh2\omega s)^{-d/2}\bigr)\eta^\beta_t(s)\dd s$. Since
$\cosh z\le e^{z^2/2}$, we have $1-(\cosh z)^{-d/2}\le\min\{1,dz^2/4\}$. With
$z=2\omega s$, this is at most $\min\{1,d\varrho s^2\}\le\max\{1,d\varrho\}
\min\{1,s\}$, and Lemma~\ref{lem:stable}(iii) bounds the first term by
$2\max\{1,d\varrho\}t$. For the second term, Lemma \ref{lem:kernel}(ii) and
Chebyshev's inequality for the Gaussian $G(\cdot,s)$, whose second moment is
$2ds$, give
\begin{align*}
\int_{|y|\ge R}\Kb^\beta_t(0,y)\dd y&\le\int_0^\infty\Bigl(\int_{|y|\ge R}G(y,s)
\dd y\Bigr)\eta^\beta_t(s)\dd s\\
&\le\int_0^\infty\min\Bigl\{1,\frac{2ds}{R^2}\Bigr\}\eta^\beta_t(s)\dd s
\le2\max\Bigl\{1,\frac{2d}{R^2}\Bigr\}t ,
\end{align*}
again by Lemma~\ref{lem:stable}(iii). This proves \eqref{eq:massdefect}.

\emph{A lower bound at the origin.} Let $h(y):=\gamma^{1/q}|y|^{\alpha}$, and
let $\nu_t$ be the measure with density $\Kb^\beta_t(0,\cdot)\chi_{B_R}$,
whose total mass is $m(t)$. Since $h$ maps $B_R\setminus\{0\}$ into
$(0,\sigma_*)$, where $g$ is convex, Jensen's inequality for the probability
measure $\nu_t/m(t)$ yields
$$
[S_\beta(t)u_0](0)=\int g\circ h\,\dd\nu_t
\ \ge\ m(t)\,g\Bigl(\frac{1}{m(t)}\int h\,\dd\nu_t\Bigr).
$$
By Lemma \ref{lem:kernel}(ii) and the self-similarity of $p_\beta$,
$$
\int h\,\dd\nu_t\le\gamma^{1/q}\int_{\R^d}p_\beta(y,t)|y|^{\alpha}\dd y
=c\,\gamma^{1/q}\,t^{\frac{\alpha}{2\beta}} .
$$
Let $\eta\in(0,1)$. For $0<t\le\eta/C_R$, we have $m(t)\ge1-\eta$ by
\eqref{eq:massdefect}. Since $g$ is decreasing and $\alpha q/(2\beta)=1$, we
obtain
$$
[S_\beta(t)u_0](0)\ \ge\ m(t)\,
g\Bigl(\frac{c}{1-\eta}\,\gamma^{1/q}t^{\frac{\alpha}{2\beta}}\Bigr)
=m(t)\,F^{-1}\bigl(c_1\gamma t\bigr),
\qquad c_1:=\Bigl(\frac{c}{1-\eta}\Bigr)^{q} .
$$

\emph{Absorbing the mass defect.} Let $c_2>c_1$. We claim that
\begin{equation}\label{eq:absorb}
m(t)\,F^{-1}(c_1\gamma t)\ \ge\ F^{-1}(c_2\gamma t)
\qquad\text{for all small }t>0 .
\end{equation}
Writing $\sigma:=\gamma t$ and using \eqref{E7.3} together with the
monotonicity of $\tau\mapsto f(F^{-1}(\tau))$,
$$
F^{-1}(c_1\sigma)-F^{-1}(c_2\sigma)=\int_{c_1\sigma}^{c_2\sigma}
f\bigl(F^{-1}(\tau)\bigr)\dd\tau\ \ge\ (c_2-c_1)\,\sigma\,
f\bigl(F^{-1}(c_2\sigma)\bigr) ,
$$
whence, with $W_i:=F^{-1}(c_i\sigma)$,
$$
1-\frac{W_2}{W_1}
\ \ge\ (c_2-c_1)\,\sigma\,\frac{f(W_2)}{W_2}\,\frac{W_2}{W_1} .
$$
By Lemma \ref{lem:Fprops}(i), $W_2/W_1\to1$ as $\sigma\to0^{+}$, and by
Lemma~\ref{lem:Fprops}(iii), $f(W_2)/W_2\to\infty$ since $W_2\to\infty$.
Hence, for $t$ small, the right-hand side exceeds $C_R\sigma/\gamma=C_Rt\ge
1-m(t)$, and \eqref{eq:absorb} follows.

\emph{Conclusion.} By \eqref{eq:nec0}, the two lower bounds and the
monotonicity of $F$, we get, for all small $t>0$,
$$
t\ \le\ F\bigl([S_\beta(t)u_0](0)\bigr)
\ \le\ F\bigl(F^{-1}(c_2\gamma t)\bigr)=c_2\gamma t ,
$$
so that $c_2\gamma\ge1$. However, $\gamma c^{q}<1$. We may therefore choose
$\eta$ so small, and then $c_2>c_1=(c/(1-\eta))^q$ so close to $c_1$, that
$c_2\gamma<1$, a contradiction.
\end{proof}

\begin{Lemma}\label{L7.3}
Under the hypotheses of Lemma \ref{L7.1}, for every $\gamma\in(0,\gamma_0)$
problem \eqref{eq:main} possesses no nonnegative integral solution with the
untruncated data $u_c(x)=F^{-1}(\gamma|x|^{2\beta})$, $x\ne0$.
\end{Lemma}

\begin{proof}
Assume that such a solution exists. Then Corollary~\ref{cor:datacomp}, with
$\mu_1=u_c$ and $\mu_2=u_c\chi_{B_R}$, contradicts Lemma~\ref{L7.1}.
\end{proof}

\begin{Remark}\rm\label{rem:onwhatisproved}
Lemmas \ref{L7.2}, \ref{lem:Fprops}, \ref{L7.1} and \ref{L7.3} are proved here
in full, for all $0<\beta\le1$ and all $\varrho>0$. The only external inputs
are \cite[Lemma 6.1(i)--(ii),(iv)]{Laister}. These are statements about
regularly varying functions, and they involve neither the equation nor the
operator.
In particular we do not appeal to \cite{Fujishima2024}, on which the
corresponding step of \cite[Section 6.3]{Laister} relies.

The threshold $\gamma_0(d,\beta)$ depends neither on $\varrho$ nor on $f$. By
Lemma~\ref{lem:freecomp}, a nonnegative solution of the \emph{free} problem
$\partial_tu+(-\Delta)^\beta u=f(u)$ with data \eqref{eq:data} would produce
one of \eqref{eq:main}. Hence Lemma~\ref{L7.1} also yields nonexistence for
the free fractional heat equation with data \eqref{eq:data}, for every
$\gamma<\gamma_0(d,\beta)$. This gives an explicit value of the constant in the
necessity part of \cite[Theorem 1.1(a)]{Fujishima2024} for this profile (with
$\theta=2\beta$).

The exponent $2\beta$ in \eqref{eq:data} is the one for which the profile is
invariant under the parabolic scaling of the free fractional operator at small
scales: for $f(u)=e^u$, for instance, $F^{-1}(\gamma|x|^{2\beta})=
-\log\gamma-2\beta\log|x|$, and the equation
$\partial_tu+(-\Delta)^\beta u=e^u$ is invariant under
$u\mapsto u(\mu x,\mu^{2\beta}t)+2\beta\log\mu$. In the Jensen step of the
proof of Lemma~\ref{L7.1} it enters through the condition $\alpha q=2\beta$
with $\alpha<2\beta$, which reflects the fact that the $\beta$-stable
subordinator has finite moments exactly of order less than $\beta$. For
$\beta<1$, the profile $F^{-1}(\gamma|x|^2)$ is larger than
$F^{-1}(\gamma'|x|^{2\beta})$ near the origin for every $\gamma'>0$. Nonexistence
for this profile therefore also follows from Corollary~\ref{cor:datacomp}, but
it is not the dilation-critical profile.
\end{Remark}

\begin{Remark}\rm[Dilation-critical singularity]\label{rem:DCS}
The nonexistence result of Lemma~\ref{L7.3} has a sharp counterpart. Let
$u_c(x)=F^{-1}(\gamma|x|^{2\beta})$ be the untruncated profile. Then there
exists $\gamma^\ast=\gamma^\ast(d,\beta,\varrho,f)$ with
$\gamma_0(d,\beta)\le\gamma^\ast<\infty$ such that \eqref{eq:main} with data
$u_c$ has no nonnegative integral solution for $\gamma<\gamma^\ast$, and has
one, belonging to $L^\infty_{\mathrm{loc}}((0,T),L^\infty)$, for
$\gamma>\gamma^\ast$. Since $u_c(x/\lambda)=F^{-1}(\gamma\lambda^{-2\beta}
|x|^{2\beta})$, this says that $F^{-1}(\gamma^\ast|x|^{2\beta})$ is a
\emph{dilation-critical singularity} for \eqref{eq:main} in the sense of
\cite[Definition 1.2]{Fujishima2024}. The proof combines the existence part of
\cite{Fujishima2024} with Lemma~\ref{lem:freecomp}, as follows.

We first check that $f$, restricted to $[0,\infty)$, satisfies the hypotheses
of \cite{Fujishima2024} with $\theta=2\beta\in(0,2]$, which is admissible
since \cite{Fujishima2024} covers $0<\theta\le2$. By $(\mathrm C)$, $f$ is
continuous, positive on $(0,\infty)$ and increasing on $[0,\infty)$, which is
condition (M) there. By $(\mathrm C)$ and $(\mathrm G)$, $f\in C^1$ and
$F(u)<\infty$ for $u>0$, which gives the superlinearity condition (S), and
$f$ is convex on $[0,\infty)$, which gives (C). Finally, $(\mathrm G)$ and
\eqref{eq:limitone} show that $q_f:=\lim_{u\to\infty}f'(u)F(u)$ exists and
equals $1$, which is (L). The relation $1/p_f+1/q_f=1$ then gives
$p_f=\infty>p_\theta=1+\theta/d$, so $f$ is supercritical in the sense of
\cite{Fujishima2024}.

Next, \cite[Theorem 1.1(b)]{Fujishima2024} asserts the following. If $\mu$ is
nonnegative and bounded away from the origin, and if
$1/F(\mu(x))\le\varepsilon|x|^{-\theta}$ near the origin for a suitable
$\varepsilon=\varepsilon(d,\theta,f)>0$, then the free problem
$\partial_tu+(-\Delta)^{\theta/2}u=f(u)$ possesses, for some $T>0$, a solution
$U\in L^\infty_{\mathrm{loc}}((0,T),L^\infty)$. Here a solution is a
nonnegative, measurable, almost everywhere finite function satisfying the
Duhamel formula with the kernel of $e^{-t(-\Delta)^{\theta/2}}$, which is
$p_\beta$ for $\theta=2\beta$; this is exactly a nonnegative integral solution
of the free problem in the sense of Definition~\ref{def:solutions}(i). The
data $u_c$ is bounded away from the origin and satisfies
$1/F(u_c(x))=\gamma^{-1}|x|^{-2\beta}$. Hence, for $\gamma\ge1/\varepsilon$,
the free problem has such a solution $U$, and Lemma~\ref{lem:freecomp}
provides a nonnegative integral solution $u\le U$ of \eqref{eq:main} with data
$u_c$; in particular $u\in L^\infty_{\mathrm{loc}}((0,T),L^\infty)$.

Finally, let $\gamma^\ast$ be the infimum of all $\gamma>0$ for which
\eqref{eq:main} with data $u_c$ has a nonnegative integral solution. Since
$F^{-1}$ is decreasing, $u_c$ decreases as $\gamma$ increases, and
Corollary~\ref{cor:datacomp} shows that solvability for some $\gamma$ implies
solvability for every larger $\gamma$. By the previous step,
$\gamma^\ast\le1/\varepsilon<\infty$, and by Lemma~\ref{L7.3},
$\gamma^\ast\ge\gamma_0(d,\beta)$. This proves the claim. The same
monotonicity argument applies to the truncated data \eqref{eq:data}.
\end{Remark}

\subsection{Proof of Theorem \ref{thm:critical}}

We use the following consequences of $(\mathrm G)$, which are
\cite[Lemma 6.1 and Remark 6.1]{Laister} and concern only the asymptotics of
$f$. We write $fF$ for the function $u\mapsto f(u)F(u)$.

\begin{Lemma}[{\cite[Lemma 6.1, Remark 6.1]{Laister}}]\label{lem:RV}
Let $f$ satisfy $(\mathrm C)$ and $(\mathrm G)$. Then $\log f$ is regularly
varying of index $\rho$, $f^{k}\approx f$ for every $k>0$, $u^{-k}f(u)\to\infty$
for every $k>0$, $\lim_{u\to\infty}f'(u)F(u)=1$, and $fF$ is regularly varying
of index $1-\rho\in[-\infty,1)$. Consequently, for $\varepsilon,\lambda\in(0,1)$
and $M>0$ there exist $y_1,y_2,y_3>0$ such that
\begin{equation}\label{eq:RV1}
1-\varepsilon\le f'(y)F(y)\le1+\varepsilon,\qquad y\ge y_1,
\end{equation}
\begin{equation}\label{eq:RV2}
\frac{f(\lambda y)F(\lambda y)}{\lambda f(y)F(y)}\ \ge\
\begin{cases}
\lambda^{-\rho}-\varepsilon, & \rho<\infty,\ y\ge y_2,\\[3pt]
M, & \rho=\infty,\ y\ge y_3 .
\end{cases}
\end{equation}
\end{Lemma}

\begin{Lemma}\label{lem:ICinf}
Let $f$ satisfy $(\mathrm C)$ and $(\mathrm G)$, and let $\Phi\in\Nf$ with
$\Phi\gtrsim f$ at infinity. Then there is $\lambda_1>0$ such that
$\int_1^\infty x^{-p_\ast}f'\bigl(\lambda_1\Phi^{-1}(x)\bigr)\dd x<\infty$.
\end{Lemma}

\begin{proof}
The function $f$ is a strictly increasing bijection of $[0,\infty)$, and
$\Phi\gtrsim f$ gives $K,y_0>0$ with $f(y)\le\Phi(Ky)$ for $y\ge y_0$. Since
$\Phi$ is a bijection, we infer that $\Phi^{-1}(x)\le Kf^{-1}(x)$ for
$x\ge f(y_0)$. Let
$\bar\lambda\in(0,1)$ and $\varepsilon\in(0,1)$ be fixed below and put
$\lambda_1:=\bar\lambda/K$. For $a$ large, substituting $x=f(y)$ and using the
right inequality in \eqref{eq:RV1} at $y$ and at $\bar\lambda y$,
\begin{equation}\label{eq:C1}
\int_a^{\infty}x^{-p_\ast}f'\bigl(\lambda_1\Phi^{-1}(x)\bigr)\dd x
\le\int_b^{\infty}f(y)^{-p_\ast}f'(\bar\lambda y)f'(y)\dd y
\le(1+\varepsilon)^2\int_b^{\infty}h(y)\dd y ,
\end{equation}
where $b=f^{-1}(a)$ and $h(y):=\bigl[f(y)^{p_\ast}F(y)F(\bar\lambda y)\bigr]^{-1}$.
Using $F'=-1/f$ and the left inequality in \eqref{eq:RV1}, we compute
\begin{align*}
h'(y)&=-h(y)\Bigl(p_\ast\frac{f'(y)}{f(y)}-\frac{1}{f(y)F(y)}
-\frac{\bar\lambda}{f(\bar\lambda y)F(\bar\lambda y)}\Bigr)\\
&\le-\frac{\bar\lambda h(y)}{f(\bar\lambda y)F(\bar\lambda y)}
\Bigl(\bigl(\tfrac{2\beta}{d}-\varepsilon p_\ast\bigr)
\frac{f(\bar\lambda y)F(\bar\lambda y)}{\bar\lambda f(y)F(y)}-1\Bigr),
\end{align*}
since $p_\ast(1-\varepsilon)-1=\tfrac{2\beta}{d}-\varepsilon p_\ast$. Fix
$M>d/(2\beta)$, let $0<\varepsilon<2\beta/(dp_\ast)$, and set
$$
\alpha_1(\rho):=
\begin{cases}
\bigl(\tfrac{2\beta}{d}-\varepsilon p_\ast\bigr)
(\bar\lambda^{-\rho}-\varepsilon)-1, & \rho\in(0,\infty),\\[4pt]
\bigl(\tfrac{2\beta}{d}-\varepsilon p_\ast\bigr)M-1, & \rho=\infty .
\end{cases}
$$
For $\rho<\infty$, we choose $\bar\lambda\in(0,1)$ so small that
$\alpha_1(\rho)>0$. For $\rho=\infty$, we choose $\varepsilon$ so small that
$\alpha_1(\rho)>0$, and $\bar\lambda\in(0,1)$ arbitrary. Then \eqref{eq:RV2}
and the right inequality in \eqref{eq:RV1} give, for $y$ large,
\begin{align*}
h'(y)&\le-\frac{\alpha_1(\rho)\bar\lambda h(y)}{f(\bar\lambda y)F(\bar\lambda y)}
\le-\frac{\alpha_1(\rho)}{1+\varepsilon}\,
\bar\lambda\frac{f'(\bar\lambda y)}{f(\bar\lambda y)}\,h(y)\\
&=-A_1\bigl(\log f(\bar\lambda y)\bigr)'h(y),
\qquad A_1:=\frac{\alpha_1(\rho)}{1+\varepsilon}>0 ,
\end{align*}
and integration yields $h(y)\le C[f(\bar\lambda y)]^{-A_1}$ for $y$ large. Since
$f$ grows faster than any polynomial, $h$ is integrable at infinity, and we
conclude that \eqref{eq:C1} is finite. On $[1,a]$, the integrand is bounded.
\end{proof}

\begin{proof}[Proof of Theorem \ref{thm:critical}\textup{(a)(i)}]
Let $\lambda_1$ be as in Lemma~\ref{lem:ICinf}. By $(\Lambda)$,
$$
\ell\bigl(\lambda_1\Phi^{-1}(x)\bigr)\le Cf'\bigl(\lambda_1\Phi^{-1}(x)\bigr)
\qquad\text{whenever }\lambda_1\Phi^{-1}(x)>\kappa_0,
$$
that is, for $x$ large. On the remaining bounded interval, the integrand in
\eqref{eq:IC-infty} is bounded. Hence \eqref{eq:IC-infty} holds with
$\lambda=\lambda_1$, and Theorem~\ref{thm:local} applies.
\end{proof}

\begin{proof}[Proof of Theorem \ref{thm:critical}\textup{(a)(ii)}]
We follow \cite[Section 6.2]{Laister}. First, there are $C_1,\mu>0$ with
$f(u)^r\le C_1\Phi(\mu u)$ for all $u\ge0$. Near $0$, this is the hypothesis
$\Phi\gtrsim f^r$. At infinity, it follows from $f^r\approx f$
(Lemma~\ref{lem:RV}) and $\Phi\gtrsim f$, by transitivity of $\lesssim$. On
the remaining compact interval $[u_1,u_2]\subset(0,\infty)$, the quotient
$f(u)^r/\Phi(\mu u)$ is bounded, because $\Phi$ is positive and increasing.

Let $v$ be an $M_\Phi$-mild solution on $[0,T)$ with data $u_0$, and fix
$T_1<T$. The set $\mathcal V:=v([0,T_1])$ is compact in $M_\Phi$, so there is
$A_{\mathcal V}<\infty$ with $\int\Phi(4\mu|\psi|)\le A_{\mathcal V}$ for
$\psi\in\mathcal V$ (cover $\mathcal V$ by finitely many balls of radius
$1/(8\mu)$ and use convexity). By uniform continuity, there is $\tau_0>0$ with
$\|v(s')-v(s)\|_\Phi\le1/(4\mu)$ whenever $s,s'\in[0,T_1]$, $|s-s'|\le\tau_0$.
For such $s\le s'$, the convexity of $\Phi$ gives
\begin{align*}
\|f(v(s'))\|_{L^r}^r&\le C^r\int f(|v(s')|)^r
\le C^rC_1\int\Phi(\mu|v(s')|)\\
&\le\frac{C^rC_1}{2}\int\bigl[\Phi(2\mu|v(s')-v(s)|)+\Phi(2\mu|v(s)|)\bigr]
\le C^rC_1\bigl(1+A_{\mathcal V}\bigr).
\end{align*}
Let $s\in[0,T_1)$ be such that Lemma~\ref{lem:restart}(i) applies (or $s=0$).
For $0<t\le\min\{\tau_0,T_1-s\}$,
$v(s+t)=S_\beta(t)v(s)+\int_0^tS_\beta(t-\sigma)f(v(s+\sigma))\dd\sigma$, and by
Lemma~\ref{lem:minkowski} and Proposition~\ref{prop:LpLq} (with the pair of
exponents $(r,\infty)$),
$$
\Bigl\|\int_0^tS_\beta(t-\sigma)f(v(s+\sigma))\dd\sigma\Bigr\|_\infty
\le C\int_0^t(t-\sigma)^{-\frac{d}{2\beta r}}\dd\sigma\
\sup_{s\le s'\le s+\tau_0}\|f(v(s'))\|_{L^r}
\le C\,\tau_0^{\,1-\frac{d}{2\beta r}} ,
$$
since $r>d/(2\beta)$, while $\|S_\beta(t)v(s)\|_\infty\le
C_\infty\alpha(t)\|v(s)\|_\Phi$. Hence $v$ is bounded on $\R^d\times[s+\epsilon,
s+\tau_0]$ for every $\epsilon>0$. Since the admissible $s$ are dense, we infer
that $v\in L^\infty_{\rm loc}((0,T_1),L^\infty)$, and $T_1<T$ was arbitrary.
Thus $v$ is a bounded $M_\Phi$-mild solution, and Theorem~\ref{thm:local}(b)
shows that $v=u(\cdot;u_0)$ and $T\le T_{u_0}$.
\end{proof}

\begin{proof}[Proof of Theorem \ref{thm:critical}\textup{(b)}]
Let $\gamma\in(0,\gamma_0(d,\beta))$ and let $u_0$ be the initial data
\eqref{eq:data}, with $R>0$ to be fixed. By Lemma \ref{L7.1}, problem
\eqref{eq:main} has no nonnegative integral solution with this data. It only
remains to check that $u_0\in L_\Phi$ for a suitable choice of $R$.

Since $\Phi\lesssim f$, there are $K,u_1>0$ with $\Phi(u)\le f(Ku)$ for
$u\ge u_1$. Let $\lambda\in(0,1)$ be fixed below, put $\lambda_0:=\lambda/K$,
and choose $R>0$ so small that $\lambda_0F^{-1}(\gamma R^{2\beta})\ge u_1$,
which is possible by \eqref{eq:Finv}. As $u_0$ is radially decreasing, we have
$\lambda_0u_0(x)\ge u_1$ on $B_R\setminus\{0\}$, and therefore
\begin{align*}
\int_{\R^d}\Phi\bigl(\lambda_0u_0\bigr)\dd x
&\le\int_{B_R}f\bigl(\lambda F^{-1}(\gamma|x|^{2\beta})\bigr)\dd x
=|\mathbb S^{d-1}|\int_0^{R}f\bigl(\lambda F^{-1}(\gamma r^{2\beta})\bigr)r^{d-1}\dd r\\
&=\frac{|\mathbb S^{d-1}|}{2\beta\gamma^{d/(2\beta)}}
\int_b^{\infty}\frac{f(\lambda s)}{f(s)}\bigl[F(s)\bigr]^{\nu}\dd s ,
\end{align*}
with $\nu:=\frac{d}{2\beta}-1$ and $b:=F^{-1}(\gamma R^{2\beta})$, where we
substituted $\gamma r^{2\beta}=F(s)$ and used $F'=-1/f$. Writing
$g(s):=f(\lambda s)f(s)^{-1}[F(s)]^{\nu}$ and using \eqref{eq:RV1}, we compute
\begin{align*}
g'(s)&=-g(s)\Bigl(\frac{f'(s)}{f(s)}-\lambda\frac{f'(\lambda s)}{f(\lambda s)}
+\frac{\nu}{f(s)F(s)}\Bigr)\\
&\le-\frac{\lambda g(s)}{f(\lambda s)F(\lambda s)}
\Bigl(\bigl(\tfrac{d}{2\beta}-\varepsilon\bigr)
\frac{f(\lambda s)F(\lambda s)}{\lambda f(s)F(s)}-(1+\varepsilon)\Bigr),
\end{align*}
the constant $\tfrac{d}{2\beta}-\varepsilon$ arising as $1-\varepsilon+\nu$.
Let $M>0$ and $\varepsilon\in(0,\min\{1,d/(2\beta)\})$, and set
$$
\alpha_2(\rho):=
\begin{cases}
\bigl(\tfrac{d}{2\beta}-\varepsilon\bigr)(\lambda^{-\rho}-\varepsilon)
-(1+\varepsilon), & \rho\in(0,\infty),\\[4pt]
\bigl(\tfrac{d}{2\beta}-\varepsilon\bigr)M-(1+\varepsilon), & \rho=\infty .
\end{cases}
$$
For $\rho<\infty$, we choose $\lambda\in(0,1)$ small; for $\rho=\infty$, we
choose $M$ large and $\lambda\in(0,1)$ arbitrary. In both cases
$\alpha_2(\rho)>0$. Then, exactly as in the proof of Lemma~\ref{lem:ICinf}, we
get $g(s)\le C[f(\lambda s)]^{-A_2}$ for $s$ large, with
$A_2=\alpha_2(\rho)/(1+\varepsilon)>0$. Since $f$ grows faster than any
polynomial, $g$ is integrable at infinity. We conclude that
$\int\Phi(\lambda_0u_0)\dd x<\infty$, that is, $u_0\in L_\Phi$.
\end{proof}

\begin{Remark}\rm\label{rem:nu}
The computation above is valid for every $d\ge1$ and $\beta\in(0,1]$, since only
the positivity of $1-\varepsilon+\nu=\frac{d}{2\beta}-\varepsilon$ is used. In
particular, no separate treatment is needed when $\nu<0$, which happens only
for $d=1$ and $\beta\in(\frac12,1]$. The local integrability of $u_0$ required
by Lemma~\ref{L7.2} was established in Lemma~\ref{L7.1}.
\end{Remark}

\begin{Remark}\rm\label{rem:notinM}
Assume that $f$ also satisfies $(\Lambda)$ and that $\Phi\in\Nf$ satisfies
$\Phi\approx f$ at infinity. Then the initial data $u_0$ constructed above lies in
$L_\Phi\setminus M_\Phi$. Otherwise, Theorem~\ref{thm:critical}(a)(i) and
Theorem~\ref{thm:local}(a) would provide a nonnegative (since $f$ is
positone) bounded $M_\Phi$-mild, hence integral, solution. In particular,
$M_\Phi\neq L_\Phi$, which is consistent with the fact that $\Phi$ cannot satisfy the
$\Delta_2$-condition when $\Phi\approx f$ is of faster than polynomial growth.
\end{Remark}

\begin{proof}[Proof of Theorem \ref{thm:critical}\textup{(c)} and Corollary
\ref{cor:C}]
Since $(\Lambda_0)$ implies $(\Lambda)$, part (a) gives \eqref{eq:IC-infty}
and the conclusions of Theorem~\ref{thm:local}. For the global statement, we
must produce $\lambda>0$ with
$$I_0=\frac{2\beta}{d}\int_0^\infty x^{-p_\ast}\ell(\lambda\Phi^{-1}(x))\dd x<1.$$
By $(\Lambda_0)$, it suffices to show that $\int_0^{\infty}H(\lambda,x)\dd x\to0$
as $\lambda\to0$, where $H(\lambda,x):=x^{-p_\ast}f'\bigl(\lambda\Phi^{-1}(x)\bigr)$.
Let $\lambda_1$ be as in Lemma~\ref{lem:ICinf} and
$\lambda':=\min\{\lambda_0,\lambda_1\}$. Since $f'$ is increasing, $0\le
H(\lambda,x)\le H(\lambda',x)$ for $\lambda\le\lambda'$, and $H(\lambda',\cdot)$
is integrable on $(0,1)$ by \eqref{eq:C-b} and on $(1,\infty)$ by
Lemma~\ref{lem:ICinf}. Since $f'$ is continuous with $f'(0)=0$, we have
$H(\lambda,x)\to0$ pointwise as $\lambda\to0$, and dominated convergence gives
$\int_0^\infty H(\lambda,x)\dd x\to0$. Hence \eqref{eq:IC-zero} holds for all
sufficiently small $\lambda\in(0,\lambda']$, and Corollary \ref{thm:global}
applies.

For Corollary \ref{cor:C}, it remains to verify \eqref{eq:C-b}. As $f'$ is
regularly varying at zero of index $m-1$, we may write
$f'(x)=x^{m-1}\mu(1/x)$ for $x\in(0,1)$ with $\mu$ slowly varying at infinity;
in particular $f'(0)=0$, because $m>p_\ast>1$. From $\Phi(u)\gtrsim u^{q}$ near
zero and the bijectivity of $\Phi$, we get $\Phi^{-1}(x)\le Kx^{1/q}$ for
$x\in(0,a)$. Hence, with $\lambda=\lambda_0K$ and $b=a^{1/q}$,
\begin{align*}
\int_0^{a}x^{-p_\ast}f'\bigl(\lambda_0\Phi^{-1}(x)\bigr)\dd x
&\le\int_0^{a}x^{-p_\ast}f'\bigl(\lambda x^{1/q}\bigr)\dd x
=q\int_0^{b}y^{\,-1-\frac{2\beta q}{d}}f'(\lambda y)\dd y\\
&=q\lambda^{m-1}\int_0^{b}y^{\,m-2-\frac{2\beta q}{d}}\mu\Bigl(\frac{1}{\lambda y}\Bigr)
\dd y\\
&=q\lambda^{m-1}\int_{1/b}^{\infty}z^{-m+\frac{2\beta q}{d}}\,
\mu\Bigl(\frac z\lambda\Bigr)\dd z ,
\end{align*}
after the substitutions $x=y^{q}$ and $y=1/z$ (we assume $\lambda b<1$, which
holds for $\lambda_0$ small). The last integral converges because
$z\mapsto\mu(z/\lambda)$ is slowly varying and $-m+2\beta q/d<-1$, which is
exactly the hypothesis $q<\frac{d(m-1)}{2\beta}$; see
\cite[Proposition 1.5.10]{BGT}. On $(a,1)$, the integrand in \eqref{eq:C-b} is
bounded.
\end{proof}

\begin{proof}[Proof of Corollary \ref{cor:Cconf}]
Theorem~\ref{thm:critical}(a)(i) gives \eqref{eq:IC-infty} and the conclusions
of Theorem~\ref{thm:local}. By $(\Lambda_0)$, $\ell(r)\le Cf'(r)$ for $r>0$.
Hence \eqref{eq:dini} holds, and Theorem~\ref{thm:confined} applies. If $f'$ is
regularly varying at zero of index $m-1>0$, then
$f'(r)=r^{m-1}\mu(1/r)\le r^{(m-1)/2}$ for small $r$, and the Dini condition
holds.
\end{proof}

\section{Examples}\label{sec:examples}

We illustrate the theory on a canonical family and on concrete
nonlinearities. Throughout, $p_\ast=1+2\beta/d$ and $\Phi_c(u)=f(u)$. For
$\Phi\in\Nf$, the change of variables $x=\Phi(y)$ shows that
\eqref{eq:IC-infty} is equivalent to
\begin{equation}\label{eq:IC-y}
\int_{\Phi^{-1}(1)}^\infty\ell(\lambda y)\,[\Phi(y)]^{-p_\ast}\,\Phi'(y)\dd y<\infty ,
\end{equation}
compare \cite[Remark 3.3(d)]{Laister}. If $f$ is odd and convex on
$[0,\infty)$, then $\ell(s)=f'(s)$ for $s>0$ \cite[Remark 3.3(c)]{Laister}.
Table~\ref{tab:examples} summarises the results of this section.

\begin{table}[ht]
\centering
\small
\renewcommand{\arraystretch}{1.3}
\begin{tabular}{>{\raggedright}p{3.2cm} >{\centering}p{1.2cm}
>{\raggedright}p{3.0cm} >{\raggedright}p{3.3cm}
>{\raggedright\arraybackslash}p{3.8cm}}
\toprule
$f(u)$ ($u\ge0$, odd extension) & $\rho$ & LWP in $M_\Phi$ &
small-data GWP (Cor.~\ref{thm:global}, Cor.~\ref{cor:C}) & small-data GWP under confinement
(Cor.~\ref{cor:Cconf}, Thm.~\ref{thm:confined})\\
\midrule
$u^mJ(u)$, $J$ as in \S\ref{subsec:canonical} & $(0,\infty]$ &
$\Phi\gtrsim J$ & $m>p_\ast$, $\Phi(u)\gtrsim u^q$ near $0$,
$1<q<\frac{d(m-1)}{2\beta}$ & $m>1$, no condition near $0$\\
$u^me^{u^p}$, $p\ge1$ & $p$ & $\Phi\gtrsim e^{u^p}$ & as above & $m>1$\\
$u^m\exp(e^{u^p})$, $p\ge1$ & $\infty$ & $\Phi\gtrsim\exp(e^{u^p})$ & as
above & $m>1$\\
$u^p$, $p>1$ & -- & $L^q$, $q>\max\{1,q_c\}$ & $L^{q}\cap L^{r}$,
$1<q<q_c<r$ & $L^q$, $q>\max\{1,q_c\}$, every $p>1$\\
$u^p\log^m(1+u)$, $p\ge p_\ast$, $m\ge0$ & -- &
$L^{q_c}\log^rL$, $r\ge1$, $r>\frac{d(m+1)}{2\beta}$ & -- & same space\\
\bottomrule
\end{tabular}
\caption{Summary of the examples; $q_c=d(p-1)/(2\beta)$. Nonexistence of
nonnegative solutions for some data in $L_\Phi$ holds whenever
$\Phi\lesssim\Phi_c$ in the first three rows (Theorem~\ref{thm:critical}(b)).}
\label{tab:examples}
\end{table}

\subsection{A canonical family}\label{subsec:canonical}

Let $J:[0,\infty)\to[0,\infty)$ satisfy $J(0)>0$, $J\in C^1$ increasing and
convex, eventually $C^2$, and suppose that $(\log J)'$ is regularly varying of
index $\rho-1$ with $\rho\in(0,\infty]$ and
$\lim_{u\to\infty}(J')^2/(J''J)=1$. For $m\ge1$ let $f$ be the odd extension of
$u\mapsto u^{m}J(u)$. By \cite[Theorem 7.1 and its proof]{Laister}, whose
arguments concern only $f$ and $J$, the function $f$ satisfies $(\mathrm C)$,
$(\Lambda_0)$ and $(\mathrm G)$, and $J\approx f$ at infinity. We deduce the
following statements.
\begin{enumerate}[label=\textup{(\roman*)}]
\item if $\Phi\in\Nf$ and $\Phi\gtrsim J$, then \eqref{eq:main} is locally
well posed in $M_\Phi$ (Theorem~\ref{thm:critical}(a)(i)); if moreover
$\liminf_{u\to0}\Phi(u)/u^D>0$ for some $D\ge1$, then uniqueness holds among
all $M_\Phi$-mild solutions (Theorem~\ref{thm:critical}(a)(ii), since $f$ is
odd and Lipschitz near $0$);
\item if $\Phi\in\Y$ and $\Phi\lesssim J$, then nonexistence of nonnegative
solutions occurs for some nonnegative data in $L_\Phi$
(Theorem~\ref{thm:critical}(b));
\item if $m>p_\ast$, $\Phi\in\Nf$, $\Phi\gtrsim J$ and $\Phi(u)\gtrsim u^{q}$
near zero with $1<q<\frac{d(m-1)}{2\beta}$, then small data in $M_\Phi$ give rise
to global solutions (Corollary~\ref{cor:C});
\item if $m>1$, $\Phi\in\Nf$ and $\Phi\gtrsim J$, then small data in $M_\Phi$
give rise to global solutions decaying like $e^{-\lambda_\beta t}$
(Corollary~\ref{cor:Cconf}); indeed $f'(u)=mu^{m-1}J(u)+u^mJ'(u)\sim mJ(0)u^{m-1}$
as $u\to0$, so that the Dini condition holds.
\end{enumerate}
The significance of this family is that $J$, and hence $f$, may be prescribed
to grow arbitrarily fast. Indeed, by \cite[Lemma 7.2]{Laister}, for every $r\in
C^2([0,\infty))$ with $r''>0$ on $(0,\infty)$, $r(0)=0$, $r'(0)=1$ and
$r'(\infty)=\infty$, there is an admissible $J$ with $J(u)\ge\exp(r(u))$ for
all $u\ge0$.

\subsection{Exponential nonlinearity}

For $m\ge 1$ and $p\ge1$ consider
\begin{equation}\label{eq:expl}
\partial_tu+\Hr^{\beta}u=|u|^{m-1}u\,e^{|u|^{p}} .
\end{equation}
Here $J(u)=e^{u^{p}}$ and $\rho=p<\infty$. Local well-posedness holds in
$M_\Phi$ for every $\Phi\in\Nf$ with $\Phi\gtrsim e^{u^p}$. Nonexistence of
nonnegative solutions occurs for some data in $L_\Phi$ whenever $\Phi\in\Y$
satisfies $\Phi\lesssim e^{u^p}$. Let $m>p_\ast$ and
$1<q<\frac{d(m-1)}{2\beta}$. Then $\Phi(u)=u^{q}e^{u^{p}}$ is an $N$-function
(a product of nonnegative increasing convex functions), and we obtain global
well-posedness for small data with
$$
\|u(t)\|_{L^\infty}\le\lambda\Phi^{-1}\bigl(t^{-\frac{d}{2\beta}}\bigr)
\quad(t>0),\qquad
\|u(t)\|_{L^\infty}\le Ce^{-\lambda_\beta t}\|u_0\|_{L_\Phi}\quad(t\ge1),
$$
The second bound is much stronger than the free decay
$\lambda\Phi^{-1}(t^{-d/(2\beta)})\sim\lambda t^{-d/(2\beta q)}$. Under
confinement, small data give global solutions for every $m>1$ and every
$\Phi\in\Nf$ with $\Phi\gtrsim e^{u^p}$. The case $p=1$ is allowed here,
whereas the theory in $\exp L^{p}$ requires $p>1$, so that $e^{u^p}-1$ is an
$N$-function \cite[Section 7.3]{Laister}.

Let us compare these results with \cite{BMMexp}, where $0<\beta\le1$ and
$p>1$. There, local well-posedness in $\exp L^p_0$ is proved under a
Lipschitz condition of exponential type on $f$, and global weak-mild solutions in
$L^\infty((0,\infty),\exp L^p)$ are obtained for small data under the
conditions $m\ge p$ and $m\ge1+2p\beta/d$. Take
$\Phi(u)=e^{u^p}-1$, which is an $N$-function for $p>1$, so that
$M_\Phi=\exp L^p_0$ and $\Phi(u)\sim u^p$ near $0$. Corollary~\ref{cor:C},
with $q=p$, then gives global bounded $M_\Phi$-mild solutions for small data
in $\exp L^p_0$ whenever $m>1+2p\beta/d$, without the condition $m\ge p$.
These solutions are unique, continuous with values in $M_\Phi$, and depend
continuously on the initial data. Under confinement,
Corollary~\ref{cor:Cconf} removes the restriction on $m$ altogether: small
data in $\exp L^p_0$ give global solutions decaying like
$e^{-\lambda_\beta t}$ for every $m>1$. On the other hand, \cite{BMMexp}
covers the limiting case $m=1+2p\beta/d$ (when $m\ge p$) for weak-mild
solutions, which is not reached by Corollary~\ref{cor:C}; it also applies to
nonlinearities which are not required to satisfy the structural assumptions
$(\mathrm C)$ and $(\mathrm G)$.

\subsection{Doubly exponential nonlinearity}

For $m\ge1$, $p\ge1$ consider
$$
\partial_tu+\Hr^{\beta}u=|u|^{m-1}u\,e^{e^{|u|^{p}}} ,
$$
so that $J(u)=\exp\exp(u^{p})$ and $\rho=\infty$ (the conditions of
\S\ref{subsec:canonical} are verified in \cite[Section 7.4]{Laister}). Local
well-posedness holds for $\Phi\in\Nf$ with $\Phi\gtrsim\exp\exp(u^{p})$, and
nonexistence occurs for $\Phi\lesssim\exp\exp(u^p)$. Global well-posedness for
small data holds with $\Phi(u)=u^{q}\exp\exp(u^{p})$, provided $m>p_\ast$ and
$1<q<\frac{d(m-1)}{2\beta}$. Under confinement, it holds for every $m>1$ and
every $\Phi\in\Nf$ with $\Phi\gtrsim\exp\exp(u^p)$. For the Laplacian, such
results were first obtained in \cite[Section 7.4]{Laister}; for the
(fractional) harmonic oscillator, they appear to be new. The series expansion
underlying the classical smoothing estimates between $L^{q}$ and $\exp L^{p}$
is unavailable once $\rho=\infty$. The same applies to any finite composition of
exponentials, and to the functions $J$ of Example~7.1 in \cite{Laister}, e.g.
$J(u)=e^{e^u-1}$.

\subsection{Power nonlinearity}\label{subsec:power}

For $p>1$, let $f(u)=|u|^{p-1}u$, so that $\ell(s)=ps^{p-1}$. Take
$\Phi(u)=u^q/q$ with $q\ge1$, so that $\Phi^{-1}(x)=(qx)^{1/q}$. Then condition
\eqref{eq:IC-infty} reads
$$
\int_1^{\infty}x^{-p_\ast}x^{\frac{p-1}{q}}\dd x<\infty
\iff
q>\frac{d(p-1)}{2\beta}=:q_c .
$$
Thus Theorem \ref{thm:local} gives local well-posedness in $L^q$ for
$q>\max\{1,q_c\}$ (here $M_\Phi=L_\Phi=L^q$), with $p_\ast=1+2\beta/d$ and
$q_c=d(p-1)/(2\beta)$ in place of the classical $1+2/n$ and $n(p-1)/2$. As in
\cite[Section 7.5]{Laister}, the endpoint $q=q_c$ is not covered by the
integral condition, which is not sharp there. 
In Lebesgue spaces, however, the endpoint can be reached by the
auxiliary-norm method of Weissler \cite{weissler1980local} and
Brezis--Cazenave \cite{brezis1996nonlinear}. This method uses only $L^q$--$L^r$
smoothing estimates for small times, which hold for $S_\beta$ by
Proposition~\ref{prop:LpLq}. Let us sketch the argument for
$f(u)=|u|^{p-1}u$ with $q_c>1$, that is, $p>p_\ast$. Since
$\frac{d}{2\beta q_c}=\frac1{p-1}$, we may fix $r$ with
$\max\{p,q_c\}<r<pq_c$ and put
$\sigma:=\frac{d}{2\beta}\bigl(\frac1{q_c}-\frac1r\bigr)>0$. We work in the
space
$$
X_T:=\Bigl\{u\in L^\infty\bigl((0,T),L^{q_c}\bigr):\
\|u\|_{X_T}:=\sup_{0<t<T}\|u(t)\|_{L^{q_c}}
+\sup_{0<t<T}t^{\sigma}\|u(t)\|_{L^r}<\infty\Bigr\}.
$$
By H\"older's inequality, $\|f(u(s))\|_{L^{r/p}}\le\|u(s)\|_{L^r}^p\le
s^{-p\sigma}\|u\|_{X_T}^p$. Proposition~\ref{prop:LpLq}, applied from
$L^{r/p}$ to $L^r$ and from $L^{r/p}$ to $L^{q_c}$ (note that
$1<r/p<q_c<r$), then gives
\begin{align*}
t^{\sigma}\Bigl\|\int_0^tS_\beta(t-s)f(u(s))\dd s\Bigr\|_{L^r}
&+\Bigl\|\int_0^tS_\beta(t-s)f(u(s))\dd s\Bigr\|_{L^{q_c}}\\
&\le C\|u\|_{X_T}^p\,t^{\,1-\frac{d}{2\beta}\frac{p-1}{q_c}}
=C\|u\|_{X_T}^p .
\end{align*}
Here the time integrals converge because the conditions $r>q_c$ and
$r<pq_c$ are equivalent to $\frac{d}{2\beta}\frac{p-1}{r}<1$ and $p\sigma<1$,
respectively. At the critical exponent $q_c$, the positive power of $T$ in the
nonlinear estimate disappears, exactly as in
\cite{weissler1980local,brezis1996nonlinear}. The required smallness comes
instead from the linear term: by density of $L^{q_c}\cap L^r$ in $L^{q_c}$ and
Proposition~\ref{prop:LpLq},
$$
\sup_{0<t<T}t^{\sigma}\|S_\beta(t)u_0\|_{L^r}\longrightarrow0
\qquad\text{as } T\to0 ,
$$
for every $u_0\in L^{q_c}$. The same estimates, combined with
$|f(u)-f(v)|\le p(|u|^{p-1}+|v|^{p-1})|u-v|$, show that the Duhamel map is a
contraction on a suitable closed subset of $X_T$ for $T$ small. We obtain a
unique local solution in that subset, with $u\in C([0,T],L^{q_c})$. Thus the
endpoint $q=q_c$ is attained for power nonlinearities. Since
\eqref{eq:IC-infty} fails for $q=q_c$, this shows that the integral condition
is sufficient, but not necessary, for local solvability in $L_\Phi$.

For global solutions, let $1<q<r$ and $\Phi(u)=\max\{u^q,u^r\}$. Then
$\Phi\in\Nf$ satisfies the $\Delta_2$-condition, and
$M_\Phi=L_\Phi=L^q\cap L^r$ \cite[Section 7.5]{Laister}. Since
$\Phi^{-1}(x)=x^{1/q}$ for $x\le1$ and $\Phi^{-1}(x)=x^{1/r}$ for $x\ge1$,
condition \eqref{eq:IC-zero} holds for small $\lambda$ (the integral being
proportional to $\lambda^{p-1}$) if and only if $q<q_c<r$. This requires
$p>p_\ast$. Under confinement, no such restriction is needed: by Theorem~\ref{thm:confined}, for every $p>1$ and every
$q>\max\{1,q_c\}$, small data in $L^q$ give global solutions decaying
exponentially, whereas for $\varrho=0$ and $1<p\le p_\ast$ all nonnegative
nontrivial solutions blow up (Remark~\ref{rem:nofujita}).
Figure~\ref{fig:power} illustrates these regimes.

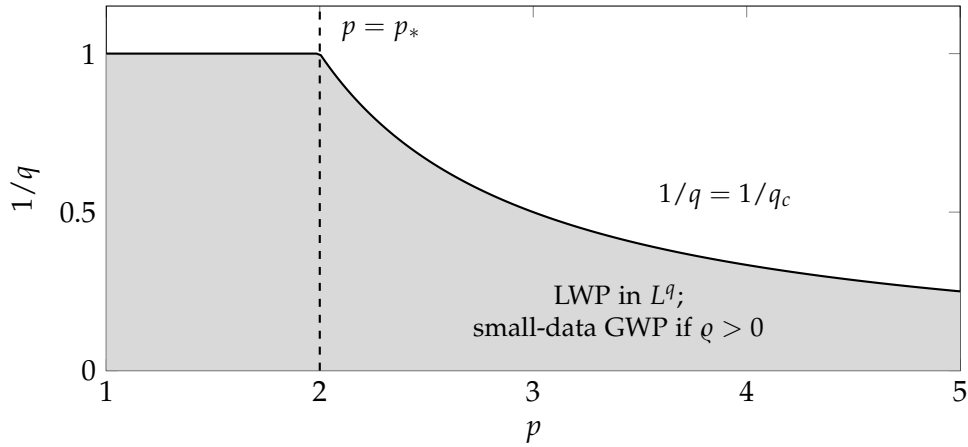
\begin{figure}[ht]
\centering
\begin{tikzpicture}
\begin{axis}[
  width=0.78\textwidth, height=6.4cm,
  xmin=1, xmax=5, ymin=0, ymax=1.15,
  xlabel={$p$}, ylabel={$1/q$},
  xtick={1,2,3,4,5}, ytick={0,0.5,1},
  samples=200, domain=1.001:5]
\addplot[name path=top, black, thick, domain=1:5] {min(1, 1/(x-1))};
\addplot[name path=bot, draw=none, domain=1:5] {0};
\addplot[black!15] fill between[of=top and bot];
\addplot[black, thick, dashed] coordinates {(2,0) (2,1.15)};
\node[font=\small, align=center] at (axis cs:3.4,0.18) {LWP in $L^q$;\\
small-data GWP if $\varrho>0$};
\node[font=\small, align=left, anchor=west] at (axis cs:2.05,1.07)
{$p=p_\ast$};
\node[font=\small] at (axis cs:3.9,0.55) {$1/q=1/q_c$};
\end{axis}
\end{tikzpicture}
\caption{The power nonlinearity $f(u)=|u|^{p-1}u$ for $d=1$, $\beta=\frac12$,
so that $p_\ast=2$ and $q_c=p-1$. In the shaded region
$1/q<\min\{1,1/q_c\}$, Theorem~\ref{thm:local} gives local well-posedness in
$L^q$, and Theorem~\ref{thm:confined} gives global solutions with exponential
decay for small data, for every $p>1$. For the free equation ($\varrho=0$) all
nonnegative nontrivial solutions blow up when $p\le p_\ast$ (to the left of
the dashed line) \cite{Sugitani1975}.}
\label{fig:power}
\end{figure}

\subsection{Logarithmically corrected power nonlinearity}

For $p\ge p_\ast$ and $m\ge0$, consider $f(u)=|u|^{p-1}u\log^{m}(1+|u|)$, and let
$\Phi(u)=u^{q_c}\log^{r}(1+u)$ with $q_c=d(p-1)/(2\beta)\ge1$ and $r\ge1$. This
$\Phi$ is an $N$-function. Indeed, $\Phi(u)/u=u^{q_c-1}\log^r(1+u)$ tends to $0$
at $0$ and to $\infty$ at $\infty$, and, with $L:=\log(1+u)$,
$$
\Phi''(u)=q_c(q_c-1)u^{q_c-2}L^r+\frac{2q_cru^{q_c-1}L^{r-1}}{1+u}
+\frac{r(r-1)u^{q_c}L^{r-2}}{(1+u)^2}-\frac{ru^{q_c}L^{r-1}}{(1+u)^2}>0 ,
$$
since the last (negative) term is smaller than half of the second one. For
$y\ge1$, we have $\ell(\lambda y)=f'(\lambda y)\le C\lambda^{p-1}y^{p-1}\log^m(1+y)$
for $\lambda\le1$, $\Phi'(y)\le Cy^{q_c-1}\log^r(1+y)$ and
$\Phi(y)^{-p_\ast}=y^{-q_cp_\ast}\log^{-rp_\ast}(1+y)$; since
$p-2-q_c(p_\ast-1)=-1$, the integrand in \eqref{eq:IC-y} is bounded by
$Cy^{-1}\log^{m-2\beta r/d}(1+y)$. This function is integrable at infinity if
$r>d(m+1)/(2\beta)$. Hence \eqref{eq:main} is well posed in $M_\Phi$ and, since
$\Phi$ satisfies the $\Delta_2$-condition, in $L_\Phi=L^{q_c}\log^rL$ itself;
compare \cite[Section 7.6]{Laister}. The case $m=0$, $p=p_\ast$ gives
well-posedness in $L\log^rL$ for $r\ge1$, $r>d/(2\beta)$, at the Fujita
exponent of the free fractional heat equation. By contrast, majorising $f$ by
$|u|^{p'-1}u$ with $p'>p$ and using Lebesgue space theory only gives
well-posedness in $L^{q}$ with $q\ge d(p'-1)/(2\beta)>q_c$. The space
$L^{q_c}\log^rL$, however, contains functions with local singularities that
are not locally $q$-integrable for any $q>q_c$. Under confinement, since
$\ell(s)\le Cs^{p-1}\log^m(1+s)$ satisfies \eqref{eq:dini},
Theorem~\ref{thm:confined} gives global solutions for small data in the same
space.

\section{Comments and open problems}\label{sec:open}

We close with some questions left open by the arguments above. They are stated
as problems, not as results.

\emph{The range of $\beta$.} Everything in this paper rests on the positivity
of the kernel $\Kb^\beta_t$, which fails in general for $\beta>1$; see Remark
\ref{rem:betarange}. For the free biharmonic heat equation, the local theory in
exponential Orlicz spaces was developed in \cite{majdoub2018ade} by exploiting
the explicit oscillatory kernel. Fixed-time estimates for
$e^{-t\Hr^\beta}$, $\beta>1$, are available in part of the range of exponents
\cite{BMMfrac}. Whether the smoothing estimate \eqref{eq:OrliczOrlicz}
survives for $\Hr^{\beta}$ with $\beta>1$, possibly with a constant larger than
$1$ in \eqref{E5.2}, and what replaces the comparison principle, is open.

\emph{Sharpness of the integral condition.} Theorem \ref{thm:critical} shows
that the order relation $\Phi\gtrsim f$ is sharp at the level of the critical
Young's function, but it does not show that \eqref{eq:IC-infty} itself is
necessary. {In fact, it is not necessary in general: the endpoint $q=q_c$
for power nonlinearities is reached by a different method in
Section~\ref{subsec:power}.} In the Lebesgue setting, a complete
characterisation of local existence is known \cite{AIHP2016}, and it would be
interesting to know what replaces it here.

\emph{Large data and the role of $\varrho$.} Theorem~\ref{thm:confined}
exploits the spectral gap only through the linear decay of $S_\beta(t)$. For
large data, one expects the confinement to compete with the nonlinearity. In
particular, one expects a threshold between global existence and finite-time
blow-up, depending on $\varrho$ and on the size of the projection of $u_0$
onto the ground state $h_0$.

\emph{Linear terms.} If $\ell(0^+)=a>0$, as for $f(u)=u\,e^{|u|^p}$ ($m=1$),
the Dini condition \eqref{eq:dini} fails. Since the linearised
equation $\partial_tv+\Hr^\beta v=av$ decays at the rate $\lambda_\beta-a$, one
expects global existence for small data whenever $a<\lambda_\beta$.

\emph{Dilation-critical singularities.} {In Remark~\ref{rem:DCS}, the
existence part of \cite{Fujishima2024} is transferred to $\Hr^\beta$ by
comparison with the free problem.} A direct proof, and the determination of the
critical value $\gamma^\ast$ of $\gamma$ for $\varrho>0$, remain open.

\emph{Uniqueness of mild solutions.} Theorem~\ref{thm:critical}(a)(ii)
requires a condition on $\Phi$ near the origin. Whether uniqueness holds in
$C([0,T),M_\Phi)$ for every $\Phi\in\Nf$ with $\Phi\gtrsim f$ is open, even for
$\beta=1$ and $\varrho=0$. Non-uniqueness results for weaker solution concepts
are known \cite{haraux1982non,IRT2019,FIRT2026}.

\emph{The case $\rho=0$.} The case $\rho=0$ in $(\mathrm G)$, which includes
both polynomially dominated nonlinearities and nonlinearities of
quasi-polynomial growth, is excluded here as it is in \cite{Laister}.

\subsection*{Acknowledgements}
M.~Majdoub
is supported by the Basic and Applied Scientific Research Center at Imam
Abdulrahman Bin Faisal University. 
R.~M. acknowledges the partial support provided by the research grants  (ANRF/ARGM/2025/000750/MTR). D.~G.~B. is grateful to an ANRF-MATRICS Grant ANRF/ARGM/2025/000580/MTR.



\end{document}